\documentclass[11pt]{article}
\usepackage[margin=1in]{geometry}
\usepackage{amsmath,amsthm,amssymb}
\usepackage{commath}
\usepackage{enumitem}
\usepackage{graphicx}
\usepackage{color}
\usepackage{subfigure}
\usepackage{hyperref}
\usepackage{bm}
\usepackage{appendix}
\usepackage{algorithm,algorithmic}

\newtheorem{theorem}{Theorem}
\newtheorem{corollary}[theorem]{Corollary}

\theoremstyle{definition}
\newtheorem{definition}{Definition}
\newtheorem{assumption}{Assumption}

\DeclareMathOperator{\ex}{\mathbb{E}}

\DeclareMathOperator{\diag}{diag}

\newcommand{\Gcal}{\mathcal{G}}
\newcommand{\Vcal}{\mathcal{V}}
\newcommand{\Ecal}{\mathcal{E}}
\newcommand{\Acal}{\mathcal{A}}

\newcommand{\PP}{\mathrm{PP}}
\newcommand{\Rbb}{\mathbb{R}}

\newcommand{\Xbm}{\bm{X}}
\newcommand{\Nbm}{\bm{N}}
\newcommand{\Rbm}{\bm{R}}

\newcommand{\Jbm}{\bm{J}}
\newcommand{\bbm}{\bm{b}}
\newcommand{\cbm}{\bm{c}}
\newcommand{\onebm}{\bm{1}}
\newcommand{\xbm}{\bm{x}}

\newcommand{\tbm}{\bm{t}}
\newcommand{\taubm}{\bm{\tau}}
\newcommand{\Xbmh}{\bm{X}^h}
\newcommand{\Xbmhl}{\bm{X}^{h,l}}
\newcommand{\Xbmhat}{\hat{\bm{X}}}

\newcommand{\dgT}{\delta_{g,T,\beta}}
\newcommand{\CgT}{C_{g,T}}
\newcommand{\DgT}{D_{g,T,\beta}}
\newcommand{\influgT}{\mu_{g,T}}

\newcommand{\sgT}{\sigma_{g,T}^2}

\newcommand{\influ}{\mu}
\newcommand{\influhat}{\hat{\mu}}
\newcommand{\xhat}{\hat{x}}

\title
{A Jump Stochastic Differential Equation Approach for Influence Prediction on Information Propagation Networks}
\author{
Yaohua Zang\thanks{Department of Mathematics, Zhejiang University, Hangzhou, Zhejiang, China 
  (\url{yhchuang@zju.edu.cn}).}
\and Gang Bao\thanks{Department of Mathematics, Zhejiang University, Hangzhou, Zhejiang, China
  (\url{baog@zju.edu.cn}).}
\and Xiaojing Ye\thanks{Department of Mathematics and Statistics, Georgia State University, Atlanta, GA, USA
  (\url{xye@gsu.edu}).}
\and Hongyuan Zha\thanks{School of Computational Science and Engineering, Georgia Institute of Technology, Atlanta, GA, USA
  (\url{zha@cc.gatech.edu}).}
\and Haomin Zhou\thanks{School of Mathematics, Georgia Institute of Technology, Atlanta, GA, USA
  (\url{hmzhou@math.gatech.edu}).}
}
\date{} 

\begin{document}

\maketitle

\begin{abstract}
We propose a novel problem formulation of continuous-time information propagation on heterogenous networks based on jump stochastic differential equations (SDE). The structure of the network and activation rates between nodes are naturally taken into account in the SDE system. This new formulation allows for efficient and stable algorithm for many challenging information propagation problems, including estimations of individual activation probability and influence level, by solving the SDE numerically. To this end, we develop an efficient numerical  algorithm incorporating variance reduction; furthermore, we provide theoretical bounds for its sample complexity. Moreover, we show that the proposed jump SDE approach can be applied to a much larger class of critical information propagation problems with more complicated settings. Numerical experiments on a variety of synthetic and real-world propagation networks show that the proposed method is more accurate and efficient compared with the state-of-the-art methods. 
\\

\noindent
\textbf{Keywords:} Propagation network, heterogenous network, influence prediction, complexity
\end{abstract}




\section{Introduction}
\label{sec:intro}

Information propagation on networks is a prevalent phenomenon in real-world \cite{Boguna:2002a,Newman:2010a,Pastor-Satorras:2015a}
Examples of information propagation include news spreading on social media \cite{Du:2013a,Farajtabar:2016b,Vergeer:2013a}, viral marketing \cite{Kempe:2003a,Kempe:2005a,Wortman:2008a}, computer malware spread, and epidemic of contagious diseases \cite{Bodo:2016a,Miller:2014a,Pastor-Satorras:2015a,Sahneh:2011a}.
More specifically, for instance, a piece of information (such as news) can be retweeted by users (nodes) on the Twitter social network formed by their followee-follower relationships.
In this case, we call a node activated, become active, or infected, if the user participates to tweet, and the followers of this user get activated if they retweet his/her tweet later.
By this, the activation process gradually progresses and the tweet spreads out.
Such information propagation behaves very similarly as epidemic spread where a type of virus can infect an individual (human, animal, or plant) and spreads to others upon their close contact.

To begin with, we describe the basic information propagation model on a network.
Let $\Gcal=(\Vcal,\Ecal)$ be a given network, i.e., directed graph, where $\Vcal=\{1,\dots,n\}$ is the node set and $\Ecal\subset \Vcal\times \Vcal$ is the edge set.
We denote $n=|\Vcal|$ and $m=|\Ecal|$.
In addition, we denote $\Acal:=\{\alpha_{ij}>0: (i,j)\in \Ecal\}$, where $\alpha_{ij}>0$ is the activation/infection rate of $i$ on $j$.
More precisely, once the node $i$ becomes active at time $t_i$, the time that $i$ needs to infect its healthy neighbor $j$, denoted by $t_{ij}:=t_j-t_i>0$, follows the $\mathrm{Exp}(\alpha_{ij})$ distribution.
Here by $t\sim \mathrm{Exp}(\alpha)$, we mean that the probability density of $t$ is $f(t)=\alpha e^{-\alpha t}$ for $t>0$.
Now suppose $i$ (called the source node) is active at time $0$, then it will start to activate its neighbors in $\Vcal_i^-:=\{j: (i,j)\in \Ecal\}$ simultaneously and independently, each at the corresponding rate given in $\{\alpha_{ij}\in\Acal: j\in\Vcal_i^-\}$.
An activated neighbor node $j$ will then start to activate its healthy neighbors in $\Vcal_j^-$ at the corresponding rates, and so on.
Thus the information originated from the node $i$ can propagate to other nodes on the network.
In a slight generalization to include recovery scenario, an active/infected node $i$ may also recover at some rate $\gamma_i>0$, become inactive/healthy and prune to activation/infection again.

Given the aforementioned propagation model, we are interested in the fundamental problem of influence prediction in this paper.
More specifically, we want to compute the influence, defined as the expected number of active nodes on the network, at any time $t>0$ given that the propagation started from a known set of source nodes.
Moreover, we are also interested in the probability that a specific node $i$ is active/infected at time $t$.
However, these seemingly simple problems turn out to be very challenging computationally:
the heterogenous structure of the given network $\Gcal$ and the variations in activation rates in $\Acal$ (and recovery rates $\mathcal{R}:=\{\gamma_i>0:i\in \Vcal\}$ if applicable) significantly complicate the computation.
For example, an analytic solution of influence prediction is shown to require computation of a Markov chain whose state space is of size $O(2^n)$ \cite{Gomez-Rodriguez:2012c,Van-Mieghem:2009a}, which quickly becomes computationally intractable as $n$ increases.

The influence prediction problem can be significantly more complicated even with slight modification to the basic propagation model above.
For example, if the activation time $t_{ij}$ follows a non-exponential distribution, such as Weibull, Rayleigh, or power-law, the propagation is no longer Markov and analytic solution is not available.
Another example is that the activation processes are not independent: the rate of $j$ being infected is not simply the sum of $\alpha_{ij}$'s over its infected parents in $\Vcal_j^+:=\{i:(i,j)\in\Ecal\}$, but some nonlinear function of them.
In this case, even direct simulation of the propagation becomes computationally prohibitive.

In this paper, we propose a novel approach to address the aforementioned computational problem of influence prediction.
Our approach is based on a reformulation of the information propagation on heterogenous network into a system of jump stochastic differential equations (SDEs).
We also develop an efficient numerical algorithm based on this approach to solve the influence prediction problem.
Our method is instantiated using the basic propagation model with exponentially distributed activation times for demonstration purpose, whereas enhancement to further extend our approach  is also provided.
More importantly, we show that our approach can be applied to a variety of  other critical information propagation problems where some or all existing methods may fail to work.

The main contribution of this paper lies in three phases:
(a) We propose a novel mathematical formulation of information propagation on deterministic heterogenous networks using jump SDEs, where the network structure and activation rates are naturally incorporated as system parameters;
(b) We develop a fast and robust numerical algorithm based on the SDE formulation, and provide a comprehensive analysis of the sampling complexity and prediction error; and
(c) We perform extensive numerical tests using a variety of network structures to demonstrate the efficiency and robustness of the proposed method.

The rest of this paper is organized as follows.
First we review the literature on information propagation on networks that are related to this work
in Section \ref{sec:related}.
In Section \ref{sec:proposed}, we provide the details of the proposed SDE reformulation.
Then we develop a numerical algorithm based on the new formulation to solve influence prediction problem in Section \ref{sec:numerical}.
We present applications of the proposed method to more general propagations in Section \ref{sec:extension}.
We demonstrate the performance of the proposed algorithm, with comparison to several state-of-the-arts methods, on a variety of networks in Section \ref{sec:results}.
Section \ref{sec:conclusion} concludes the paper.


\section{Related Work}
\label{sec:related}

The basic information propagation model with constant activation (infection) rates is equivalent to the classical susceptible-infected (SI) model (or a variation, called susceptible-infected-susceptible, or SIS, model where an infected node can recover and become prune to infection again) which has been extensively studied for \textit{statistically homogenous} networks where individual nodes are indistinguishable in the past decades.
A recent comprehensive survey can be found in \cite{Pastor-Satorras:2015a}.
In contrast, existing work on SI/SIS for fixed and heterogenous networks considered in this paper is very limited,
mainly due to the significant complications and computational challenges due to the prescribed diversities of nodes and their interactions.
In this case, a solution to exactly describe the process requires a state space of
size $O(2^m)$, where $m$ is the size of the network, and hence is computationally intractable in practice
\cite{Gomez-Rodriguez:2012c,Van-Mieghem:2009a}.

One of the major approaches to \textit{approximately} quantify influence or infection probability for the basic propagation model with constant infection/recovery rates is based on mean field theory.
In \cite{Van-Mieghem:2013a,Van-Mieghem:2009a}, mean field approximation is applied to reduce the $O(2^m)$ linear
system describing the Markov SIS process to an $n$-intertwined model, which
is a system of $n$ nonlinear ordinary differential equations.
This method adopts a first-order moment closure that ignores dependencies between infection states of neighbor nodes, and hence its solution gives an upper bound of infection probability for each node. Mean-field approximation is then also applied to more complicated cases including multilayer network \cite{Sahneh:2013a}, weighted network \cite{Yang:2012b}, and hypergraph \cite{Bodo:2016a}.
The mean-field approach is also applied to the competing bi-virus model in \cite{Liu:2016b}.
In \cite{Sahneh:2011a}, an additional ``alert'' state of nodes is considered, such that individuals with infected neighbors
may enter an alert state and become less prune to infection.
In \cite{Cator:2012a}, a second order mean-field approximation is employed
which improves the estimate of epidemic threshold over first order mean-field at the cost of
significantly more computation complexity.
In \cite{Van-Mieghem:2015a}, the authors propose an accuracy criterion of mean-field approximation
using the the covariance between infection states of adjacent nodes.
However, it is computationally intractable to estimate this error due to its exponentially large size.
In \cite{Chow:2018a}, a discrete Fokker-Planck equation based on aggregated activation states is proposed which yields fast and accurate computation of influence without the presence fo recovery. 
Discussions on limitations of mean-field are also provided in \cite{Givan:2011a}.
In \cite{Van-Mieghem:2013b}, the authors showed that the basic propagation model with exponentially-distrubted infection time can be unrealistic in real-world applications and the Markov property that mean-field approaches are based on may not hold.

In recent years, there is a significant increase of interests in information propagation in network and data science due to its prevalent applications in social networking and cyber security. Most literature in this field focus on \textit{discrete-time} information propagations where infections can occur only at discrete time points.
In contrast, continuous-time information propagation studied in the literature closely mimics the SI/SIS model but is also much more challenging computationally as mentioned above. 
In \cite{Gomez-Rodriguez:2012c}, an analytic solution is derived based on the observation that infection time is the length of the stochastic shortest path from the source nodes to the node.
This method establish Markov chains for each node and can estimate individual infection probability, however, the complexity still grows at order $O(2^n)$ for general networks and hence is not scalable.
In \cite{Du:2014a,Du:2013a}, the authors propose a novel sampling technique to estimate the coverage function in information propagation, and developed efficient algorithms to approximate influence even for large networks.
In \cite{Scaman:2015a}, the authors derived bounds of influence and characterized phase transition using the spectral radii of the Laplace Hazard matrix.
Learnability of influence function, which is the core of influence prediction in these methods, is considered in \cite{Narasimhan:2015a}.

\section{SDE Formulation of Information Propagation}
\label{sec:proposed}
In this section, we propose a novel formulation the information propagation on network using jump stochastic differential equations.
This new formulation has an exact and concise mathematical interpretation of the complex random propagation process.
Moreover, we provide an efficient numerical algorithm to solve the influence prediction problem based on this formulation.

For ease of presentation, we first focus on the basic propagation model where the activation times follow independent exponential distributions as described in Section \ref{sec:intro}.
That is, the time $t_{ij}$ for a just-activated node $i$ to activate/infect its inactive/healthy neighbor $j$ follows the $\text{Exp}(\alpha_{ij})$ distribution.
In the presence of recovery, a just-activated node $i$ can also recover and become inactive at rate $\gamma_i$, i.e., the time needed for recovery follows the $\mathrm{Exp}(\gamma_i)$ distribution.
We assume all these activation/infection and recovery times are independent.

To represent the propagation process as a system of jump SDEs, we first denote the stochastic process $X_{i}(t)$ as the time-evolving activation state of the node $i$ at time $t$.
Namely, $X_{i}(t)=1$ if node $i$ is active/infected at time $t$, and $X_{i}(t)=0$ otherwise.
Therefore, each $X_i(t)$ is a right continuous function of time $t$ with left limit.
We denote $X_i(t^-):=\lim_{\tau\to t^-} X_i(\tau)$ the left limit of $X$ at $t$.
For each edge $(i,j)\in \Ecal$ associated with activation rate $\alpha_{ij}$, we introduce an auxiliary temporal point process $N_{ij}(t)\sim \PP(\alpha_{ij})$.
Here by $N(t)\sim \PP(\alpha)$ we meant that $N(t)$ is the time homogeneous Poisson process with intensity $\alpha$, namely, $\lim_{\delta\to0^+}(1/\delta)\cdot\ex[N(t+\delta)-N(t)]=\alpha$ for all $t$.
In other words, $N(t)$ can be thought of as a counting process whose value is $0$ at time $t=0$ and increases (jumps) by 1 at each time $\tau_1,\tau_2,\dots$, where $\{\tau_{k+1}-\tau_k:k\geq0\}$ are i.i.d.\ $\mathrm{Exp}(\alpha)$ random variables ($\tau_0=0$ by convention).
Hence $N(t)=k$ if $\tau_k\le t < \tau_{k+1}$.
In the presence of recovery, we also introduce $R_i(t)\sim \PP(\gamma_i)$ for each $i\in \Vcal$.
Note that $\{N_{ij}(t),R_i(t):i\in\Vcal, (i,j)\in\Ecal\}$ is a finite set of Poisson processes, and hence two or more of them jumping at the same time has probability $0$.

Now the key observation is that we can think $N_{ij}(t)$ as of $i$ ``sending an activation signal'' to $j$ at times $\tau_1,\tau_2,\dots$.
Therefore, $\dif N_{ij}(t)=1$ if $t=\tau_k$ for some $k$ or otherwise $\dif N_{ij}(t)=0$.
We note that a node $j$ becomes activated by $i$ at $t$ successfully if and only if $X_i(t^-)=1$, $X_j(t^-)=0$, and $\dif N_{ij}(t)=1$. By this we have $\dif X_i(t)=1$, and $X_i(t)$ jumps from $0$ to $1$.
Considering that there may be multiple parent nodes in $\Vcal_j^+$ sending activation signals to $j$ simultaneously, we can write the activation process of $j$ as
\begin{equation}\label{eq:jsde_no_recover}
\dif X_j(t) = (1-X_j(t^-))\sum\nolimits_{i\in \Vcal_j^+} X_i(t^-) \dif N_{ij}(t)
\end{equation}
for every $j=1,\dots,n$.
In the presence of recovery, we know an active node $j$ becomes deactivated/recovered if and only if $X_j(t^-)=1$ and $\dif R_j(t)=1$.
Therefore, we can add this recovery term to \ref{eq:jsde_no_recover} and obtain
\begin{equation}\label{eq:jsde_recover}
\dif X_j(t)= \sbr[2]{ (1-X_j(t^-))\sum\nolimits_{i\in \Vcal_j^+}X_i(t^-)\dif N_{ij}(t) } -X_j(t^-)\dif R_j(t)
\end{equation}
for $j=1,\dots,n$.
Compared to \eqref{eq:jsde_no_recover}, the additional term $-X_j(t^-)\dif R_j(t)$ in \eqref{eq:jsde_recover} indicates that $X_j(t)$ can jump from $1$ to $0$ ($\dif X_j(t)=-1$) if $j$ recovers at time $t$.

To rewrite \eqref{eq:jsde_recover} in concise matrix form, we first introduce the following vectors
(unless otherwise noted, all vectors are column vectors):
\begin{align}
\Xbm(t) & = (X_1(t),\dots,X_n(t))^\top \in \Rbb^n \label{eq:X}\\
\Nbm_{\cdot j}(t) & = (N_{k_1j}(t),N_{k_2j}(t),\dots,N_{k_{|\Vcal_j^+|}j}(t) )^\top \in \Rbb^{|\Vcal_j^+|} \label{eq:Nj}\\
\Nbm(t) & = (\Nbm_{\cdot 1}(t)^\top , \dots, \Nbm_{\cdot n}(t)^\top )^\top \in \Rbb^m \label{eq:N}\\
\Rbm(t) & = (R_1(t),\dots,R_n(t))^\top \in \Rbb^n \label{eq:R}\\
\Jbm(t) & = (\Rbm(t)^\top , \Nbm(t)^\top )^\top \in \Rbb^{n+m} \label{eq:Z}
\end{align}
where $k_1<k_2<\dots$ are the parent nodes of $j$ in $\Vcal_j^+$.
Furthermore, $\forall\,\Xbm=(X_1,\dots,X_n)^\top \in \Rbb^n$, we define matrix functions $\cbm_0(\Xbm),\cbm_1(\Xbm),\cbm(\Xbm)$ as
\begin{align}
\cbm_0(\Xbm) & = \diag(X_1,\dots,X_n) \in \Rbb^{n\times n} \label{eq:c0}\\
\cbm_1(\Xbm) & = \diag(\bbm_1(\Xbm)^\top ,\dots,\bbm_n(\Xbm)^{\top}) \in \Rbb^{n\times m} \label{eq:c1}\\
\cbm(\Xbm)\, & = [-\cbm_0(\Xbm), \cbm_1(\Xbm)] \in \Rbb^{n\times (n+m)} \label{eq:c}
\end{align}
where $\bbm_j(\Xbm)=((1-X_j)X_{k_1},\dots,(1-X_j)X_{k_{|\Vcal_j^+|}})^\top \in\Rbb^{|\Vcal_j^+|}$.
Note that $\cbm_0(\Xbm)$ in \eqref{eq:c0} is a diagonal matrix, and $\cbm_1(\Xbm)$ in \eqref{eq:c1} is a block-diagonal matrix with row vectors $\bbm_j(\Xbm)^{\top}$ as the $(j,j)$-block.

Using the vector and matrix notations above, we can rewrite \eqref{eq:jsde_recover} concisely as follows:
\begin{equation}\label{eq:JSDEmtx}
\dif \Xbm(t) = \cbm(\Xbm(t^-)) \dif \Jbm(t) \ .
\end{equation}
The initial $\Xbm(0)=(X_1(0),\dots,X_n(0))^\top $ is determined such that $X_i(0)=1$ if $i$ is a source node and $0$ otherwise.
Note that, without recovery, the system \eqref{eq:JSDEmtx} reduces to
\begin{equation}\label{eq:JSDEmtx_no_recover}
\dif \Xbm(t) = \cbm_1(\Xbm(t^-)) \dif \Nbm(t),
\end{equation}
which is equivalent to \eqref{eq:jsde_no_recover}.

We remark that the system of $n$ coupled jump SDEs, \eqref{eq:JSDEmtx}, or equivalently \eqref{eq:jsde_recover}, represents the basic propagation model exactly.
Since the stochastic process $X_i(t)$ is binary-valued, we know that the probability of a node $i$ being active at time $t$ is $x_i(t):=\ex[X_i(t)$].
Moreover, the influence, defined by the expected number of active nodes at time $t$, is $\mu(t):=\ex[\onebm^\top \Xbm(t)]=\ex[\sum_iX_i(t)]=\sum_i x_i(t)$.
Therefore, the system \eqref{eq:JSDEmtx} and its solution $\Xbm(t)$ play the central role of our algorithmic development for influence prediction below.

\section{Algorithm and Complexity}
\label{sec:numerical}
As showed above, the basic information propagation process can be formulated as the system of jump SDEs \eqref{eq:jsde_recover}.
To obtain individual activation probability $x_i(t)$ and influence $\mu(t)$, we need to solve the SDE system \eqref{eq:jsde_recover} numerically to estimate $\ex[X_i(t)]$ and $\ex[\onebm^\top \Xbm(t)]$, respectively.
In the literature, numerical approximations to the solution of an SDE can be generally categorized in two types: strong approximation and weak approximation.
Strong approximation is used to estimate the solution $\Xbm(t)$ pathwisely.
Weak approximation, on the other hand, is used to compute the expectations of (functions of) $\Xbm(t)$, such as $\ex[X_i(t)]$ and $\ex[\onebm^\top \Xbm(t)]$.
Therefore, we use weak approximations in this paper as it suffices for our influence prediction problems.

In what follows, we first introduce the (weak) Euler and Taylor schemes with standard time discretization for solving \eqref{eq:JSDEmtx} in Section \ref{subsec:euler}.
Then we employ a variance reduction technique and present our algorithm in Section \ref{subsec:alg}.
In Section \ref{subsec:complexity}, we establish the relation between approximation error and sampling complexity of the proposed algorithm.

\subsection{Euler and Taylor schemes}
\label{subsec:euler}
We consider the solution of the system \eqref{eq:JSDEmtx} over time interval $[0,T]$ for some prescribed time horizon $T$.
For ease of presentation, we partition $[0,T]$ into $K(h):=T/h$ equal segments using discretization points $t_k=kh$ for $k=0,1,\dots,K(h)$, where without loss of generality we assume the step size $h\in(0,1)$.
Now we want to compute $\Xbmh(t)$ which approximate $\Xbm(t)$ at all $t_k$ for $k=0,1,\dots,K(h)$.
For notation simplicity, we often denote $\Xbmh_k:=\Xbmh(t_k)$ in the derivation below.

\begin{definition}[Weak order]
An approximation $\Xbmh$ is said to converge to $\Xbm$ at weak order $\beta>0$ if for any $g\in C_P^{2\beta+2}(\mathbb{R}^{n+m};\mathbb{R})$, there exists a constant $C_{g,T}>0$ independent of $h$ such that
\begin{equation}\label{eq:weakorder}
|\ex[g(\Xbmh(T))]-\ex[g(\Xbm(T))]| \le C_{g,T} h^{\beta}.
\end{equation}
for all $h>0$ sufficiently small. 
Here $C_P^{\beta}(\mathbb{R}^{n};\mathbb{R})$ denotes the space of $\beta$ times continuously differentiable functions $\Rbb^n \to \Rbb$ with partial derivatives with polynomial growth, 
\end{definition}
Note that the weak order $\beta$ of convergence for a discretization scheme has important practical consequences for simulation efficiency.

To generate $\Xbmh$ that approximates the solution $\Xbm$ in \eqref{eq:JSDEmtx}, a basic algorithm is the Euler scheme:
\begin{equation}\label{eq:euler}
\Xbmh_{k+1} = \Xbmh_k +  \cbm(t_k, \Xbmh_k) \Delta \Jbm_k,
\end{equation}
where we also included $t$ in function $\cbm$ so that the method can generalize to cases where $\cbm$ is time-dependent. 
In \eqref{eq:euler}, the increment $\Delta \Jbm_k=\Jbm(t_{k+1})-\Jbm(t_{k})$ is a vector of independent Poisson
random numbers.
More precisely, $\Delta \Jbm_k=(\Delta \Rbm_k^\top , \Delta \Nbm_k^\top)^\top\in\Rbb^{n+m}$ where the components of $\Delta\Rbm_k$ and $\Delta\Nbm_k$ are generated from Poisson distributions with $\{\gamma_ih\}$ and $\{\alpha_{ij}h\}$ as parameters, respectively.
The Euler scheme \eqref{eq:euler} is known to be convergent at weak order $\beta=1$ \cite{Platen:2010a}.

By adding more terms in the Wagner-Platen expansion \cite{Platen:2010a}, we can obtain numerical schemes with a higher order of convergence.
For example, the Taylor scheme with convergence of weak order $\beta=2$ is given by
\begin{equation}\label{eq:taylor1}
\left\{
\begin{split}
\Xbmh_{k+.5}&= \Xbmh_k +  \cbm_k \Delta \Jbm_k \\
\Xbmh_{k+1}&= \Xbmh_k + \cbm_k \Delta \Jbm_k + (1/2) \cdot (\cbm_{k+.5}-\cbm_k) \Jbm_k (\Jbm_k - 1)
\end{split}
\right.
\end{equation}
where $\cbm_k:=\cbm(t_k, \Xbmh_k)$, $\cbm_{k+.5} := \cbm(t_k, \Xbmh_{k+.5})$, and the subtraction and multiplication in $\Jbm_k (\Jbm_k - 1)$ are executed componentwisely to obtain a vector in $\Rbb^{m+n}$.

Higher-order Taylor schemes can also be obtained by adding more Wagner-Platen terms to \eqref{eq:taylor1}.
We refer interested readers to \cite{Platen:2010a}.
In this paper, we only use the Euler scheme \eqref{eq:euler} with $\beta=1$, and sometimes the Taylor scheme \eqref{eq:taylor1} with $\beta=2$, since they are accurate and cost-effective for our influence prediction problem.

It is also worth noting that, for mark-independent pure jump SDEs such as \eqref{eq:JSDEmtx}, jump-adapted time discretization can directly simulate jump times for non-uniform time discretization.
More precisely, in jump-adapted discretization, we first simulate a trajectory $0=\tau_0<\tau_1<\dots\leq T$ of the Poisson process $\Jbm(t)$, and apply it to the SDE \eqref{eq:JSDEmtx} so that the next sample $\Xbmhat(\tau_{k+1})$ can be computed directly given $\Xbmhat(\tau_k)$.
For $t\in[\tau_k, \tau_{k+1})$ the approximation $\Xbmhat(t)$ remains as constant $\Xbmhat(\tau_k)$.
One advantage of such jump-adapted Euler scheme produces is that it does not generate discretization error, and it is particularly efficient for low intensity jump processes. In this work, we will incorporate the jump-adapted discretization into the regular uniform discretization of the Euler and Taylor schemes in our algorithm.

\subsection{Variance reduction}
\label{subsec:alg}
The numerical solution $\Xbmh$ obtained by the Euler \eqref{eq:euler} or Taylor \eqref{eq:taylor1} schemes is one sample approximation of the stochastic process $\Xbm$ defined in \eqref{eq:JSDEmtx}.
To estimate $\ex[g(\Xbmh(T))]$ in the applications of influence prediction and individual activation probability, we can employ \eqref{eq:euler} or \eqref{eq:taylor1} for $L$ times to obtain $\{\Xbmhl:1\le l \le L\}$, and use their sample mean to approximate $\ex[g(\Xbmh(T))]$.
More specifically, for the prescribed time horizon $T$, the sample mean $u_T(h,L)$ is defined by
\begin{equation}\label{eq:avgest}
u_T(h,L)=(1/L)\cdot \sum\nolimits_{l=1}^{L} g(\Xbmhl(T))
\end{equation}
which depends on both of the discretization step size $h$ and the number of sample trajectories $L$.
This sample mean $u_T(h,L)$ is our approximation to $\ex[g(\Xbmh(T)]$.

As we will show later, the mean square error of the approximation $u_T(h,L)$ depends on the its variance $(1/L)\cdot \mathrm{var}(g(\Xbmhl(T)))$.
To reduce this approximation error in practice, we employ a variance reduction technique introduced in \cite{Maginnis:2014a} in our sampling algorithm.
More specifically, in each sampled propagation (also called sampling trajectory), we generate a pair of antithetical samples $Z^+$ and $Z^-$ from each Poisson random variable of mean $\gamma_iT$ and $\alpha_{ij}T$ for $i\in \Vcal$ and $(i,j)\in \Ecal$, respectively. 
Then we sample $Z^+$ and $Z^-$ points independently and uniformly on $[0,T]$ respectively.
Thus each set of points forms a trajectory of the Poisson process, which can be used in the Euler scheme \eqref{eq:euler} or the Taylor scheme \eqref{eq:taylor1} to obtain a sample of solution $\Xbmh_K=\Xbmh(t_K)=\Xbmh(T)$ to \eqref{eq:JSDEmtx}.
The antithetical property of $Z^\pm$ reduces the variance of $\Xbmh(t)$ in practice.
The algorithm is summarized in Algorithm \ref{alg:jsde}.
For ease of presentation, we assume that the sample size $L$ is even, and the step size $h$ is chosen such that $T/h$ is an integer.


\begin{algorithm}
\caption{Sample approximation $u_{T}(h,L)$ of influence $\influgT=\ex[g(\Xbm(T))]$}
\label{alg:jsde}
\begin{algorithmic}
\STATE \textbf{Input:} $\Gcal=(\Vcal,\Ecal)$, $\{\alpha_{ij},\gamma_i:i\in\Vcal, (i,j)\in\Ecal\}$, and $T$. Set $h$, $L$.
\STATE Set $K=L/h$ and $\tbm_k:= [kh,(k+1)h)$ for $k = 0, \dots, K-1$. Set $s=0$.
\FOR{$l=1,2,\dots,L/2$}
\FOR{each $i\in \Vcal$}
\STATE {Set $F$ to the cumulative distribution of $\text{Poisson}(\gamma_iT)$;}
\STATE {Draw $U\sim\text{Uniform}(0,1)$, and set $Z^{+}=F^{-1}(1-U)$ and $Z^{-}=F^{-1}(U)$;}
\STATE {Sample $Z^{+}$ and $Z^{-}$ points, $\taubm_{\gamma_i}^l:=\{\tau_z^l:1\le z\le Z^+\}$ and $\taubm_{\gamma_i}^{(L/2)+l}:=\{\tau_z^{L/2+l}:1\le z \le Z^-\}$ respectively, each on $[0,T]$ uniformly;}
\STATE {Set $(\Delta\Rbm_k^\ell)_i=|\taubm_{\gamma_i}^\ell\cap \tbm_k|$ for $k=0,\dots, K-1$ and $\ell=l,(L/2)+l$;}
\ENDFOR
\FOR{each $(i,j)\in \Ecal$}
\STATE {Set $F$ to the cumulative distribution of $\text{Poisson}(\alpha_{ij}T)$;}
\STATE {Draw $U\sim\text{Uniform}(0,1)$, and set $Z^{+}=F^{-1}(1-U)$ and $Z^{-}=F^{-1}(U)$;}
\STATE {Sample $Z^{+}$ and $Z^{-}$ points, $\taubm_{\alpha_{ij}}^l:=\{\tau_z^l:1\le z\le Z^+\}$ and $\taubm_{\alpha_{ij}}^{(L/2)+l}:=\{\tau_z^{L/2+l}:1\le z \le Z^-\}$ respectively, each on $[0,T]$ uniformly;}
\STATE {Set $(\Delta\Nbm_k^\ell)_{ij}=|\taubm_{\alpha_{ij}}^\ell\cap \tbm_k|$ for $k=0,\dots, K-1$ and $\ell=l,(L/2)+l$;}
\ENDFOR
\STATE Solve \eqref{eq:euler} or \eqref{eq:taylor1} for $\Xbm_{K}^{h,\ell}$ using $\{\Delta\Rbm_k^\ell,\Delta\Nbm_k^\ell:k\}$ for $\ell=l,(L/2)+l$;
\STATE $s \leftarrow s + g(\Xbm_K^{h,l}) + g(\Xbm_K^{h,(L/2)+l}) $;
\ENDFOR
\STATE \textbf{Output:} $u_{T}(h,L)=s/L$.
\end{algorithmic}
\end{algorithm}

\subsection{Sample complexity analysis}
\label{subsec:complexity}
To establish the relation of approximation accuracy and cost, we provide a comprehensive analysis of computation and sampling complexity for our SDE-based influence prediction method.
In particular, for any fixed time horizon $T>0$ and influence evaluation function $g$, we will derive upper bound of the root mean squared error (RMSE), denoted by $e_T(h,L)$, of the sample approximation $u_T(h,L)$ to $\influgT:=\ex[g(\Xbm(T))]$. 
%
The RMSE $e_T(h,T)$ is defined by
\begin{align}\label{eq:rmse}
e_T(h,L) = \cbr[2]{\ex \sbr[2]{|u_T(h,L)-\influgT|^2} }^{1/2}.
\end{align}
Without loss of generalization, we again assume that $h\in(0,1)$ and $T/h$ is an integer.
In addition, our complexity analysis requires the following conditions.
\begin{assumption}\label{as:bound}
The stochastic process $\Xbm(t)$ and the influence evaluation function $g$ satisfy:
\begin{enumerate}
\item The function $g$ satisfies polynomial growth. Namely, $\exists$ $C>0$ and positive integer $s>0$ such that $g(\xbm)\le C(1+\|\xbm\|^s)$ for all $\xbm\in\Rbb^n$.
\item The influence $g(\Xbm(t))$ has bounded second moment, i.e., $\ex[g^2(\Xbm(T))] < \infty$.
\end{enumerate}
\end{assumption}
%
%

Now we are ready to present the first result that links the RMSE to the step size $h$ and the number of sampling trajectories $L$ in the approximation $u_T(h,L)$.
\begin{theorem}\label{thm:e_bound}
Let $u_T(h,L)$ be the sample approximation to $\influgT$ generated by Algorithm \ref{alg:jsde} with a numerical SDE scheme \eqref{eq:euler} or \eqref{eq:taylor1} of weak order $\beta>0$.
Suppose Assumption \ref{as:bound} holds for $g$ and $\Xbm(t)$.
Then there exists a constant $\sgT$ dependent on $g$ and $T$ but not on $h$ and $\beta$, such that
\begin{equation}\label{eq:e_bound}
e_T(h,L) \le \del[2]{ \frac{\sgT}{L} + \CgT^2 h^{2\beta} }^{1/2}.
\end{equation}
\end{theorem}

\begin{proof}
As $g$ has polynomial growth, so does $g^2$.
Therefore we know there exists $C>0$, independent of $\beta$, such that $\mathrm{var}(g(\Xbmh(T)))$, the variance of $g(\Xbmh(T))$, has the following bound:
\begin{equation}\label{eq:var_bound}
\mathrm{var}(g(\Xbmh(T))) \le \ex[g^2(\Xbmh(T))] \le \ex[g^2(\Xbm(T)]+ Ch^\beta \le \ex[g^2(\Xbm(T))] + C =:\sgT,
\end{equation}
where the second inequality is due to $|\ex[g^2(\Xbmh(T))] - \ex[g^2(\Xbm(T))]| \le C h^{\beta}$ for some $C$ independent of $h$ by Theorem 12.3.4 in \cite{Platen:2010a}, and the last inequality is due to $h\in (0,1)$ and $\beta>0$.
Furthermore, the RMSE $e_T(h,L)$ of $u_T(h,L)$ defined in \eqref{eq:rmse} satisfies
\begin{align}
e_T^2(h,L)
& =\ \ex[|u_T(h,L)- \influgT|^2] \nonumber\\
& =\ \ex[|u_T(h,L)-\ex[g(\Xbmh(T))]+\ex[g(\Xbmh(T))]-\influgT|^2] \nonumber\\
& =\ \ex[|u_T(h,L)-\ex[g(\Xbmh(T))]|^2] + |\ex[g(\Xbmh(T))]-\influgT|^2 \nonumber \\
& =\ \ex[|u_T(h,L)-\ex[g(\Xbmh(T))]|^2] + C_{g,T}^2 h^{2\beta} \nonumber \\
& =\ (1/L)\cdot \mathrm{var}(g(\Xbmh(T))) + C_{g,T}^2 h^{2\beta} \nonumber \\
& \le\ (1/L)\cdot \sgT + C_{g,T}^2 h^{2\beta}, \nonumber
\end{align}
where we used the fact $\ex[u_{h,L}(T)]=\ex[g(\Xbmh(T))]$ in the third equality, \eqref{eq:weakorder} in the fourth equality, and the fact that every $\Xbmhl$ has the same distribution as $\Xbmh$ in the fifth equality, and \eqref{eq:var_bound} in the last inequality.
Taking square root on both sides yields \eqref{eq:e_bound}.
\end{proof}

\begin{corollary}
Suppose the conditions in Theorem \ref{thm:e_bound} hold.
For any $\epsilon>0$, the RMSE satisfies $e_T(h,L)\le \epsilon$ if the step size $h$ and the number of trajectories $L$ are set to
\begin{equation}\label{eq:opt_hL}
h=T \dgT \DgT^{-1/\beta} \epsilon^{1/\beta}\quad \mbox{and} \quad L=\dgT \DgT^2 \epsilon^{-2},
\end{equation}
where the constants $\dgT$ and $\DgT$ only depend on $g,T,\beta$ as follows,
\begin{equation}\label{eq:def_d}
\dgT = \del[2]{\frac{\sgT}{2\beta \CgT T^{2\beta}}}^{1/(2\beta)}\quad \mbox{and} \quad \DgT = \del[2]{\frac{\sgT}{\dgT} + \CgT T^{2\beta}\dgT^{2\beta}}^{1/2}.
\end{equation}
In particular, there is $e_T(h,L)\le \epsilon$ if the total sampling complexity is
\begin{equation}\label{eq:samp_cmp}
\del[2]{\sum\nolimits_{i,j}\alpha_{ij} + \sum\nolimits_{i}\gamma_i}T\dgT\DgT^2\epsilon^{-2}=O(\epsilon^{-2})
\end{equation}
on expectation and the computation complexity is 
\begin{equation}\label{eq:comp_cmp}
O((m+n)(\epsilon/\DgT)^{-(2\beta+1)/\beta})=O(\epsilon^{-(2\beta+1)/\beta}).
\end{equation}
\end{corollary}

\begin{proof}
As the total computation complexity is linear in $T/h$ and $L$, we denote the cost $B=LT/h$.
For any $B$, the minimum of the bound of $e_T(h,L)$ in \eqref{eq:e_bound} can be obtained by solving $\min_{h,L} \{(\sgT/L) + C_{g,T}^2 h^{2\beta} \}$ subject to $B=LT/h$, which yields solution
\begin{equation}\label{eq:hL}
h = T \dgT B^{-1/(2\beta+1)}\quad \mbox{and}\quad L = \dgT B^{2\beta/(2\beta+1)},
\end{equation}
where $\dgT$ is defined in \eqref{eq:def_d}.
In this case, the bound given in \eqref{eq:e_bound} can be written as $e_T(h,L) \le \DgT B^{-\beta/(2\beta+1)}$.
In order to have $e_T(h,L) \le \epsilon$, it suffices to have $B=(\epsilon/\DgT)^{-(2\beta+1)/\beta}$, which together with \eqref{eq:hL} yields \eqref{eq:opt_hL}.

In Algorithm \ref{alg:jsde}, each trajectory of \eqref{eq:euler} or \eqref{eq:taylor1} needs to have $(\sum_{i,j}\alpha_{ij}+\sum_i \gamma_i)T$ sampled points on expectation. Hence the total sampling complexity of $L$ trajectories is $(\sum_{i,j}\alpha_{ij}+\sum_i \gamma_i)TL$, which together with $L$ \eqref{eq:opt_hL} yields \eqref{eq:samp_cmp}.
In addition, each step of \eqref{eq:euler} or \eqref{eq:taylor1} is $O(m+n)$, and hence the total computational complexity is $O((m+n)LT/h)$. Substituting $h$ and $L$ by the values in \eqref{eq:opt_hL} yields \eqref{eq:comp_cmp}. This completes the proof.
\end{proof}


\section{SDE for more General Propagation Problems}
\label{sec:extension}
The basic propagation model with constant activation (and recovery) rates discussed above is widely used in a variety of applications including news and disease spread etc.
However, in many real-world applications, the propagations are often time and/or state dependent and the basic model with constant activations rates is not accurate.
In these cases, the vast existing methods relying on the constant rates in basic propagation model are not suitable.
On the other hand, our method based on jump SDE can be readily modified handle these cases by making the coefficients and jump intensity time and state dependent.
In this section, we depict such generalization of our method to two scenarios.
For conciseness of the present paper, we will report more in-depth analysis and numerical experiments of these cases in a forthcoming work.

\subsection{Time varying activation rates}
In some real-world applications, the impact of $i$ on $j$ may be diminishing along time, which mimics the phenomenon that older news/message makes less impact to a user's action. 
In this case, the activation rate $\alpha_{ij}(t)$ is time varying, for example, can be modeled as decaying such that $t_{ij}=t_j-t_i\sim \text{Weibull}(\alpha_{ij},\beta_{ij})$ for some $\beta_{ij}>0$ instead of $\text{Exp}(\alpha_{ij})$.
Here by $t\sim\text{Weibull}(\alpha,\beta)$ we mean that the probability density of $t$ is $f(t)=\beta\alpha^{\beta}t^{\beta-1}e^{-(\alpha t)^\beta}$ for $t>0$.
This yields a time varying activation rate $\alpha_{ij}(t)=\beta_{ij} \alpha_{ij}^{\beta_{ij}}t^{\beta_{ij}-1}$ of $i$ on $j$ where $t$ is the time since $i$ got activated.
Note that with $\beta_{ij}=1$ there is $t_{ij}\sim\text{Exp}(\alpha_{ij})$ and $\alpha_{ij}(t)\equiv \alpha_{ij}$, which reduces to the basic propagation model.
Similar modifications can be made to the recovery rate so that $\gamma_i(t)$ is time varying.

Time varying activation rates cause significant computational challenge for existing methods, since the propagation is no longer Markov.
Moreover, Monte Carlo simulation becomes difficult and computationally demanding due to the dependency of propagation on its entire history.
However, our approach can easily address this issue by incorporating additional variable into the SDE system.
More specifically, we introduce an auxiliary variable $U_j(t)$ for every node $j$, and establish a system of SDEs of $\{U_j(t),X_j(t):j\in\Vcal,t\in[0,T]\}$.
For ease of presentation, we consider time-varying activation rates but still assume constant recovery rates, since the further generalization is trivial.

We first observe that the key to incorporating the time dependency of activation rates is to record the time elapsed since the last activation of every node $j$.
This time is denoted by $U_j(t)$, and the coupled SDE system of $U_j(t)$ and $X_j(t)$ is give below:
\begin{equation}\label{eq:nonexpJSDE}
\left\{
\begin{split}
\dif X_j(t) & = \sbr[2]{ (1-X_j(t^-))\sum\nolimits_{i\in \Vcal_j^+}X_i(t^-)\dif N_{ij}(t,U_i(t^-)) } -X_j(t^-)\dif R_j(t) \\
\dif U_j(t) & = X_j(t^-) \dif t - U_j(t^-)\dif R_j(t)
\end{split}
\right.
\end{equation}
where $U_j(t)$ is the time since last activation of $j$, and $U_j(t)=0$ if $j$ is currently inactive at time $t$.
To see this, we first observe that $X_j(t)$ is binary-valued at 0 and 1 to indicate the activation state of $j$ at time $t$.
Therefore, the instantaneous rate of $U_j(t)$ is given by $X_j(t^-)$ and hence $U_j(t)$ accumulates at rate $1$ when $j$ is active, whereas the value of $U_j(t)$ drops to $0$ every time $j$ recovers, i.e., $X_j(t^-)\dif R_j(t)=1$.
In \eqref{eq:nonexpJSDE}, the intensity of the Poisson process $N_{ij}(t,U_i(t^-))$ depends on $U_i$, the time since the last activation of the parent node $i$.
For example, if the activation time follows the Weibull distribution with $\alpha_{ij}$ and $\beta_{ij}$ as parameters, then we can obtain the intensity of $N_{ij}(t,U_i(t))$ as $\beta_{ij} \alpha_{ij}^{\beta_{ij}}(U_i(t))^{\beta_{ij}-1}$.
Derivations for other types of distributions are similar.

\subsection{State-dependent propagation}
In some applications, the rate $\alpha_j(t)$ for a node $j$ to get activated at time $t$ is not simply the sum of $\alpha_{ij}$ over its active parents. 
Instead, $\alpha_j$ can be a nonlinear function of these $\alpha_{ij}$.
For example, the rate for $j$ to get activated is $\alpha_j(t) = \min\{a_j, \sum_{i}\alpha_{ij}X_{i}(t^-)\}$ as a nonlinear function of activation states of its parents.
Namely, $a_j>0$ is a personal threshold such that the rate $\alpha_j$ is throttled and more active parents will not further increase the rate $\alpha_j$.
In this case, we can write the processes $N_{ij}(t)$ to be state dependent, i.e., $N_{ij}(t)=N_{ij}(t,\Xbm(t^-))$, with intensity $\min\{a_j,\sum_{i}\alpha_{ij}X_{i}(t^-)\}$. 
The SDE solver can be employed in the same way.

\section{Numerical Experiments}
\label{sec:results}
To demonstrate the effectiveness of the proposed method, we conduct extensive numerical experiments of influence prediction using various networks. 
More specifically, we test the proposed method on influence predictions problems on artificially generated networks for which ground truth influence can be obtained by large number of naive Monte Carlo simulations.
For comparison purpose, we also implemented two state-of-the-art influence estimation methods and plot their results to show the improved efficiency and flexibility of the proposed method.

\subsection{Experiment setup}
\label{subsec:setup}
To generate networks in our experiments, we used the CONTEST package \cite{Taylor:2009a} which is freely available to public online.
The CONTEST package can generate many types of networks given network size and specific parameters.
We conducted our numerical tests by using different types of networks, and found that the influence prediction results are very similar.
Therefore, for sake of conciseness, we only show the results using three types of networks: Erd\H{o}s-R\'{e}nyi network, small-world network, and scale-free (preferential attachment) network in this paper, as they are typical networks and are widely used in real-world applications.
Unless otherwise noted, the source set of nodes are randomly selected, and fix it for all comparison methods in each test.

The ground truth influence of each test is obtained by simulating a large number of propagations using Monte Carlo (MC) method and taking the sample mean.
More specifically, we generate $10,000$ propagations for a given network and source set selection, and obtain the empirical probability that node $i$ is activate at time $t$, denoted $x_i(t)$. 
The total influence is obtained by $\influ(t)=\sum_i x_i(t)$.
Therefore, $\influ(t)$ and $x_i(t)$ correspond to influence evaluation functions $g(\xbm)=1^\top \xbm$ and $g_i(\xbm)=x_i$ respectively.
In other words, $\ex[g(\Xbm(t))]=\sum_{i=1}^n \ex[X_i(t)]$ is the expected number of active nodes at time $t$, and $\ex[g_i(\Xbm(t))]=\ex[X_i(t)]$ is the probability that node $i$ is active at time $t$.
The estimated values are denoted by $\influhat(t)$ and $\xhat(t)$ respectively.

%
%
%
%

For comparison purpose, we also implemented two state-of-the-art methods to estimate influence on heterogenous networks: the mean-field approximation (labeled as \textbf{Mean Field}) with the first-order moment closure \cite{Van-Mieghem:2009a} and a sampling-based approximation method for continuous-time influence estimation (labeled as \textbf{ConTinEst}) \cite{Du:2013a}.
Mean Field approximates $x_i(t)$ by solving the following deterministic system of $n$ coupled nonlinear differential equations numerically using e.g., 4th order Runge-Kutta method:
\begin{equation}\label{eq:mfa}
x_j'(t)= \sbr[2]{(1-x_j(t))\sum\nolimits_{i\in \Vcal_j^+}\alpha_{ij}x_i(t) }-\gamma_j x_j(t),\quad j=1,\dots,n.
\end{equation}
The computational cost of Mean Field is very low since no stochasticity is involved. 
It is also worth noting that \eqref{eq:mfa} can be deduced by taking expectation on both sides of the jump SDE system \eqref{eq:jsde_recover}, (incorrectly) ignoring all correlations between terms in multiplications, and using the facts that $\ex[\dif N_{ij}(t)]=\alpha_{ij}\dif t$, $\ex[\dif R_j(t)] = \gamma_j(t)\dif t$ and $\ex[X_i(t^-)]=x_i(t)$.
More specifically, the product $\ex[X_i(t^-)X_j(t^-)]$ being replaced by $\ex[X_i(t^-)]\ex[X_j(t^-)]$ is due to the first-order moment closure as mentioned in \cite{Van-Mieghem:2009a} from a different point of view, and hence the result of Mean Field is overestimating the true influence.
Unfortunately, the error caused by such moment-closure cannot be estimated in general and hence it is unclear how to further improve the accuracy of Mean Field.
ConTinEst, on the other hand, is a recently developed approximation method based on effective samplings.
It employs the Least-Label-List technique presented in \cite{Cohen:1997a}.
ConTinEst can be applied to information propagation with infection time
distribution other than exponential. However, as the influence is estimated based on
coverage function, ConTinEst cannot estimate probability of an individual's activation state, nor the influence of propagations with recovery.

\subsection{Experiments on synthetic data}
In the first experiment, we test on the influence prediction problem using the basic propagation model (i.e., activation rates) without recovery scenario.
We first generate three networks: ER, small-world, and scale-free networks, each with $n=200$ nodes.  For Erd\H{o}s-R\'{e}nyi network, we randomly generate $m=[n\log(n)/2]$ (where $[a]$ denotes the integer closest to $a\in \Rbb$) edges and form a directed graph.
For small-world network, we start with a ring graph, and for each node we create a new link with probability $0.2$ and connect it to another node on the network randomly chosen with uniform distribution. 
For scale-free (preferential attachment) network, we add nodes one by one to the existing network, where each new node has $2$ links to the existing nodes, and the probability of linking to an existing node is proportional to the degree of that node.
For each network, we choose two nodes at random to form the source set. 
Then we simulate $10,000$ propagations using Monte Carlo method and compute the empirical mean of influence as the true influence (labeled as \textbf{True}).  
For ConTinEst method, we sample 1,000 trajectories and simulate 15 random labels in each trajectory. 
For the proposed jump SDE method, i.e., Algorithm \ref{alg:jsde}, we fix the time step size $h=0.01$ and sample $L=1,000$ trajectories. 
For the basic propagation model without recovery, the result of the comparison methods on these three networks are shown in Figure \ref{fig:influence_SI}. 

From Figure \ref{fig:influence_SI}, we can see that the proposed method generates highly accurate predictions as the estimated influence $\influhat(t)$ and individual activation probability $\xhat_i(t)$ have smallest relative error to the true $\influ(t)$ and $x_i(t)$ respectively (shown in the middle and right columns in Figure \ref{fig:influence_SI}) in every test network. 
This is in sharp contrast to Mean Field which yields much larger error than the proposed PropNet SDE.
ConTinEst also accurately estimated $\influ(t)$, however, it cannot estimate the individual activation probability $x_i(t)$ as the proposed PropNet SDE. 
\begin{figure}[t!]
\centering
\includegraphics[width=0.3\textwidth]{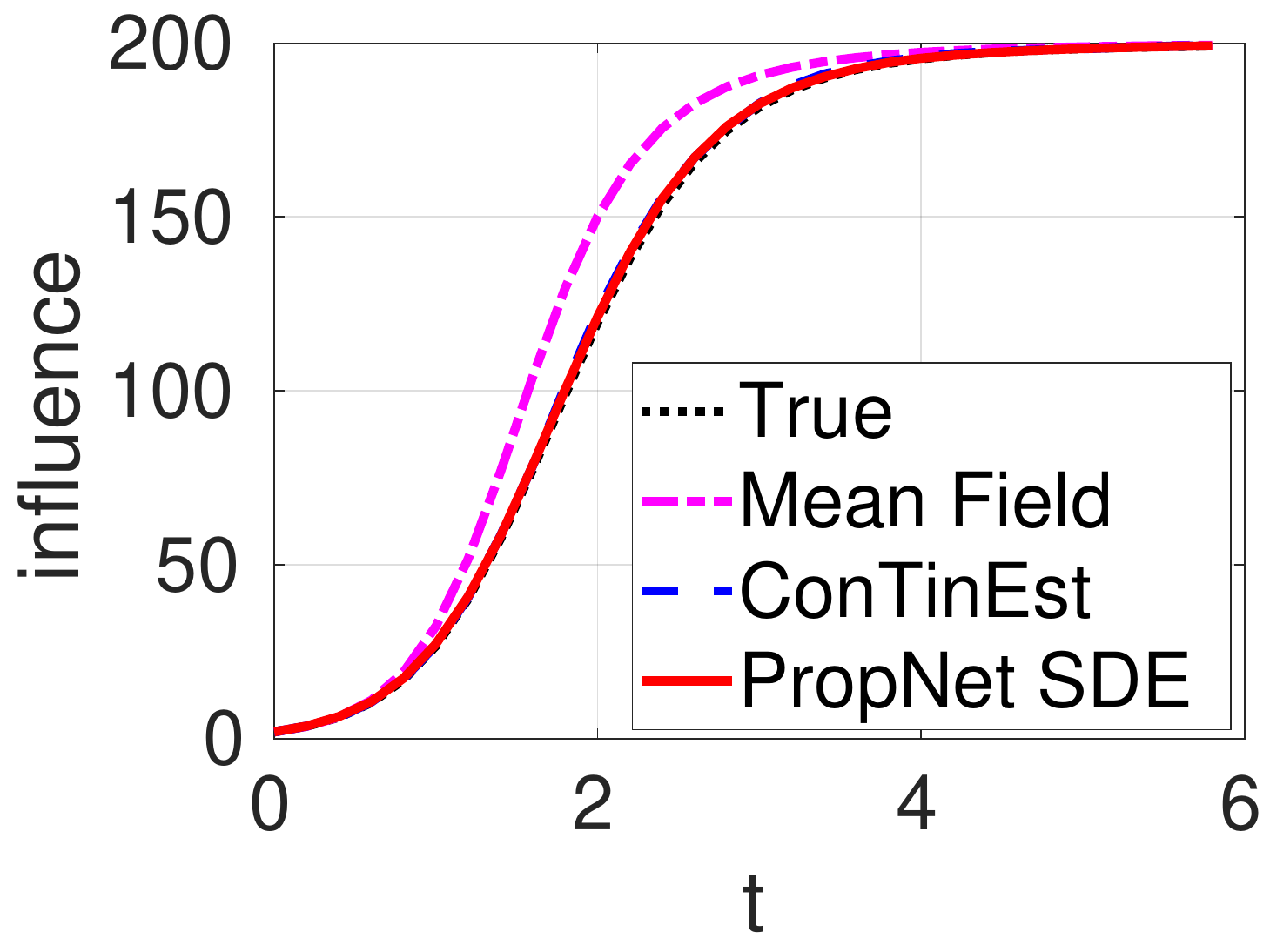}
\includegraphics[width=0.3\textwidth]{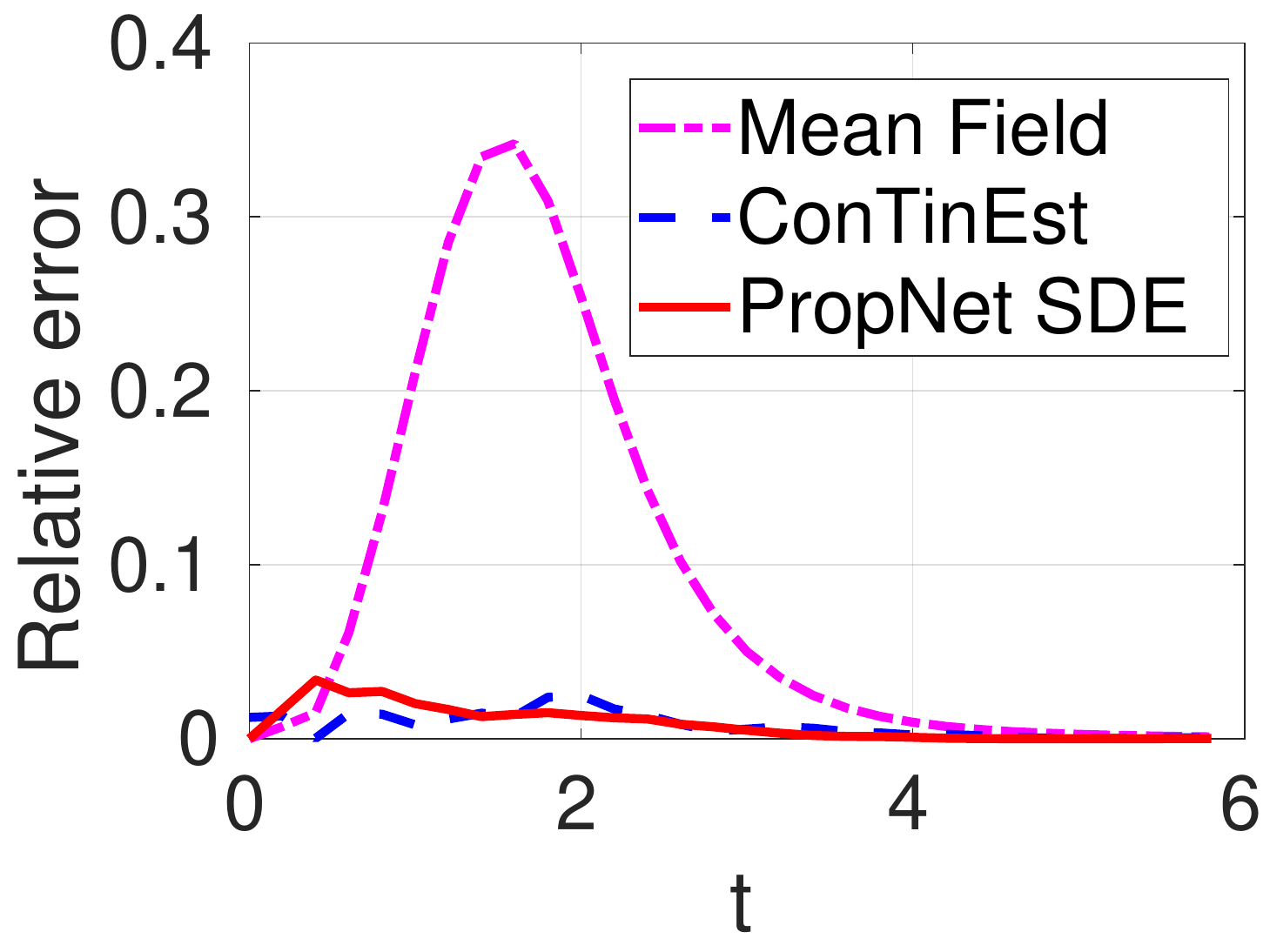}
\includegraphics[width=0.3\textwidth]{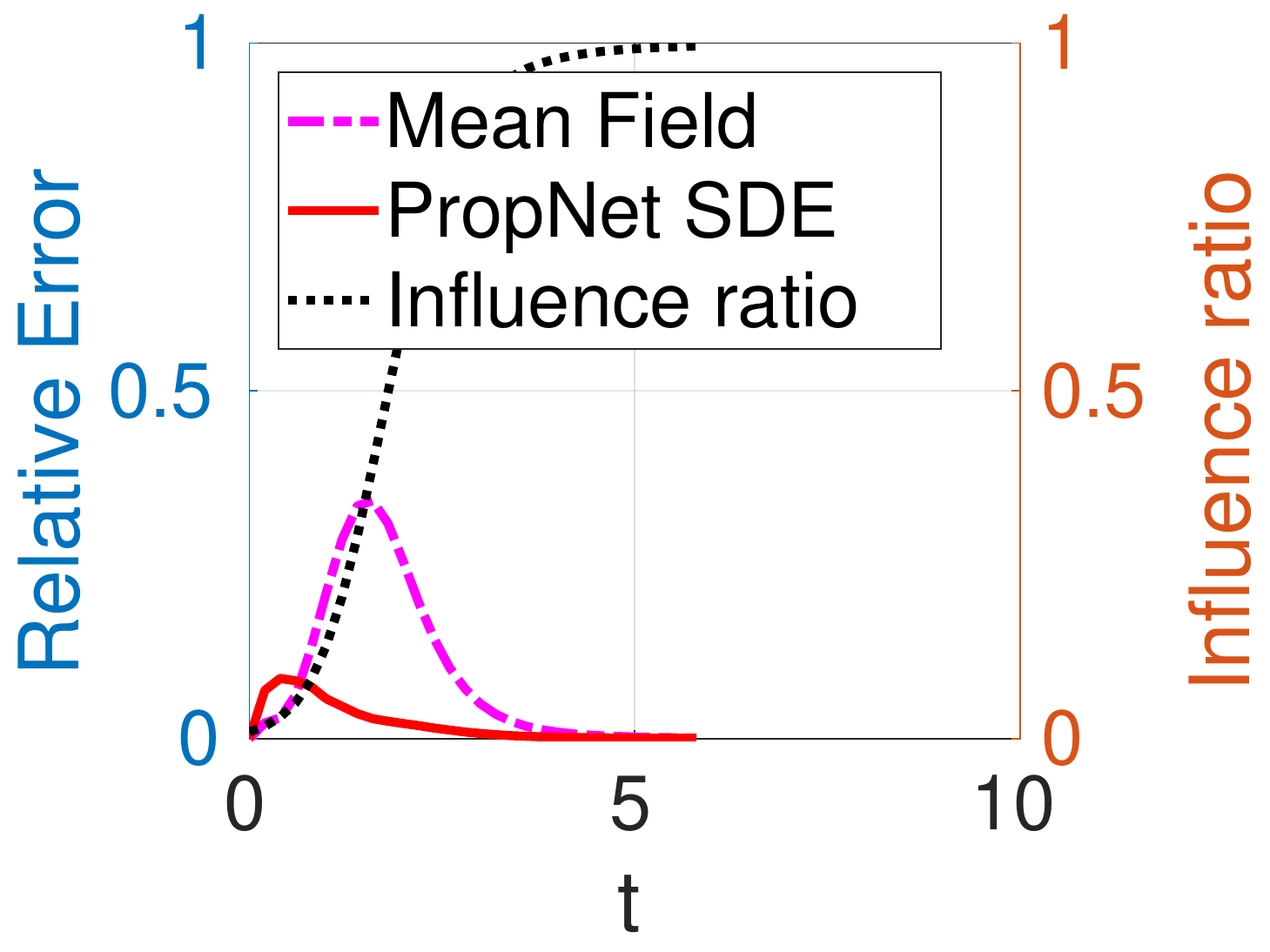}\\
\includegraphics[width=0.3\textwidth]{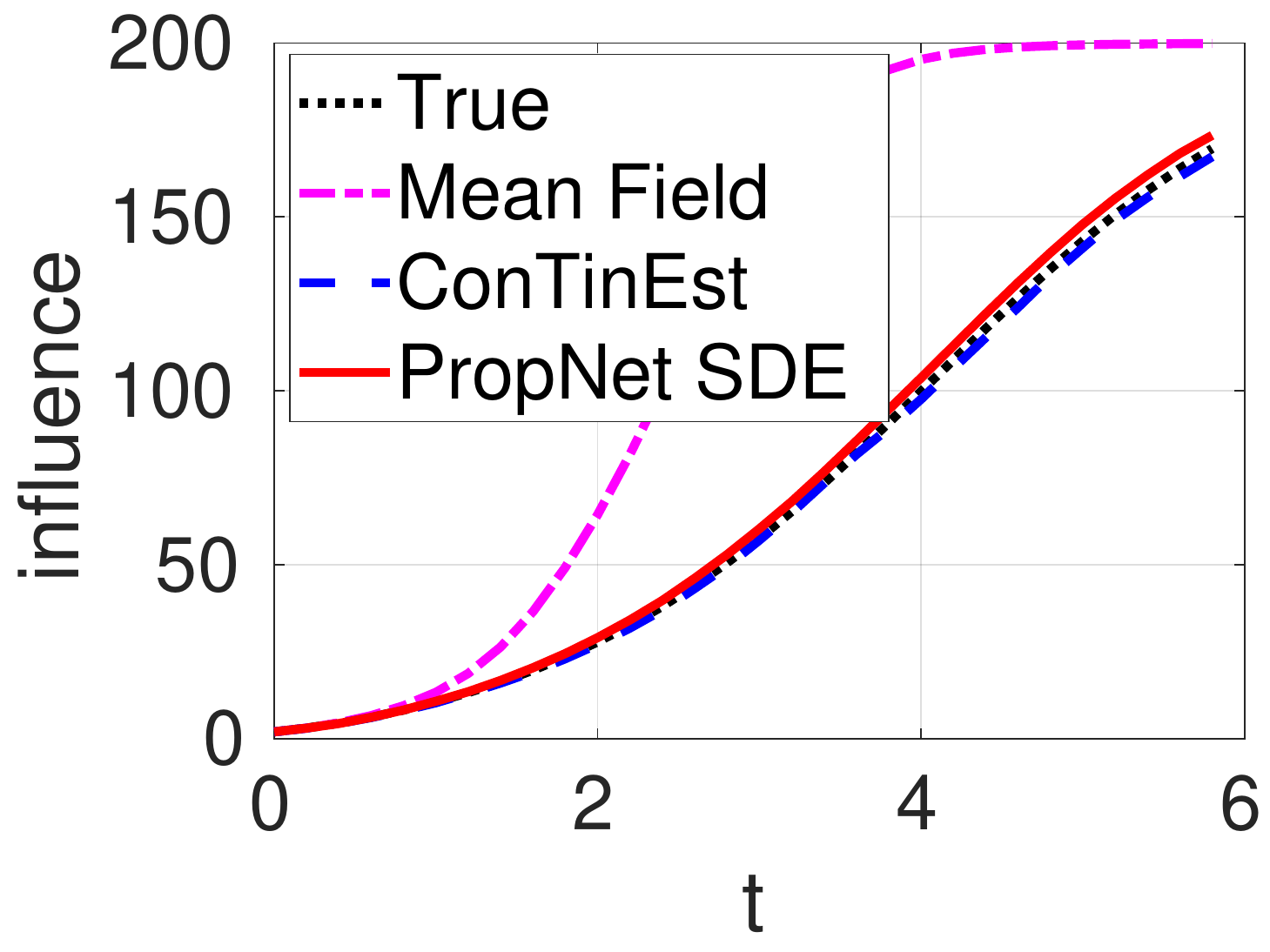}
\includegraphics[width=0.3\textwidth]{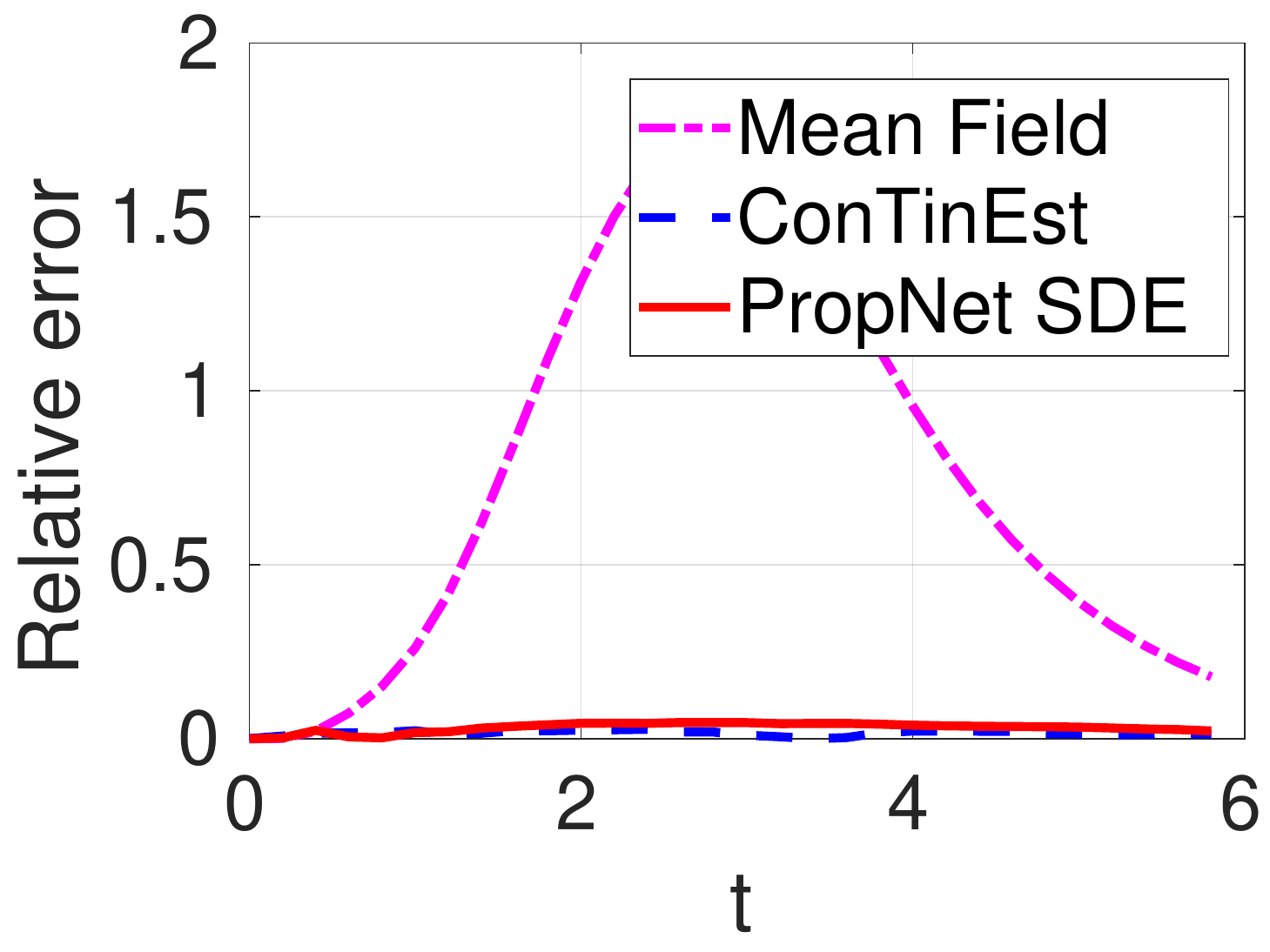}
\includegraphics[width=0.3\textwidth]{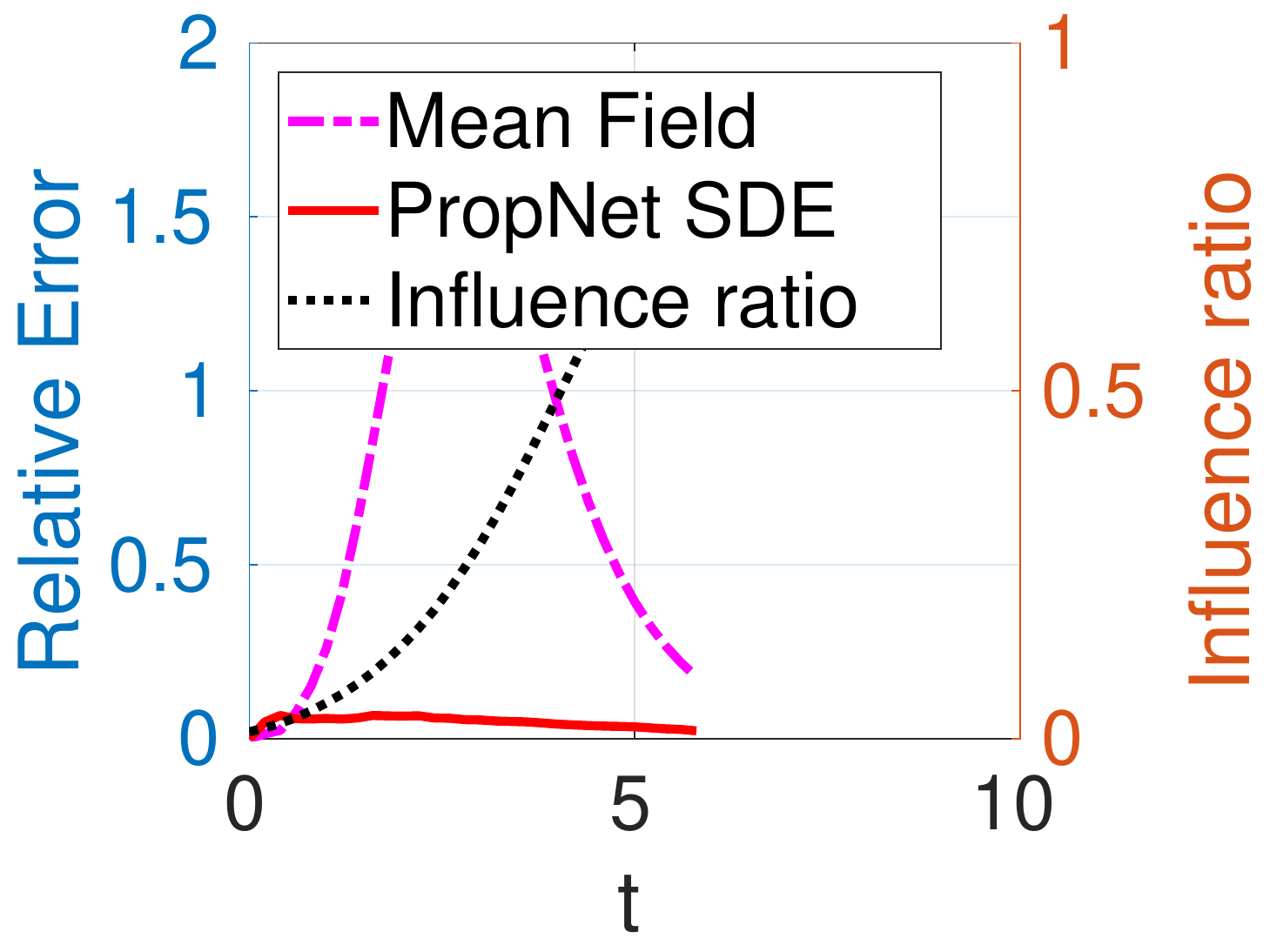}\\
\includegraphics[width=0.3\textwidth]{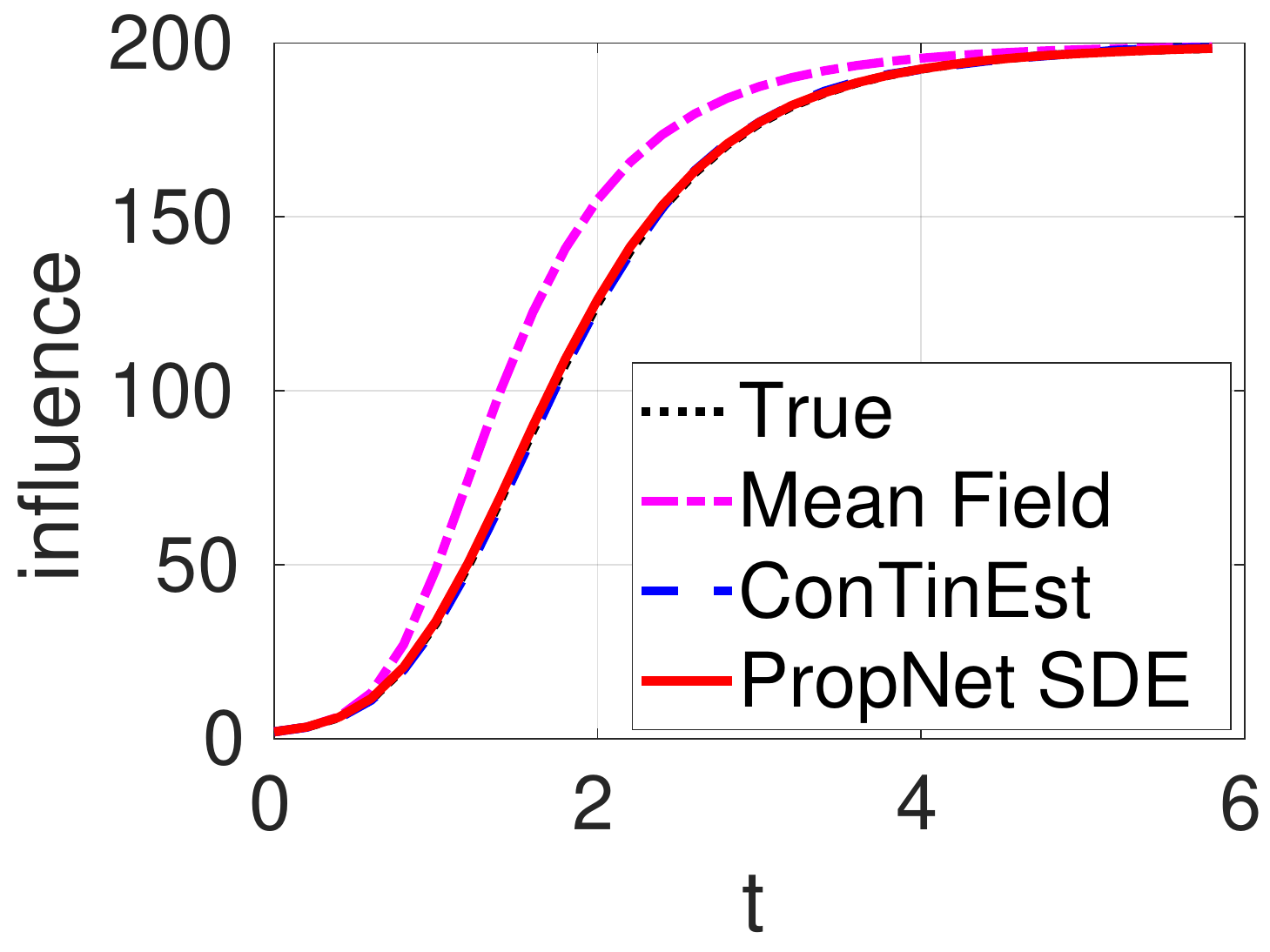}
\includegraphics[width=0.3\textwidth]{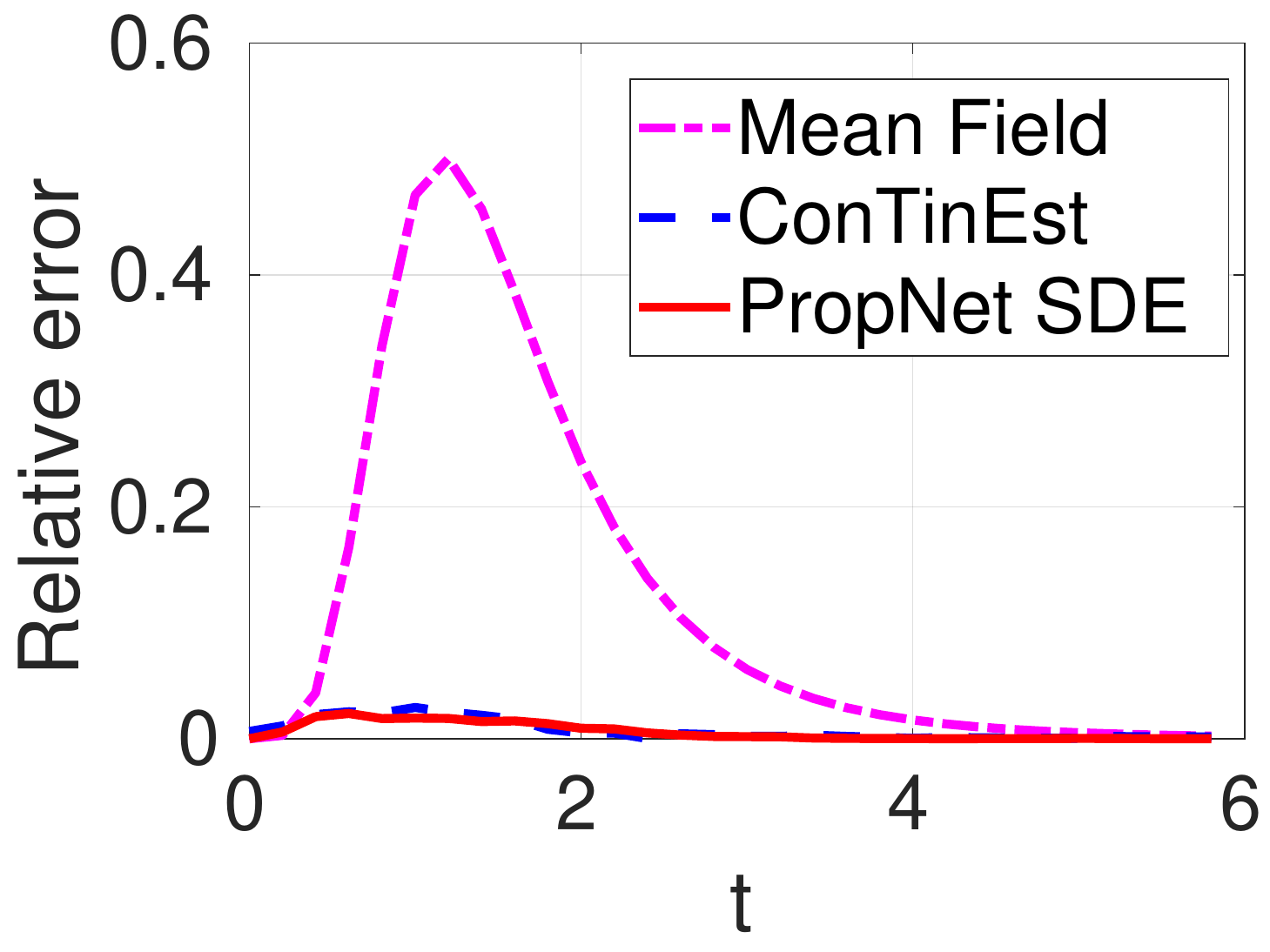}
\includegraphics[width=0.3\textwidth]{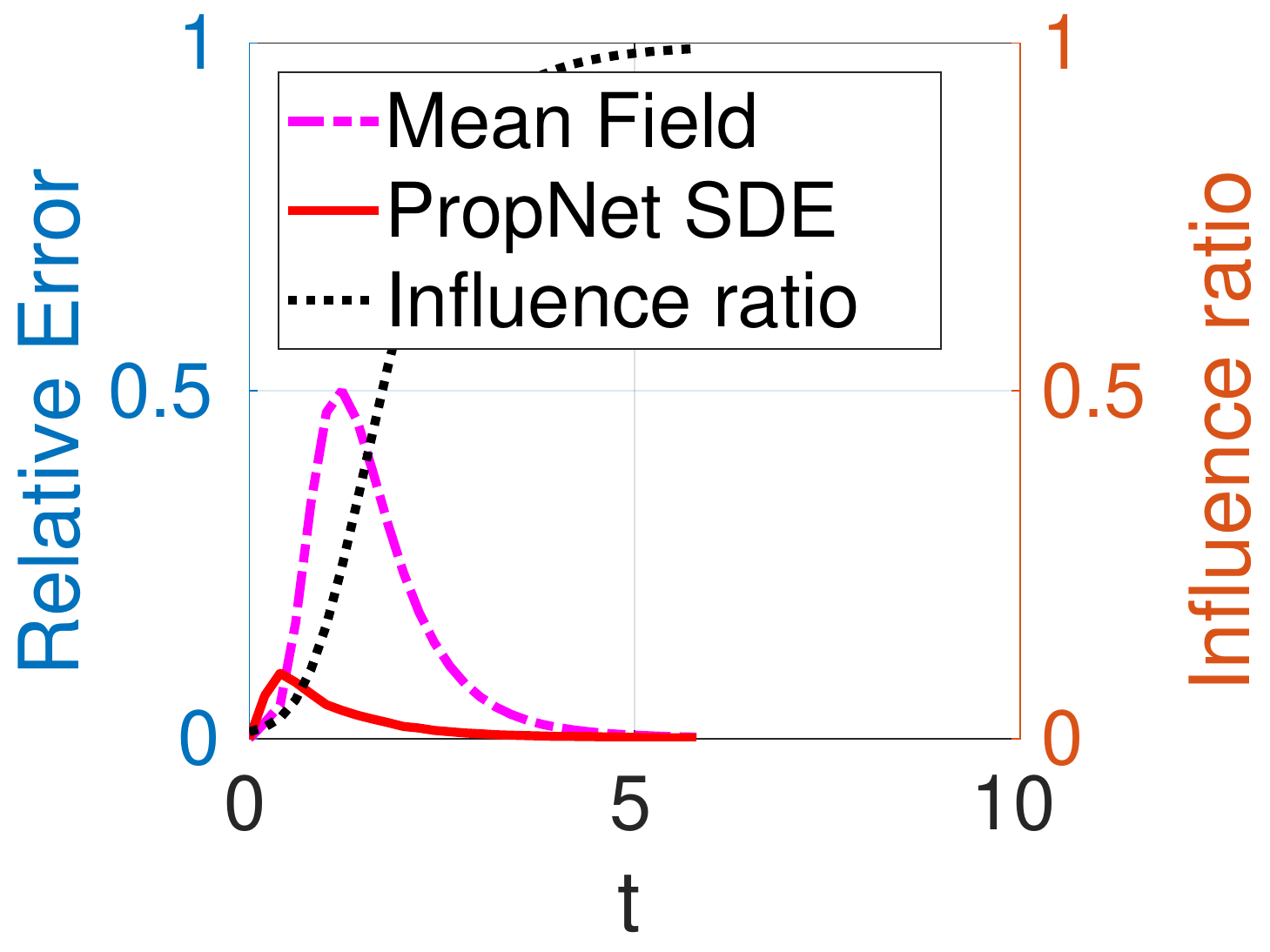}
\caption{Influence prediction by comparison methods on Erd\H{o}s-R\'{e}nyi network (top row), small-world network (middle row), and scale-free network (bottom row), all of size $n=200$, using the basic propagation model \underline{without} recovery. 
\textit{Left column}: True total influence $\influ(t)=\sum_i x_i(t)$ and influences $\influhat(t)=\sum_i \xhat_i(x)$ obtained by Mean Field, ConTinEst, and PropNet SDE.
\textit{Middle column}: relative error in influence $|\influhat(t)-\influ(t)|/\influ(t)$.
\textit{Right column}: relative error in individual activation probability $\sum_i|\xhat_i(t)-x_i(t)|/\sum_i x_i(t)$ (influence ratio $\influ(t)/n$ is plotted in black dotted line for reference).}
\label{fig:influence_SI}
\end{figure}

The second experiment is set similarly as the first one, but we incorporate recovery scenario such that an active node $i$ can recover at a constant rate $\gamma_i$.
In this test, the recovery rate for each node are chosen uniformly from $(0, 0.4)$. 
Since ConTinEst cannot handle the case with recoveries, we only compare Mean Field and PropNet SDE in this test.
The results are shown in Figure \ref{fig:influence_SIS}. 
We again observe that PropNet SDE has significant improvement over Mean Field in terms of accuracy, since the relative errors to $\influ(t)$ and $x_i(t)$ using PropNet SDE are much smaller than that using Mean Field.
\begin{figure}[t!]
\centering
\includegraphics[width=0.3\textwidth]{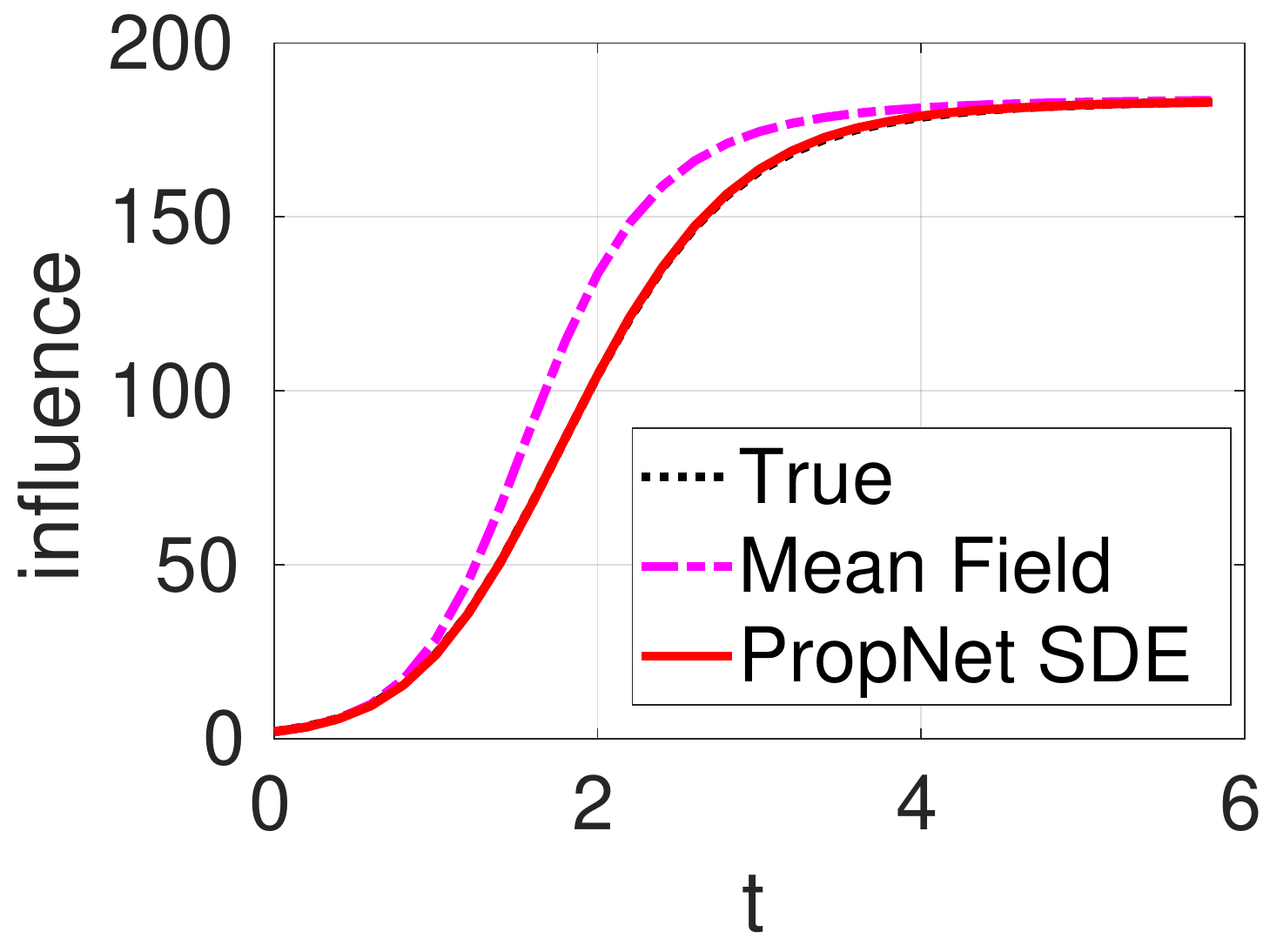}
\includegraphics[width=0.3\textwidth]{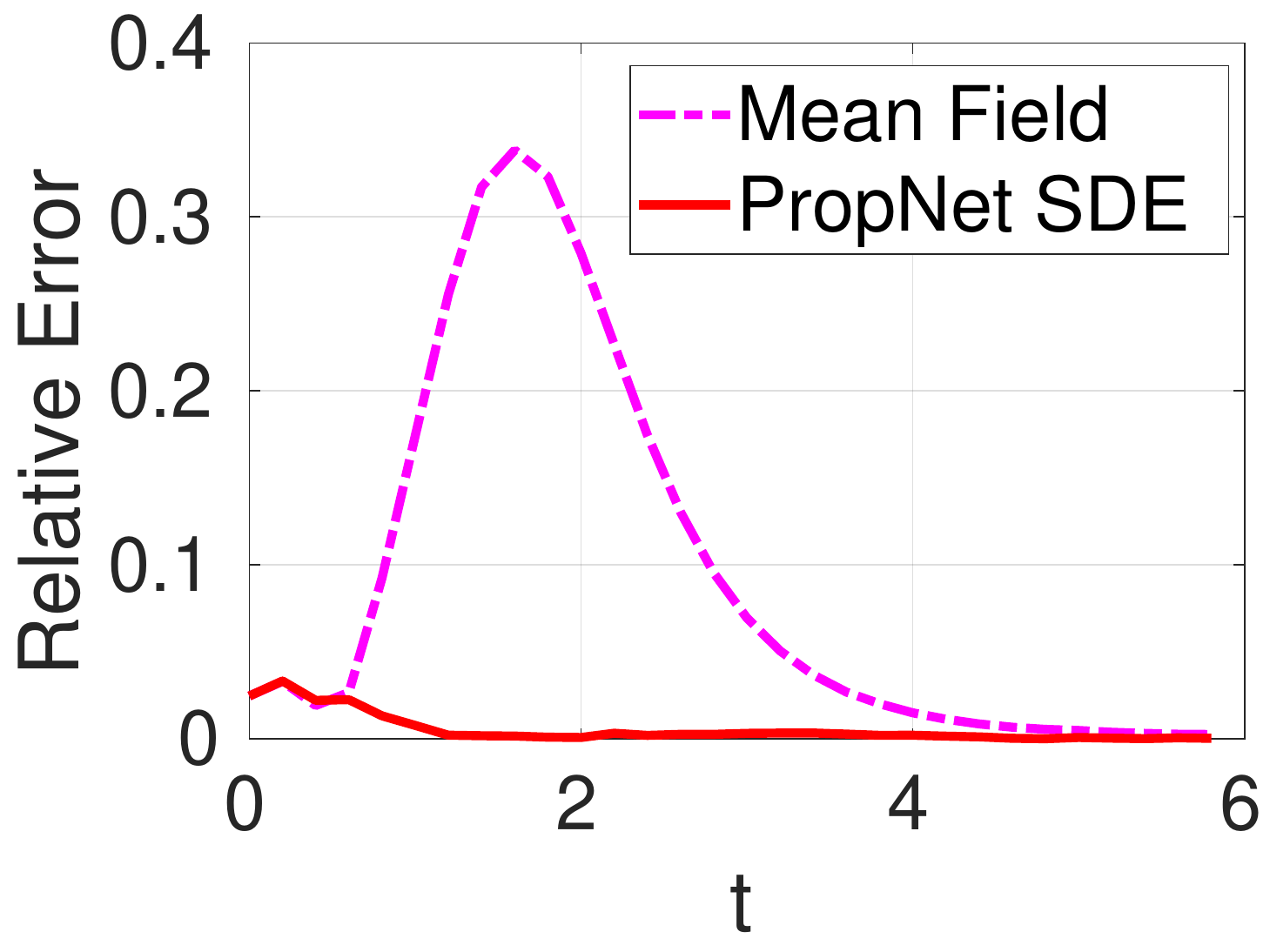}
\includegraphics[width=0.3\textwidth]{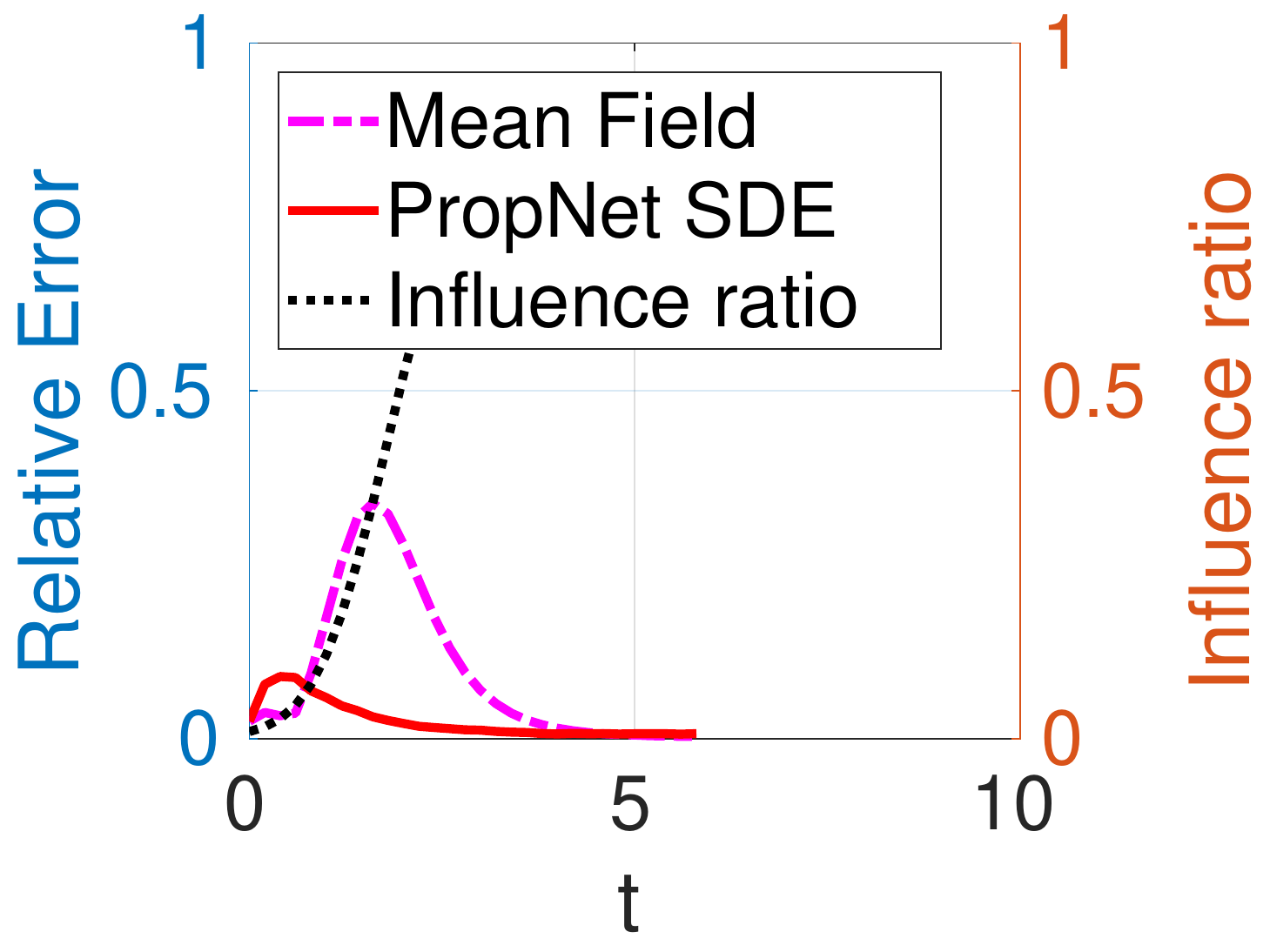}\\
\includegraphics[width=0.3\textwidth]{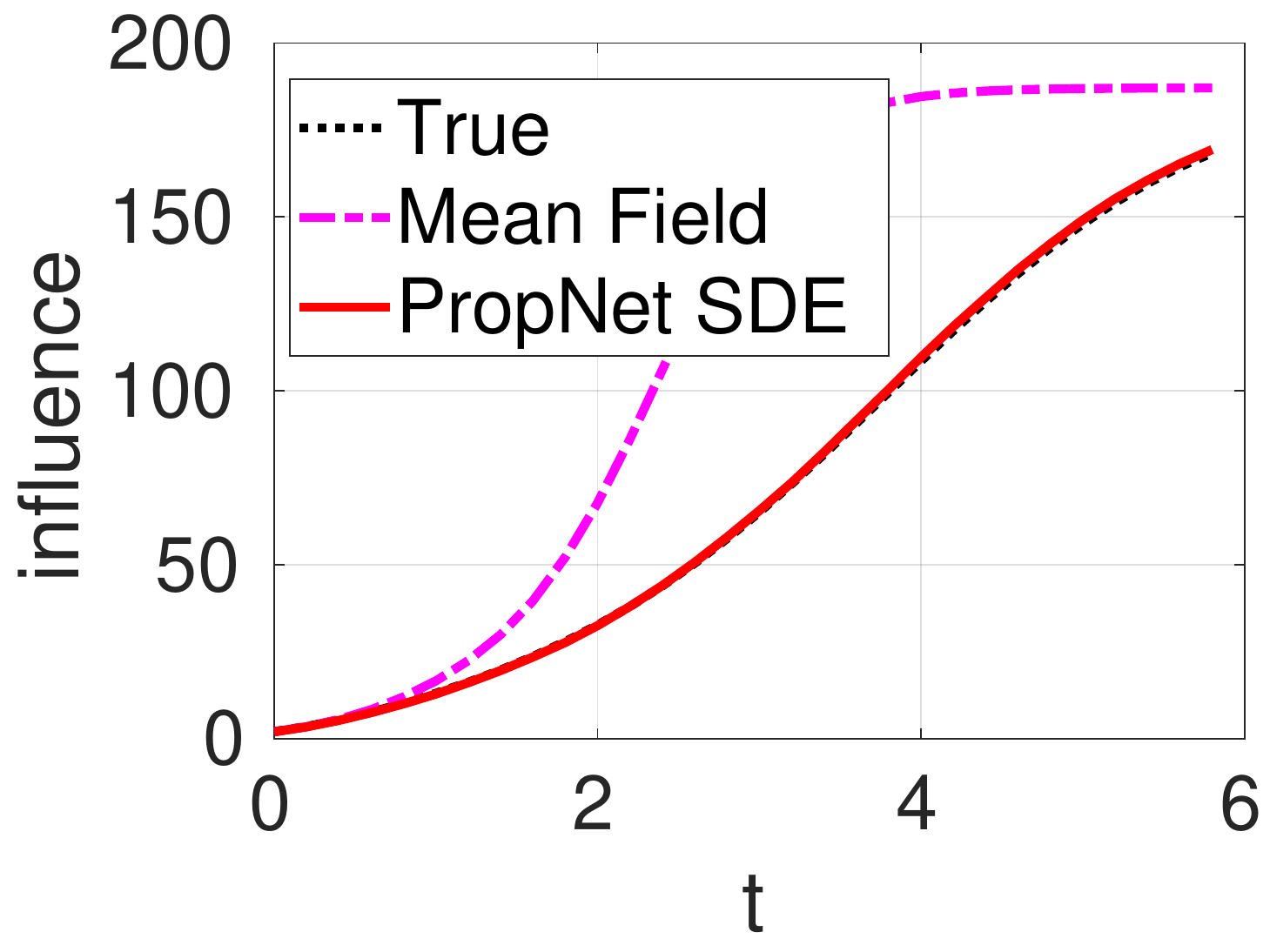}
\includegraphics[width=0.3\textwidth]{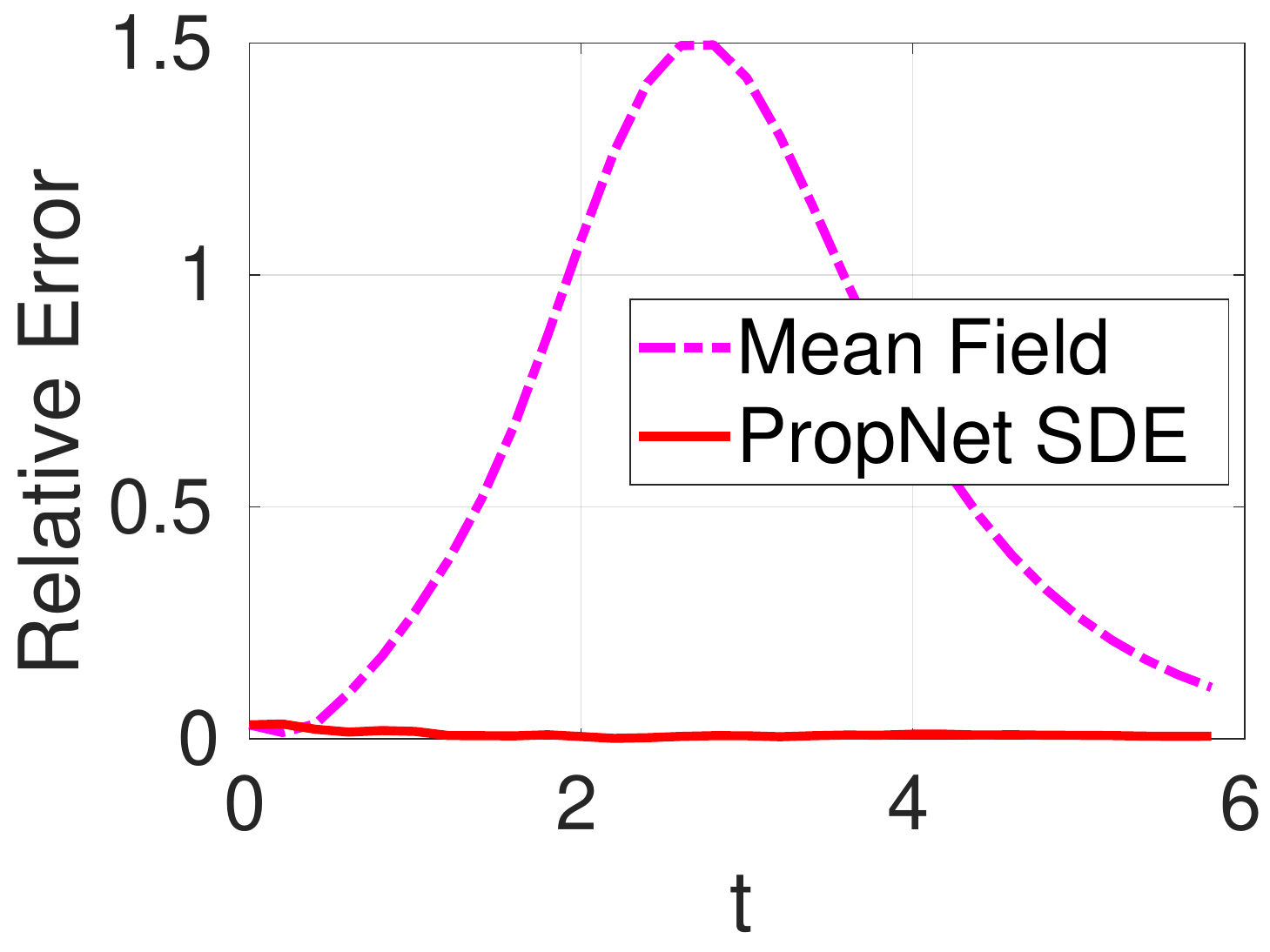}
\includegraphics[width=0.3\textwidth]{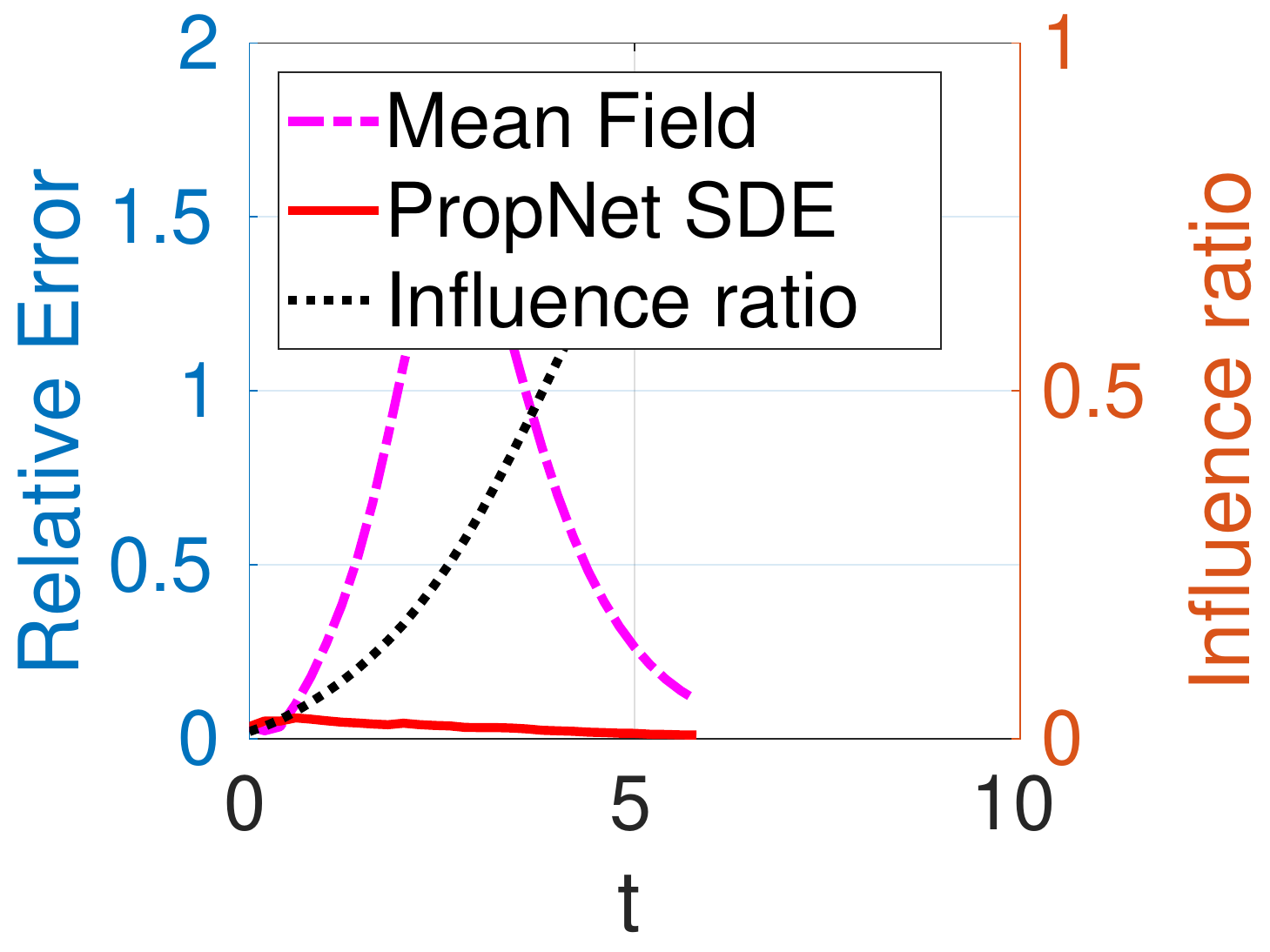}\\
\includegraphics[width=0.3\textwidth]{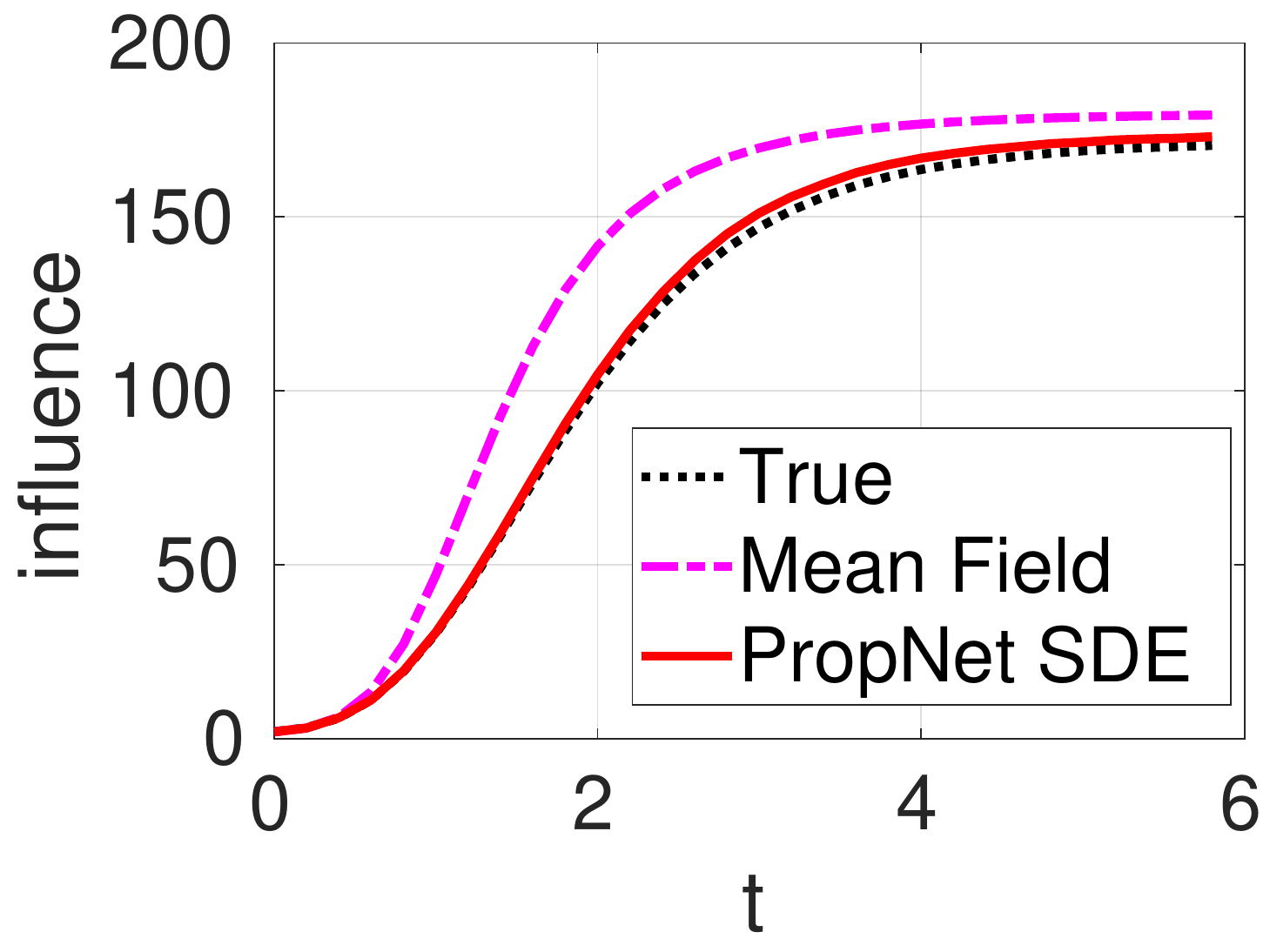}
\includegraphics[width=0.3\textwidth]{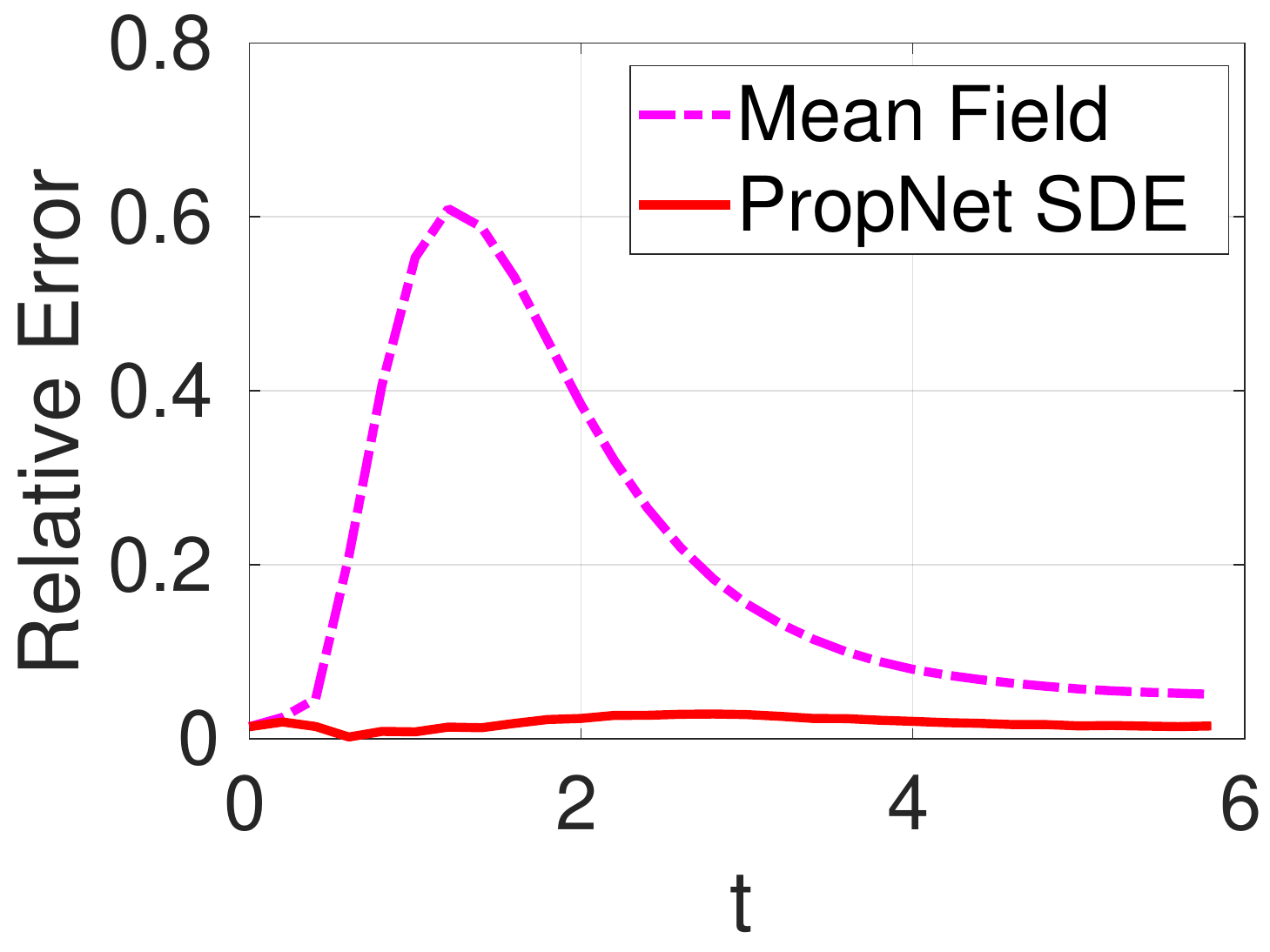}
\includegraphics[width=0.3\textwidth]{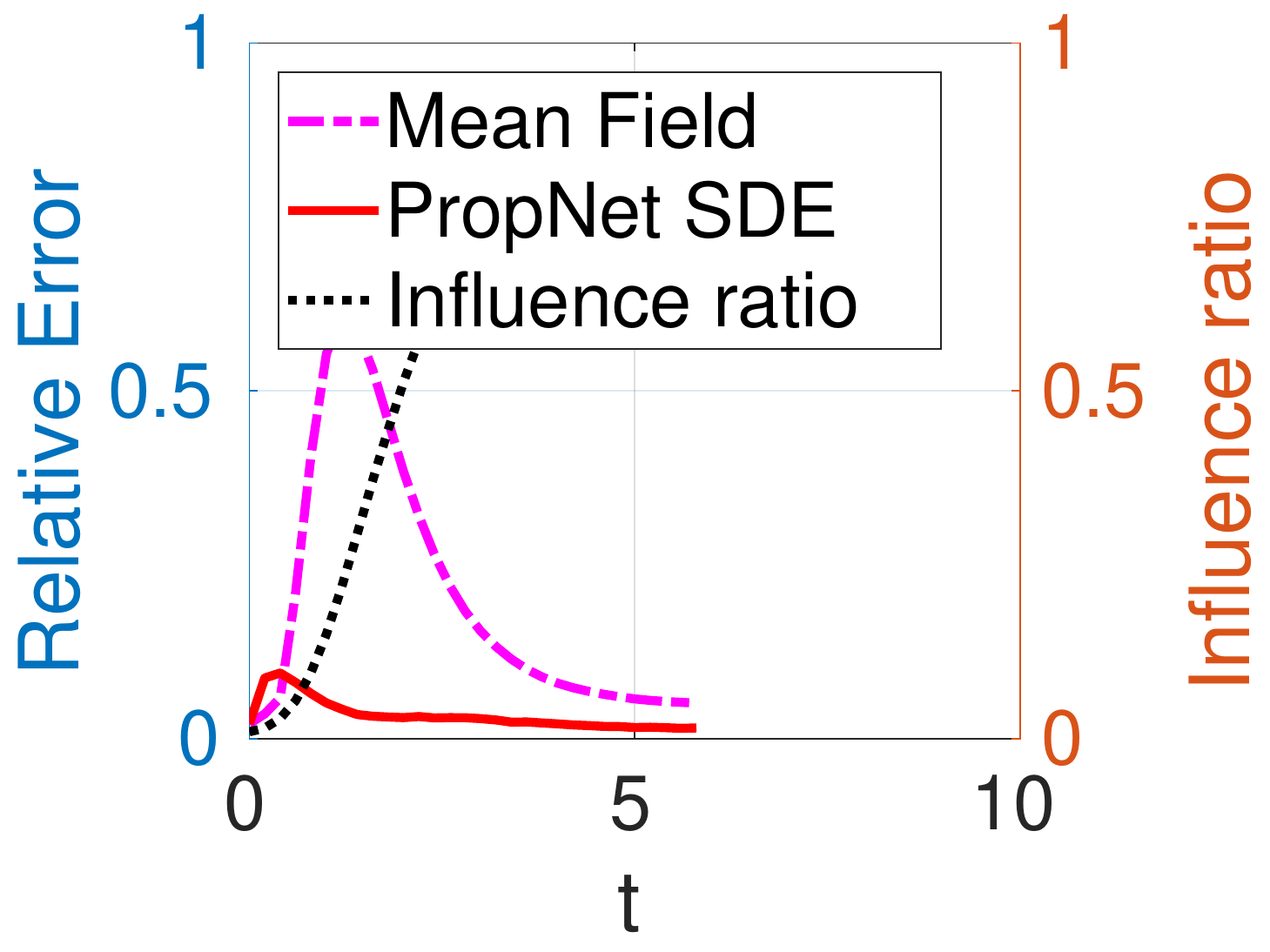}
\caption{Influence prediction by comparison methods on Erd\H{o}s-R\'{e}nyi network (top row), small-world network (middle row), and scale-free network (bottom row), all of size $n=200$, using the basic propagation model \underline{with} recovery. 
\textit{Left column}: True total influence $\influ(t)=\sum_i x_i(t)$ and influences $\influhat(t)=\sum_i \xhat_i(x)$ obtained by Mean Field and PropNet SDE.
\textit{Middle column}: relative error in influence $|\influhat(t)-\influ(t)|/\influ(t)$.
\textit{Right column}: relative error in individual activation probability $\sum_i|\xhat_i(t)-x_i(t)|/\sum_i x_i(t)$ (influence ratio $\influ(t)/n$ is plotted in black dotted line for reference).}
\label{fig:influence_SIS}
\end{figure}

In the third experiment, we evaluate the robustness of the comparison methods when the size of source set and network density vary.
For the first part of this experiment, we use the same three networks in the first and second experiments, and set five different sizes of source set: $n_0=2,5,10,15,20$. 
For each of the three networks and each of the five different $n_0$, we randomly select $n_0$ nodes as the source set, and apply the comparison methods to estimate influence of the basic propagation without recovery on the network and source set combination.
Such test is repeated for 20 times for each of the network and source set size combination, and the maximum absolute error, defined by $\max_{0\le t \le T} |\influhat(t) - \influ(t)|$, of all methods are shown in the first column of Figure \ref{fig:influence_density_source}(a).
In Figure \ref{fig:influence_density_source}(a), the standard deviation of ConTinEst and PropNet SDE are also shown at each source size (Mean Field is deterministic and hence no variance).
From Figure \ref{fig:influence_density_source}(a), we can see that PropNet SDE produces the lowest absolute error among the comparison methods, which indicates that it outperforms the others in terms of accuracy.
Figure \ref{fig:influence_density_source}(b) shows the result on the basic propagation model with recovery scenario, where we again observe that PropNet SDE is much more accurate than Mean Field, and ConTinEst is not capable to handle this situation.

In the second part of the third experiment, we generate the same three types of networks with different density levels (i.e., average node degree).
For Erd\H{o}s-R\'{e}nyi network, the number of edges is $m=[n\log(\kappa n)/2]$ for $\kappa=2,3,4,5,6$ repsectively.
For small-world network, the starting network is a regular graph where each node is connected to its respectively $\kappa=2,3,4,5,6$ nearest neighbors and the probability of creating a short cut of each node is set to $0.2$.
For scale-free network, the number of links that each newly added node has is set to $\kappa=2,3,4,5,6$ respectively.
For each network type, we generate a network using each of those density parameters $\kappa$ and apply ConTinEst and PropNet SDE for 20 times to obtain their means and standard deviations of maximum absolute error.
The results of the error versus the density parameter $\kappa$ by the comparison methods are given in Figure \ref{fig:influence_density_source} (c) for the case without recovery and (d) for the case with recovery, respectively.
From these plots, we can see that PropNet SDE consistently achieves smaller error than other methods in both cases without and with recovery.
\begin{figure}[t!]
\centering
\begin{tabular}{cccc}
\includegraphics[width=0.22\textwidth]{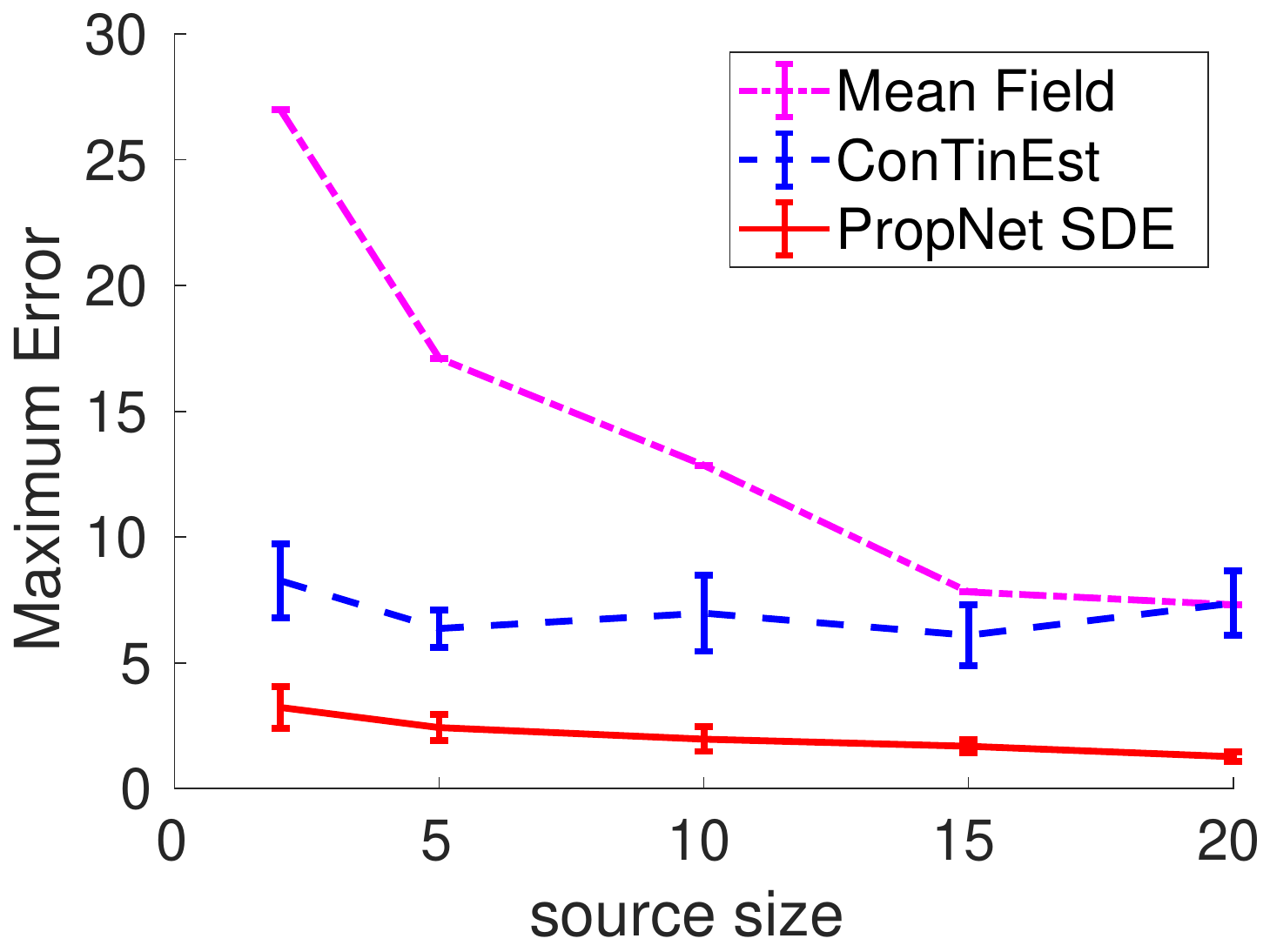} &
\includegraphics[width=0.22\textwidth]{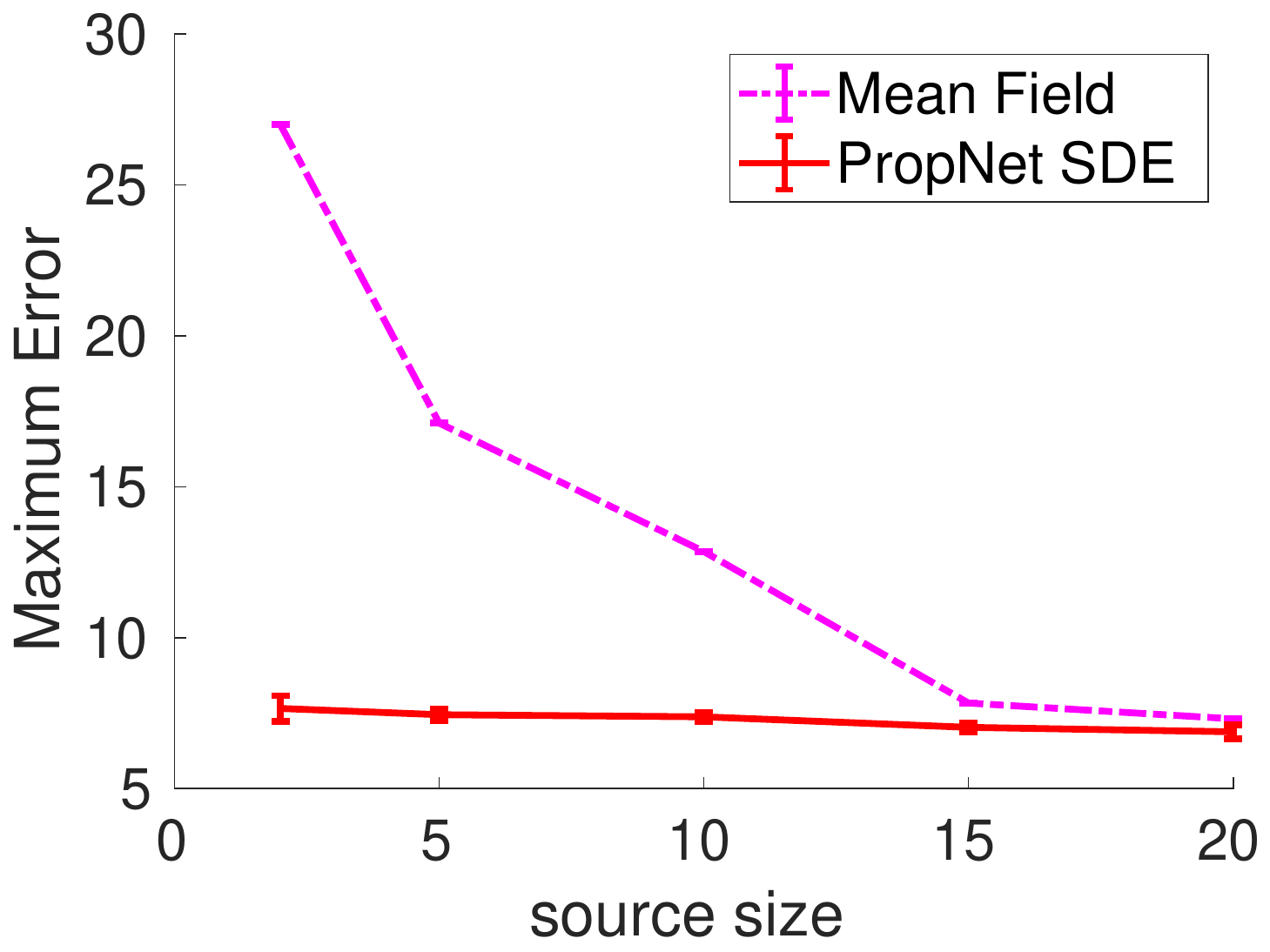} &
\includegraphics[width=0.22\textwidth]{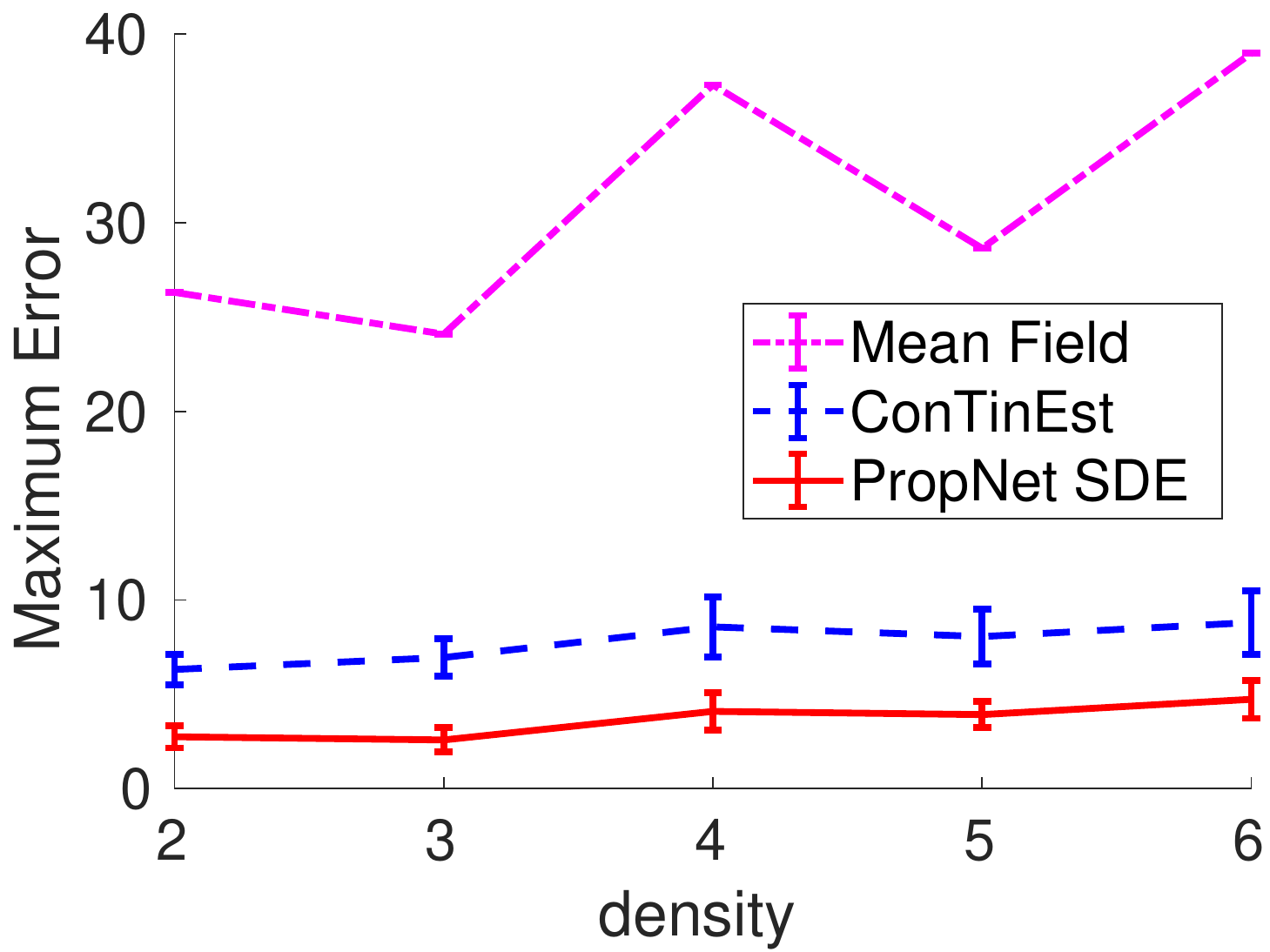} &
\includegraphics[width=0.22\textwidth]{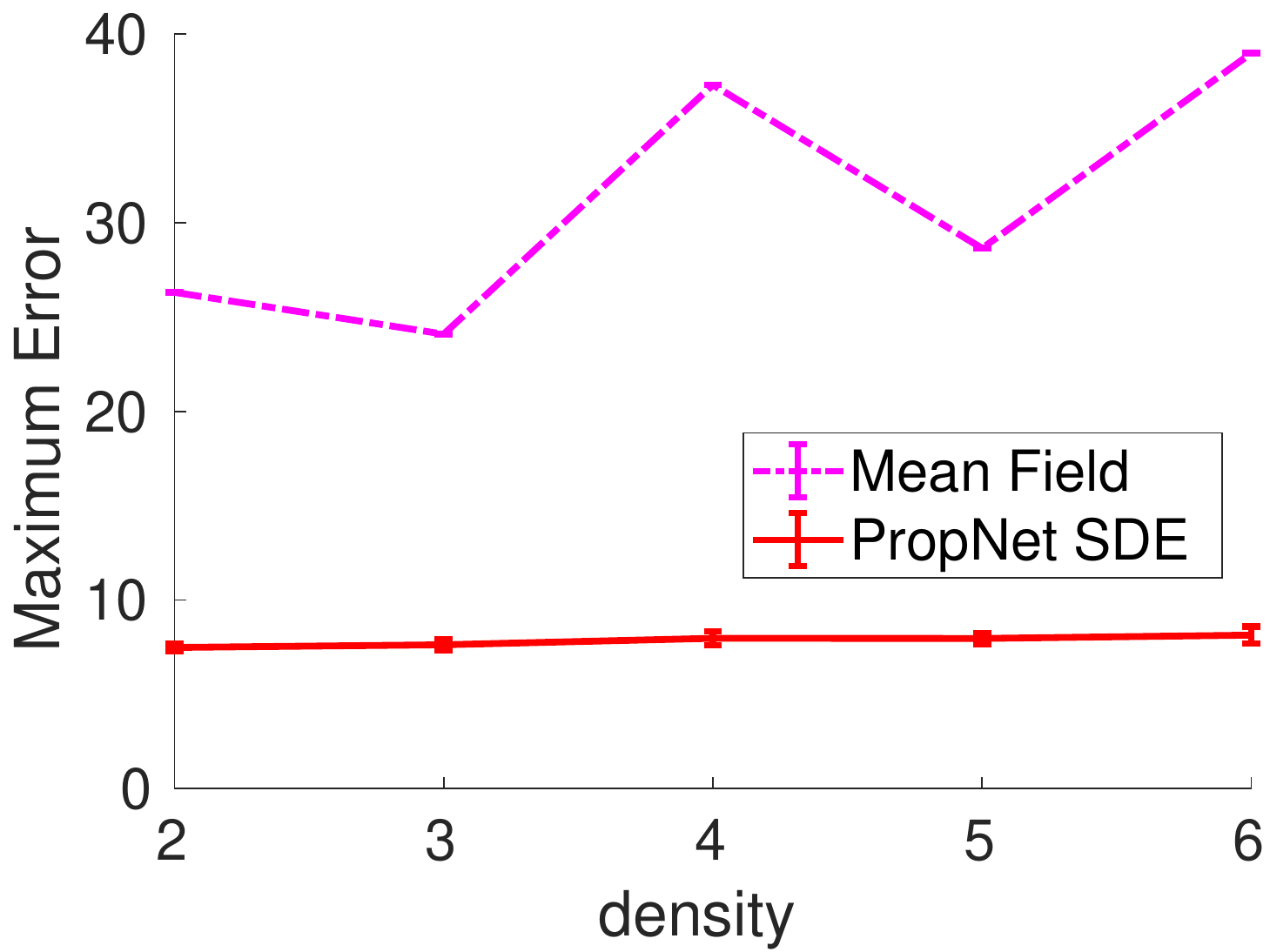}\\
\includegraphics[width=0.22\textwidth]{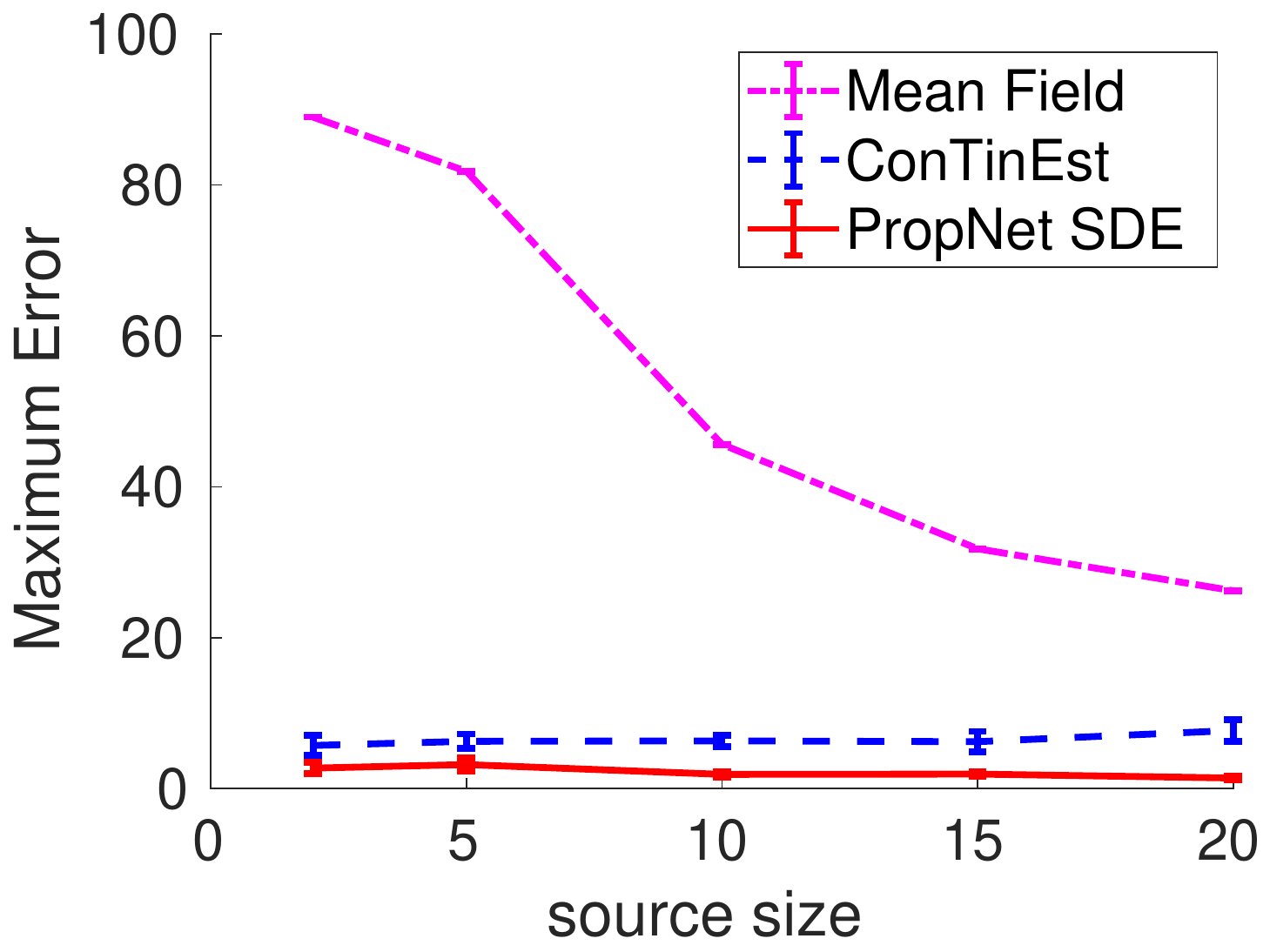} &
\includegraphics[width=0.22\textwidth]{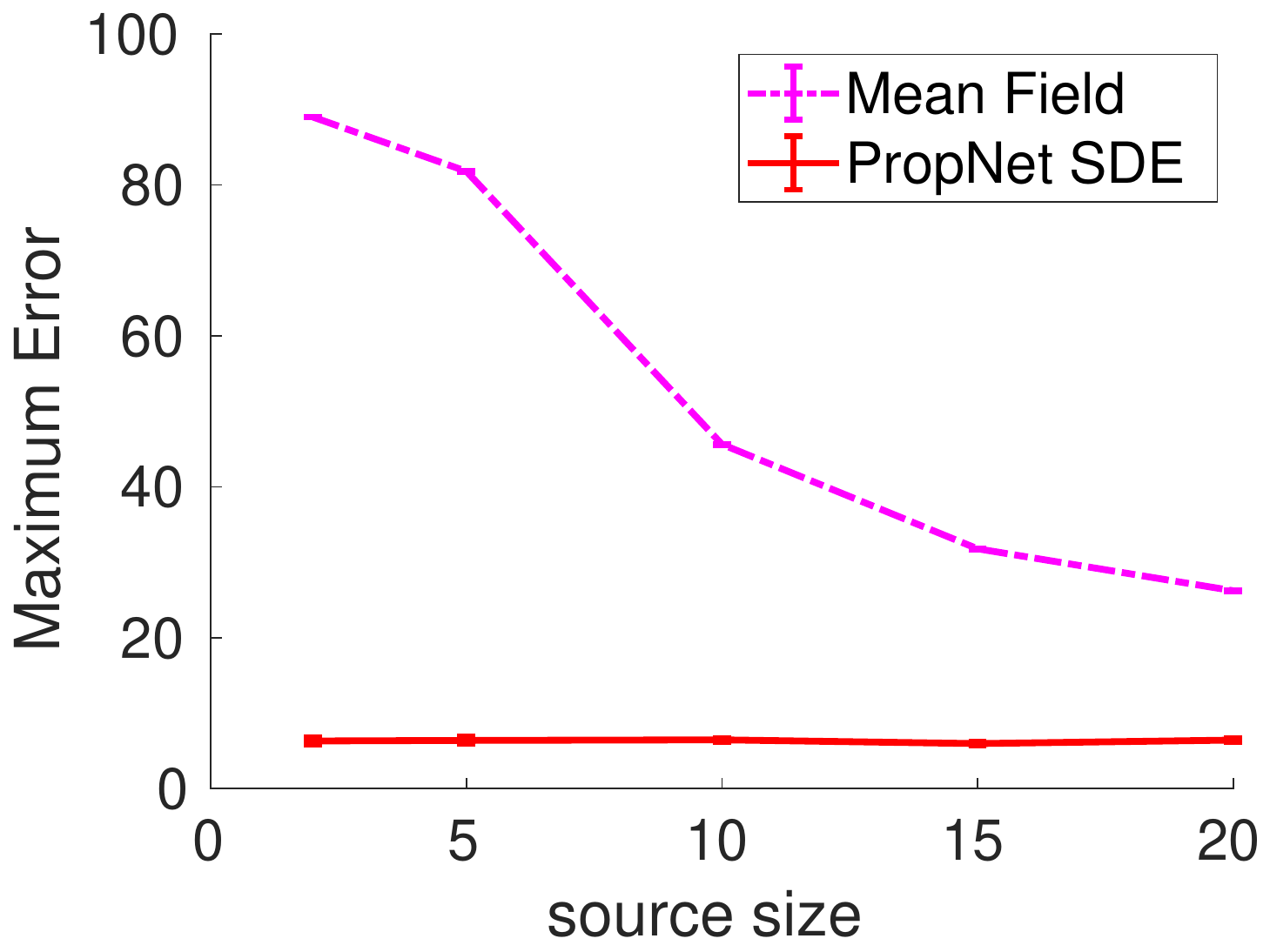} &
\includegraphics[width=0.22\textwidth]{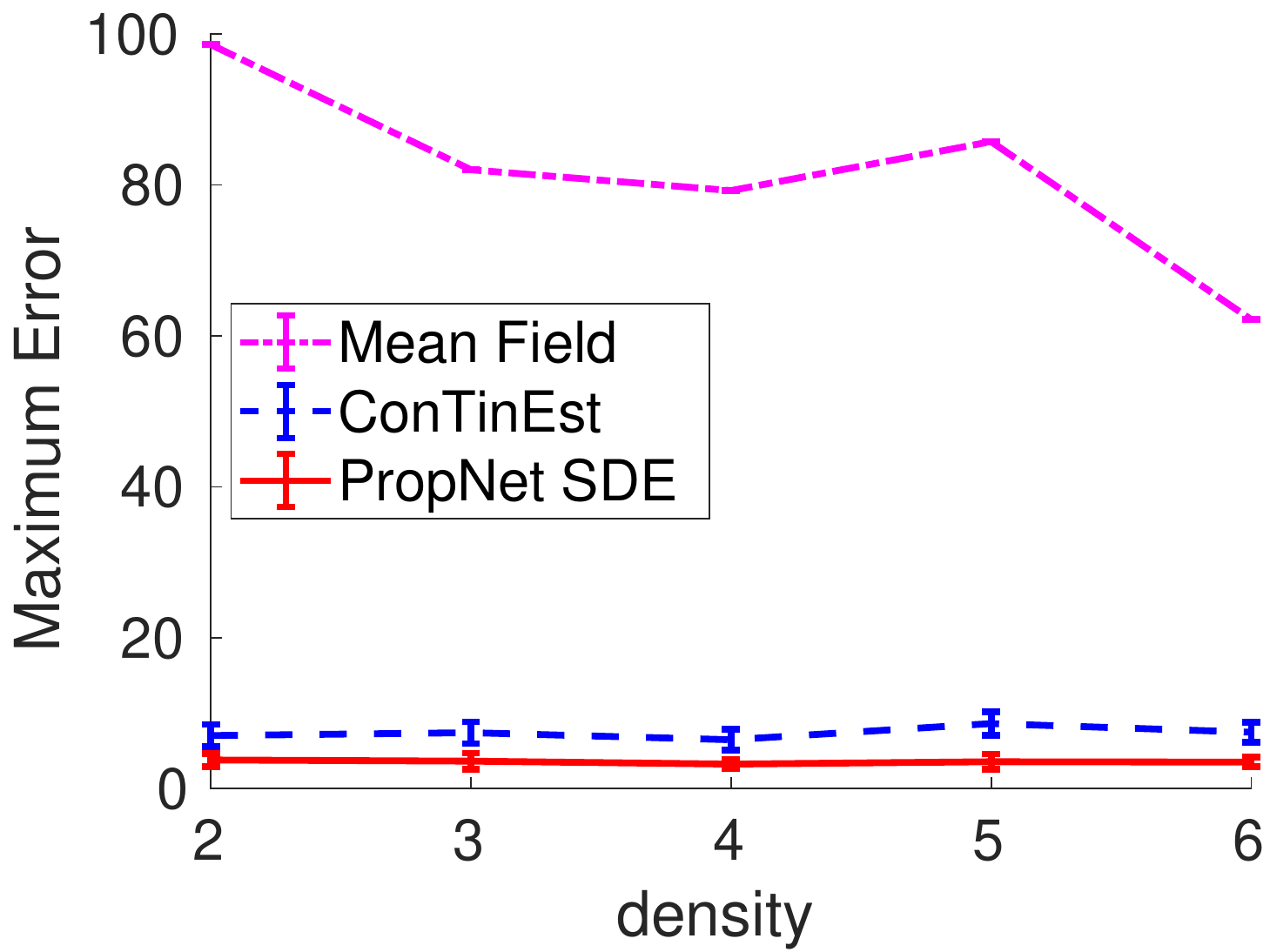} &
\includegraphics[width=0.22\textwidth]{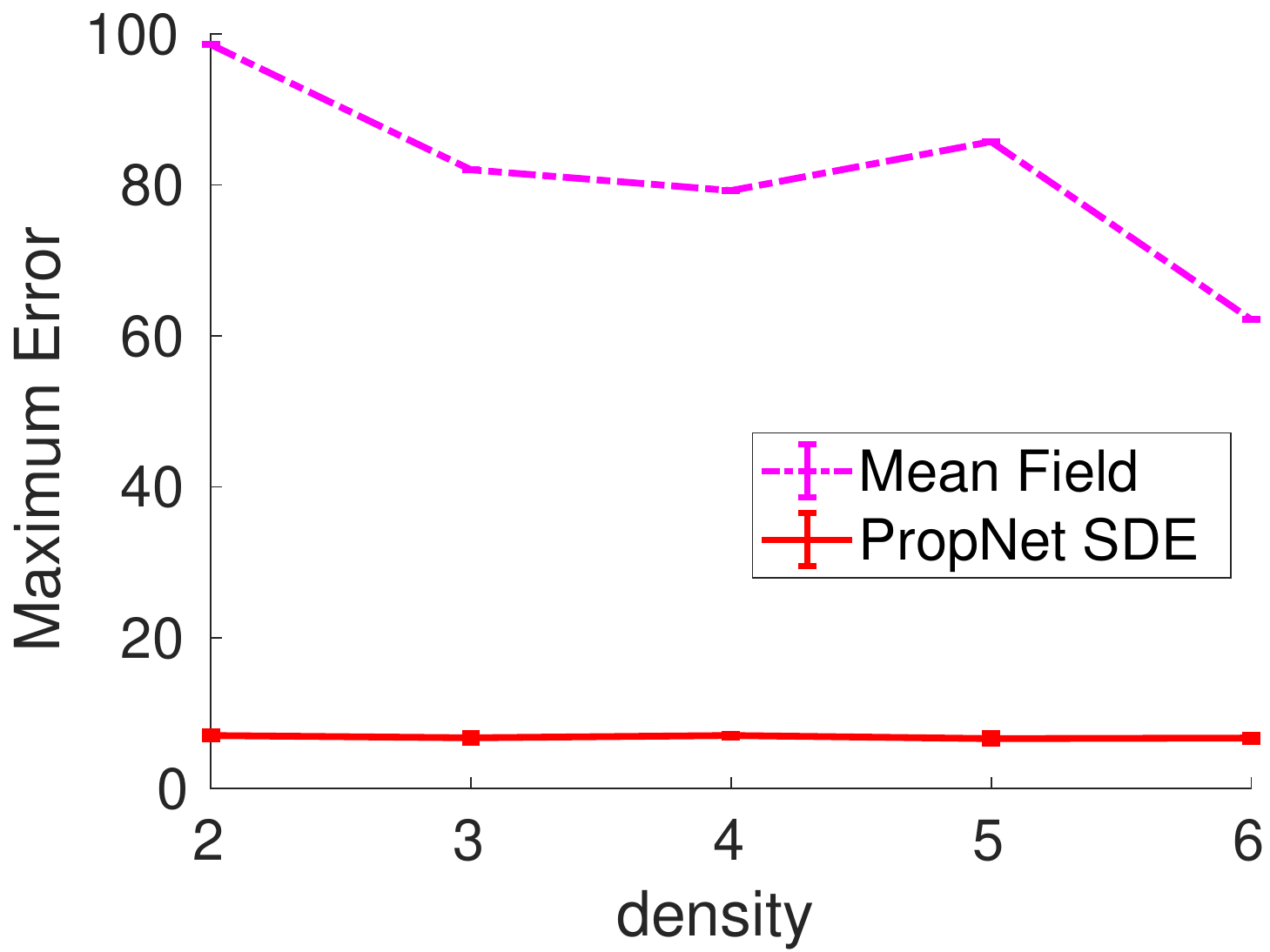}\\\includegraphics[width=0.22\textwidth]{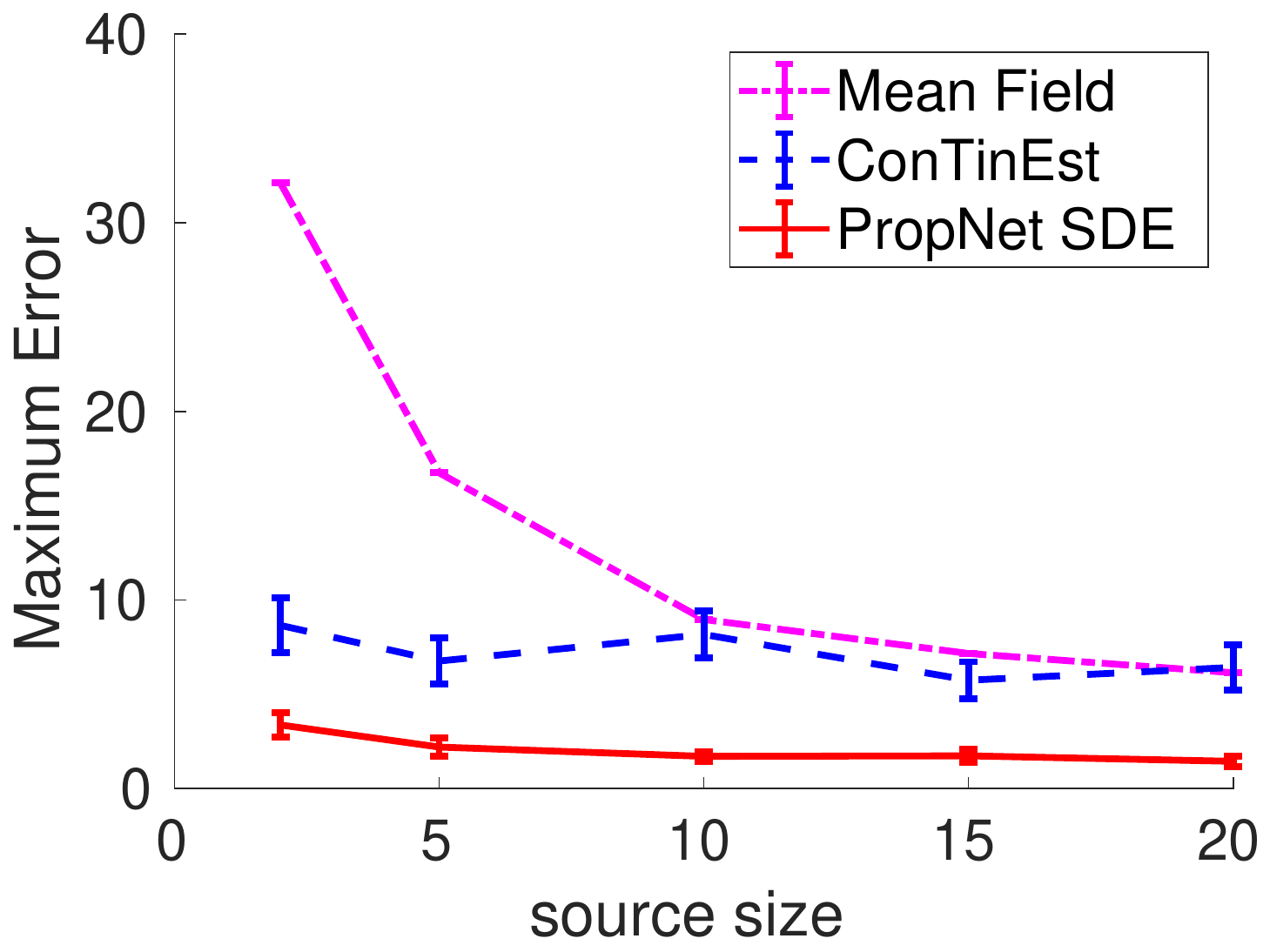} &
\includegraphics[width=0.22\textwidth]{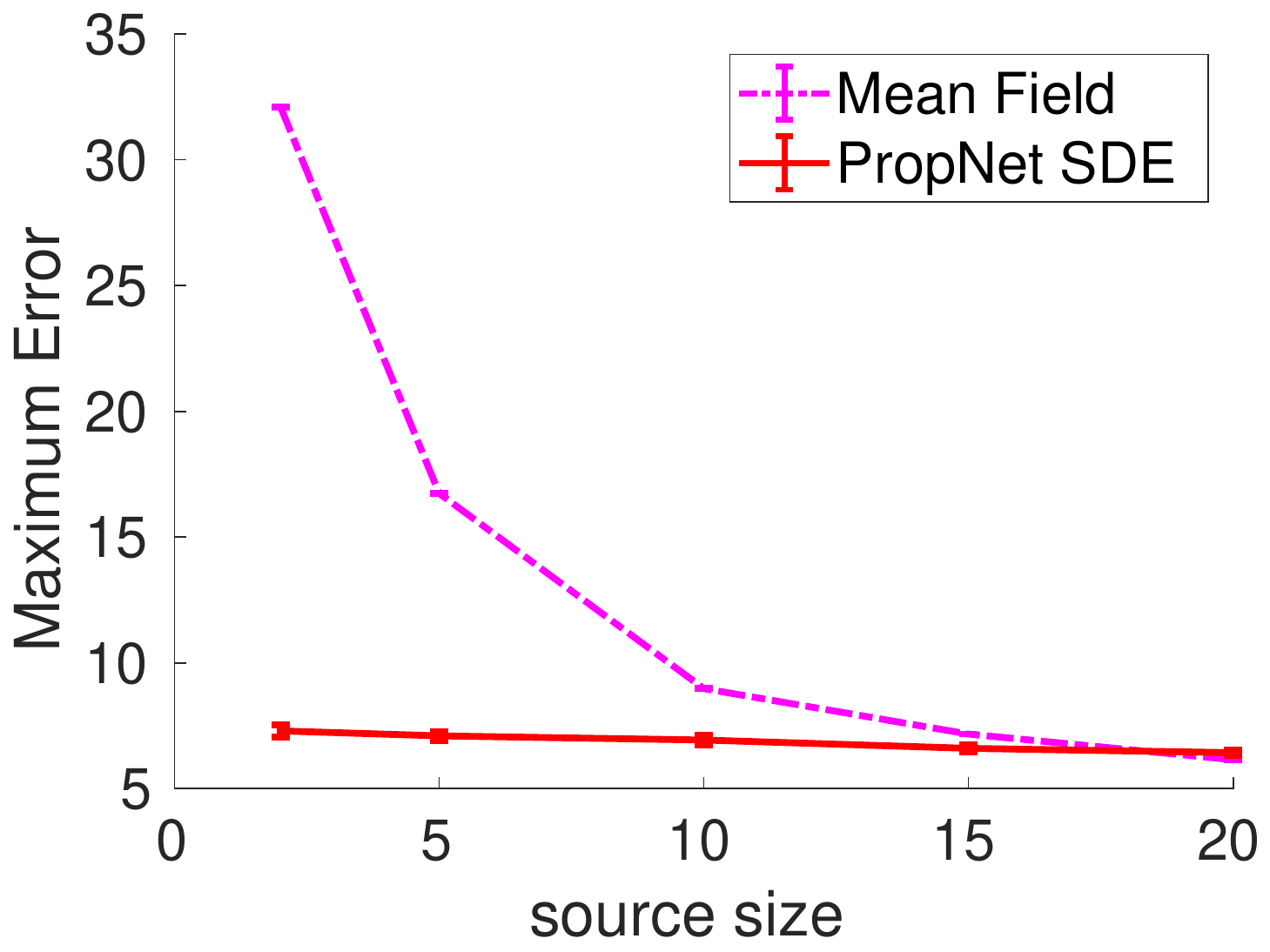} &
\includegraphics[width=0.22\textwidth]{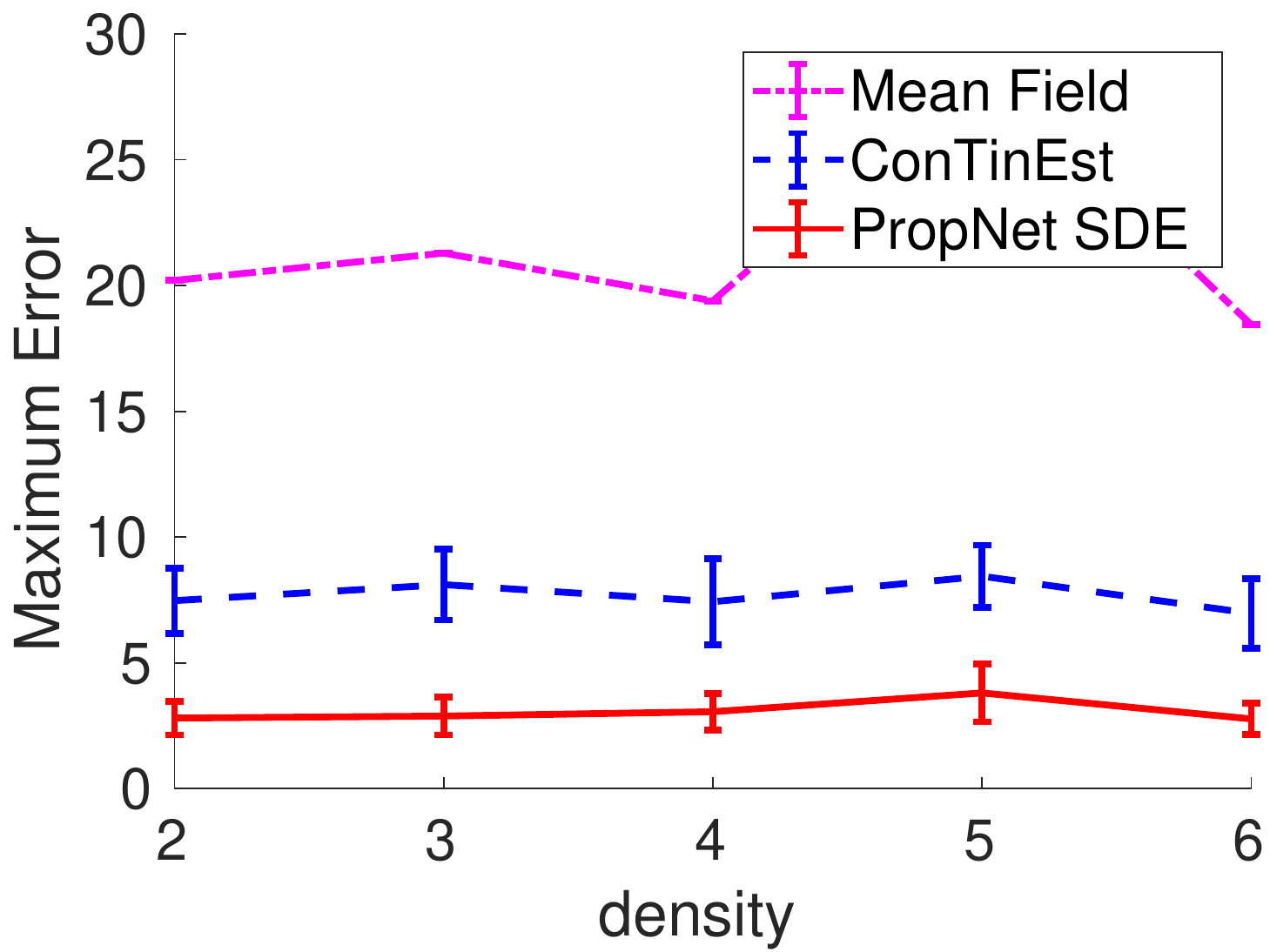} &
\includegraphics[width=0.22\textwidth]{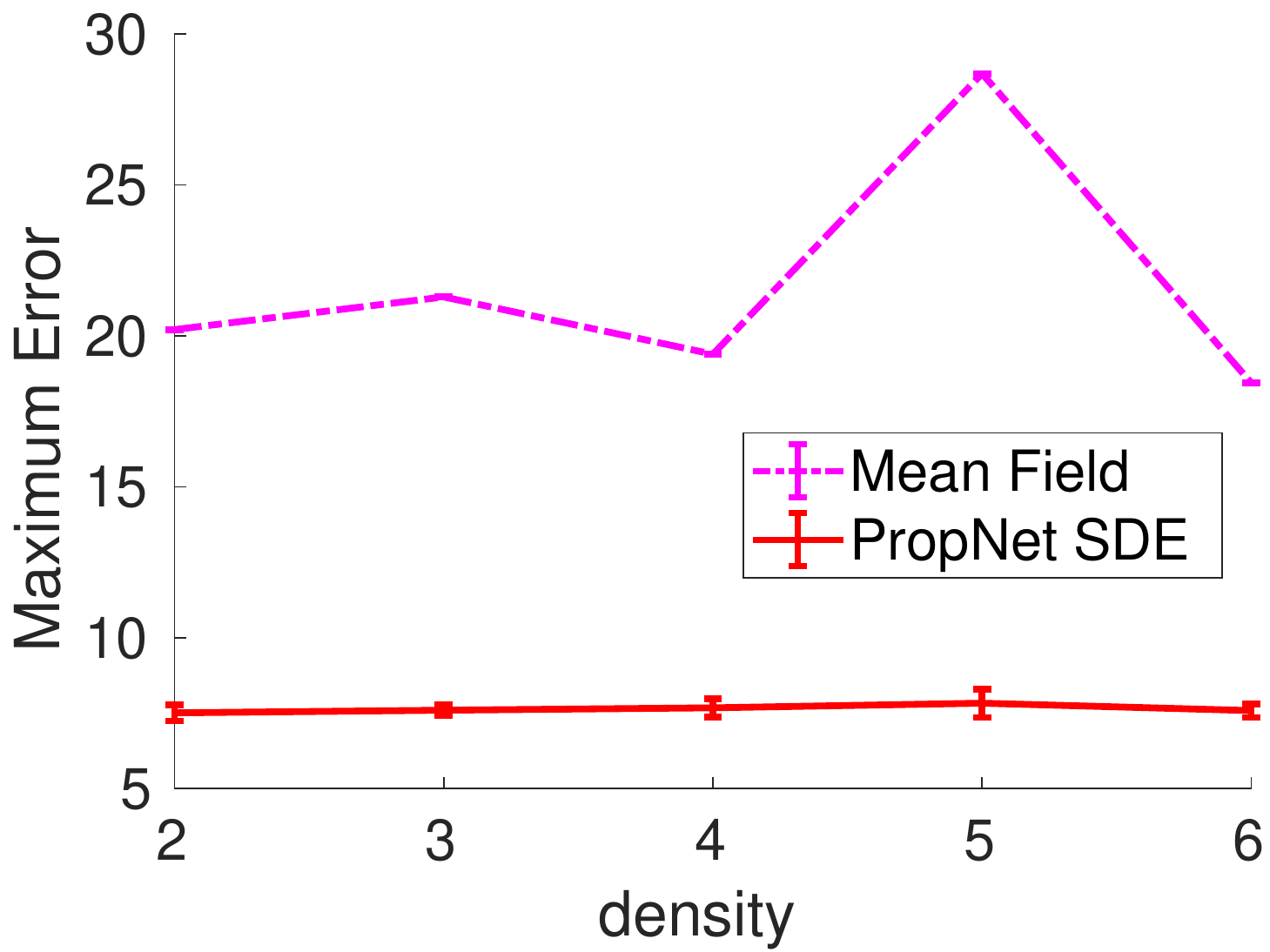} \\
(a) & (b) & (c) & (d)
\end{tabular}
\caption{Robustness test of comparison methods on Erd\H{o}s-R\'{e}nyi network (top row), small-world network (middle row), and scale-free network (bottom row).
\textit{Column (a)}: the maximum absolute error $\max_{0\le t\le T}|\influhat(t)-\influ(t)|/\influ(t)$ versus different sizes of source set using the basic propagation model without recovery;
\textit{Column (b)}: the maximum absolute error versus different sizes of source set using the basic propagation model with recovery;
\textit{Column (c)}: the maximum absolute error versus different density levels using the basic propagation model without recovery;
\textit{Column (d)}: the maximum absolute error versus different density levels using the basic propagation model with recovery.}
\label{fig:influence_density_source}
\end{figure}

In the fourth and the last experiment, we consider the propagation model where activation rates are time-varying instead of constant.
More specifically, $t_{ij}$, the time for node $i$ to activate $j$, follows the Rayleigh distribution (a specific type of Weibull distribution with $\theta=2$).
For ConTinEst, we generate $1,000$ simulations where each one has $15$ least labels.
For PropNet SDE, we use step size $h=0.01$ and $L=1,000$.
The results without recovery scenario are shown in Figure \ref{fig:influence_weibull_SI}.
Mean Field relies on the Markov property of constant activation rates and hence is not capable to handle this situation.
For ConTinEst and PropNet SDE, we show their estimated influence $\influhat(t)$ and the corresponding relative error $|\influhat(t)-\influ(t)|/\influ(t)$ in the first and second columns of Figure \ref{fig:influence_weibull_SI}, respectively.
In particular, we observe that both methods are accurate, whereas the proposed PropNet SDE tends to yield smaller error in middle to late stages of the propagation.
In addition, we also show the relative accumulated error in individual activation probability, $\sum_i|\xhat_i(t)-x_i(t)|/\sum_i x_i(t)$, of the proposed PropNet SDE method in the right column of Figure \ref{fig:influence_weibull_SI}.
The influence ratio $\influ(t)/n$ is again plotted in black dotted line for reference.
As we can see, PropNet SDE attains very small relative error which indicates that all $x_i(t)$ are estimated accurately.
Note that ConTinEst is not capable to compute such estimations. 
The results are shown in Figure \ref{fig:influence_weibull_SIS}.
We again observe accurate prediction of both total influence $\influ(t)$ and individual activation probability in the middle and right columns of Figure \ref{fig:influence_weibull_SIS}.
\begin{figure}[t!]
\centering
\includegraphics[width=0.3\textwidth]{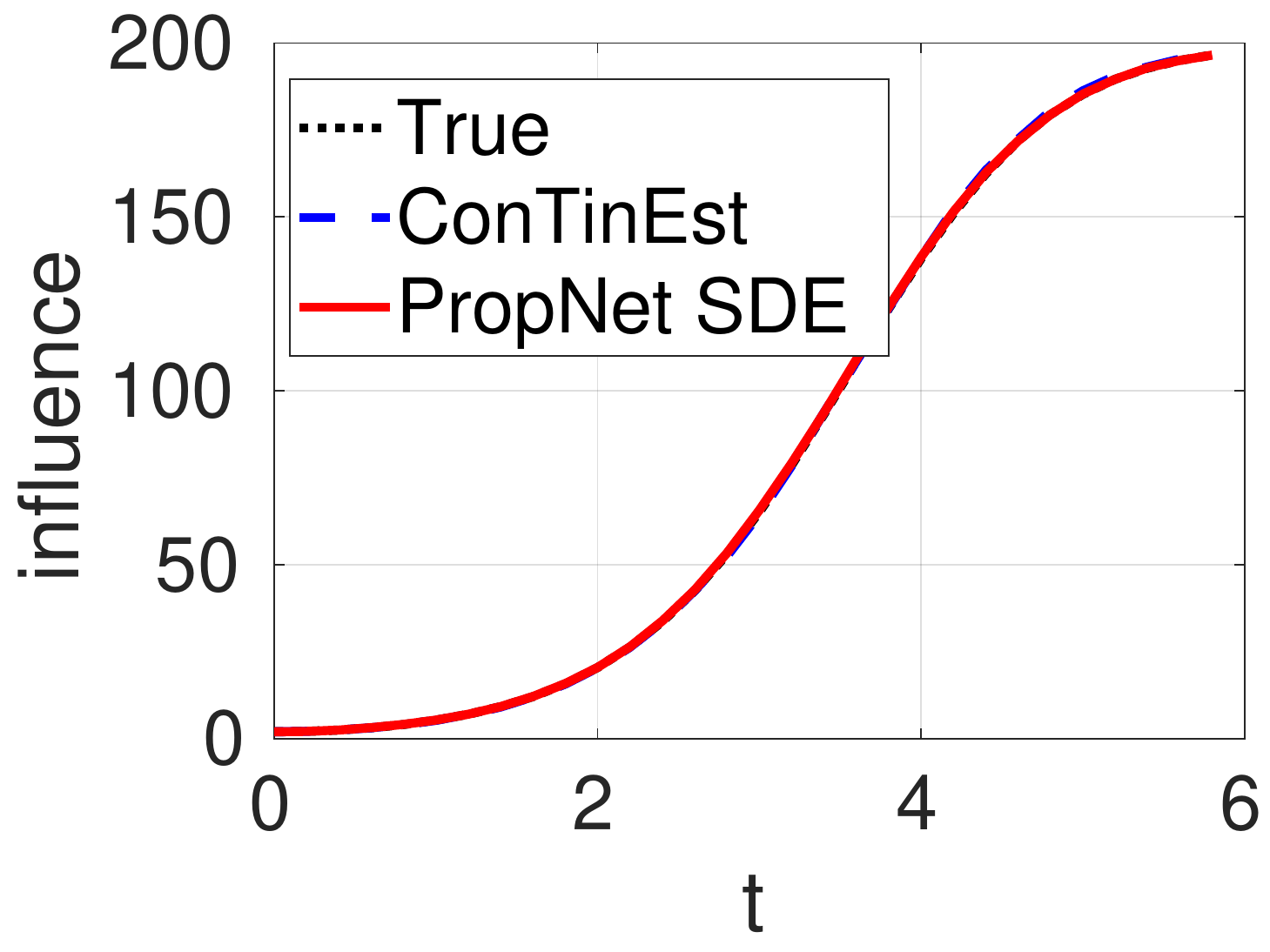}
\includegraphics[width=0.3\textwidth]{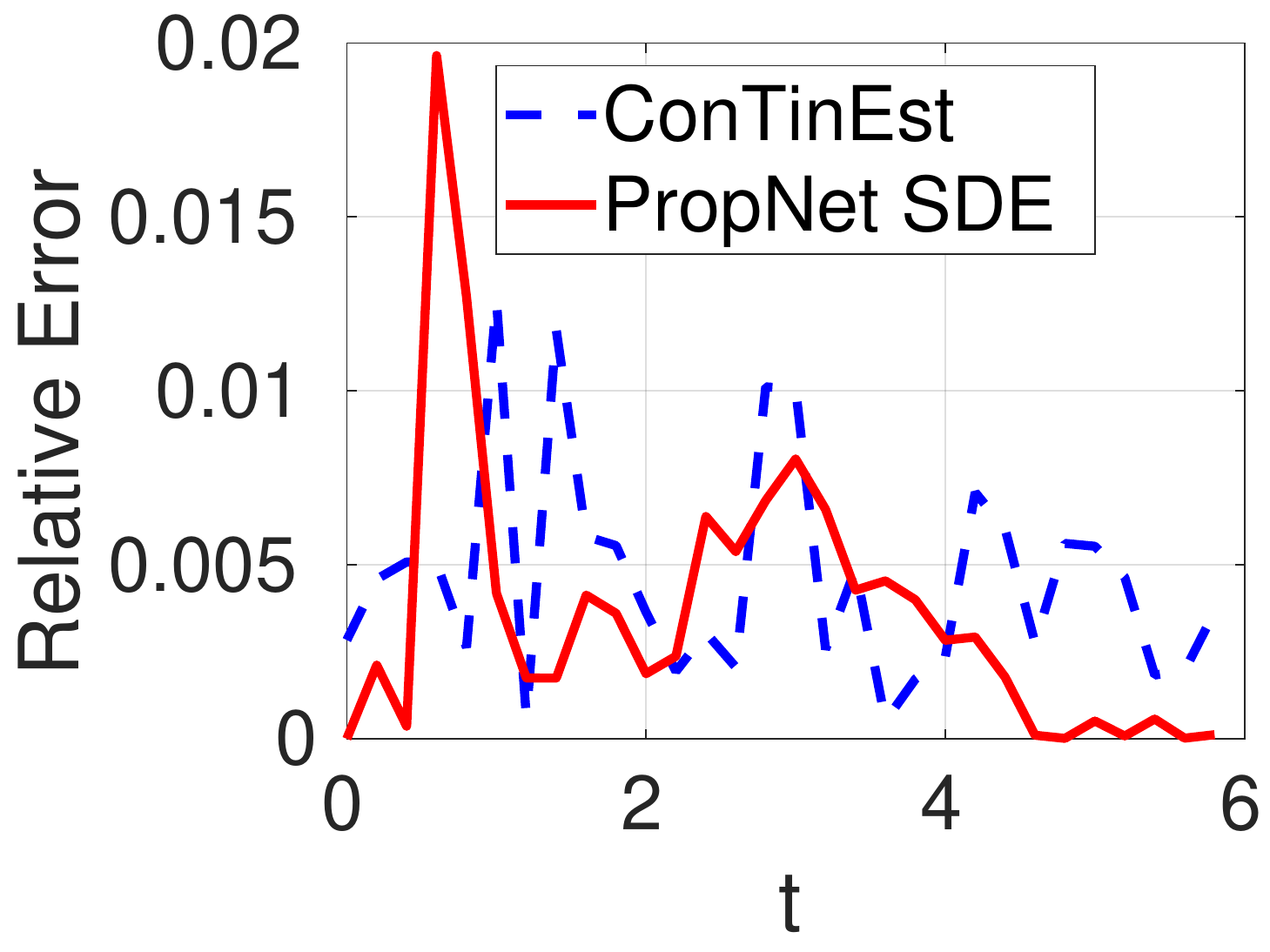}
\includegraphics[width=0.3\textwidth]{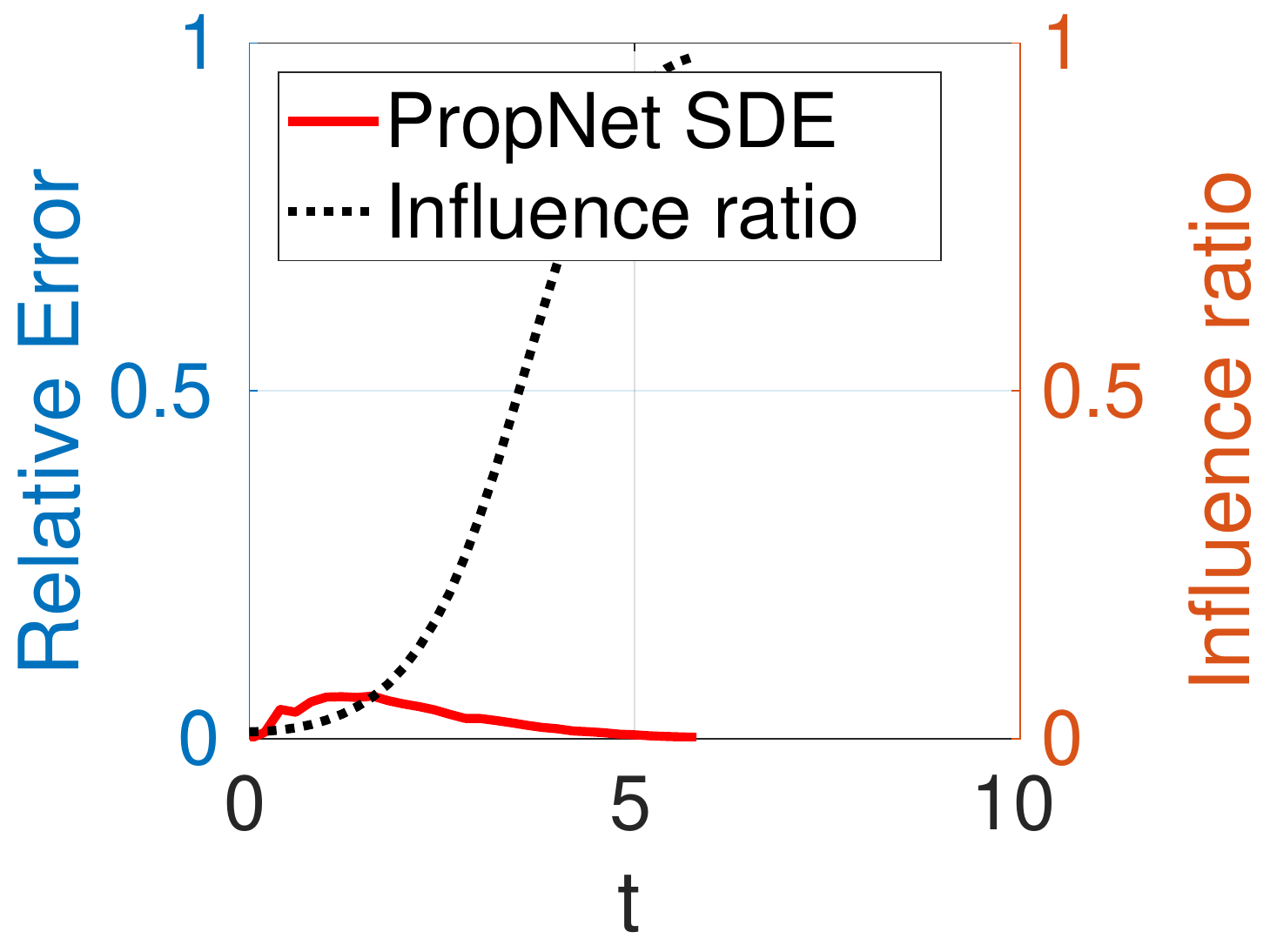}\\
\includegraphics[width=0.3\textwidth]{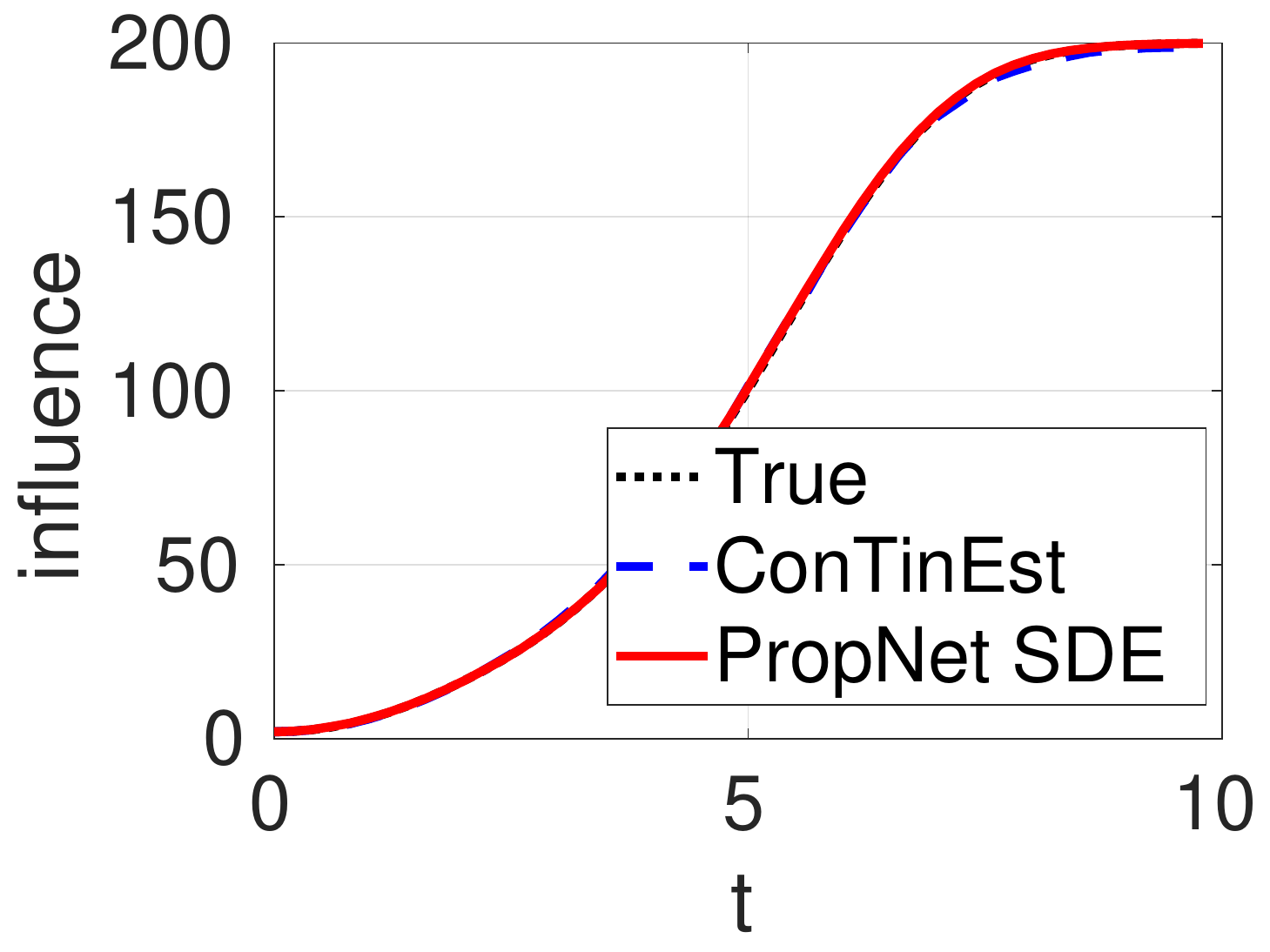}
\includegraphics[width=0.3\textwidth]{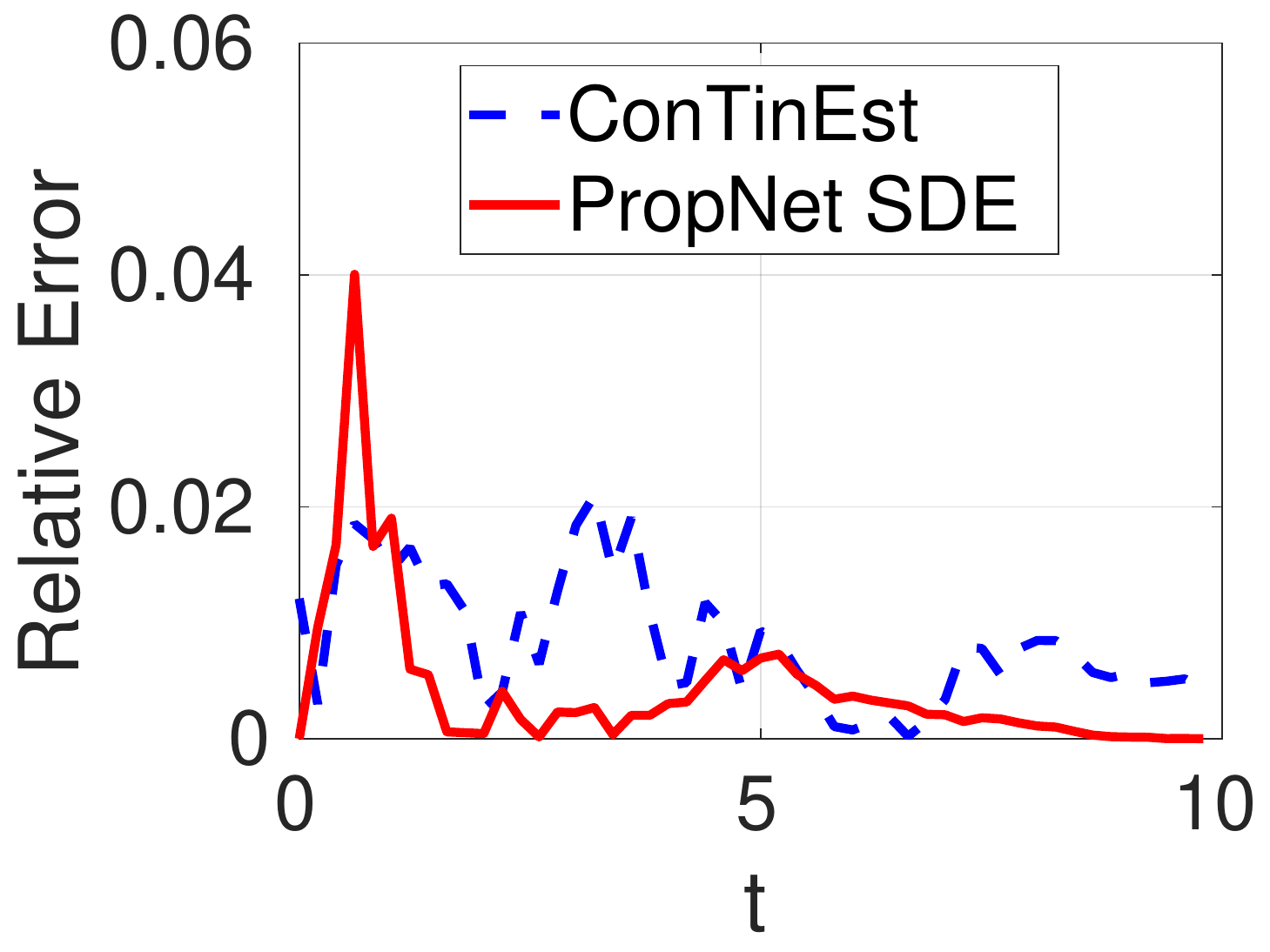}
\includegraphics[width=0.3\textwidth]{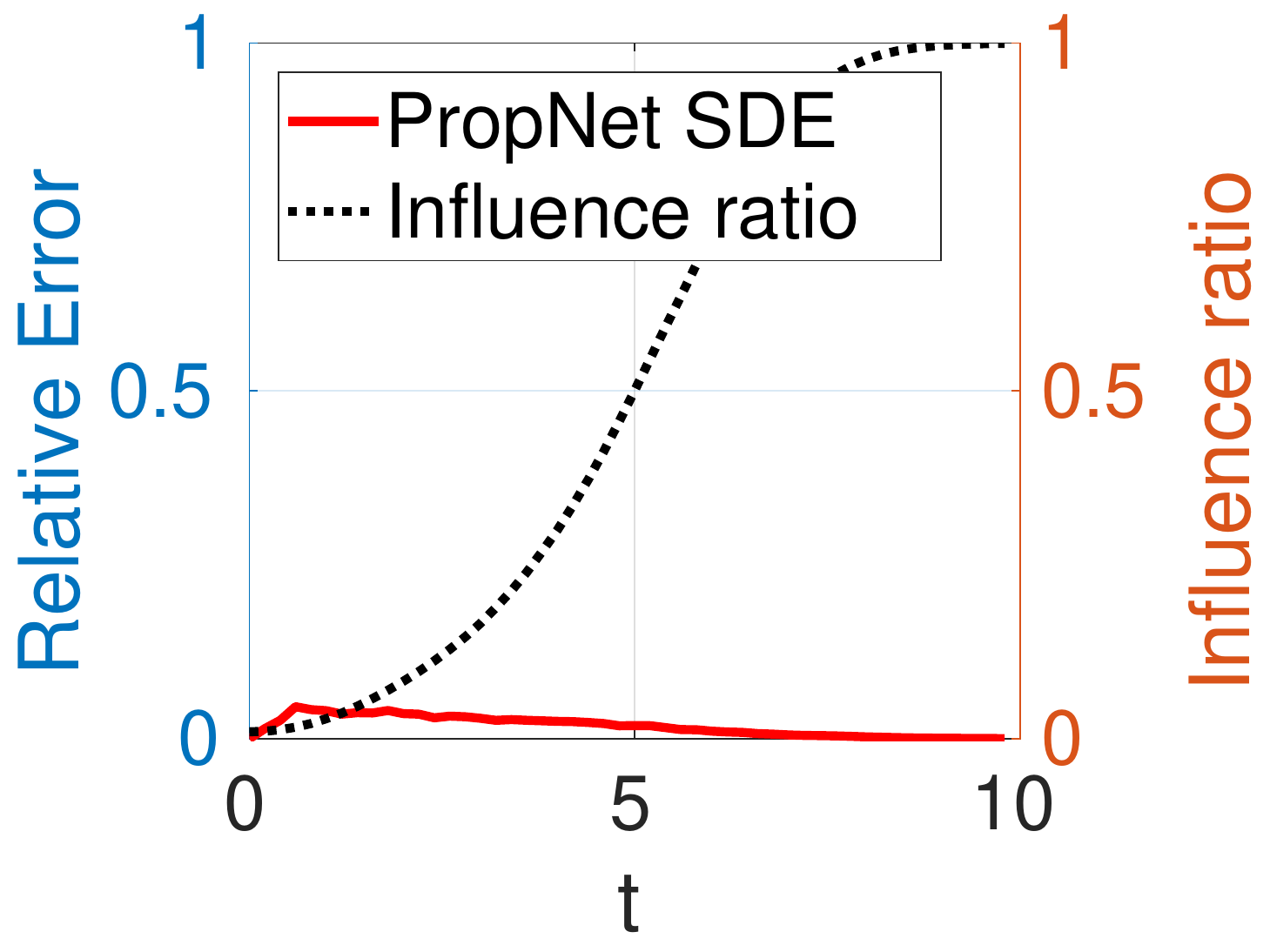}\\
\includegraphics[width=0.3\textwidth]{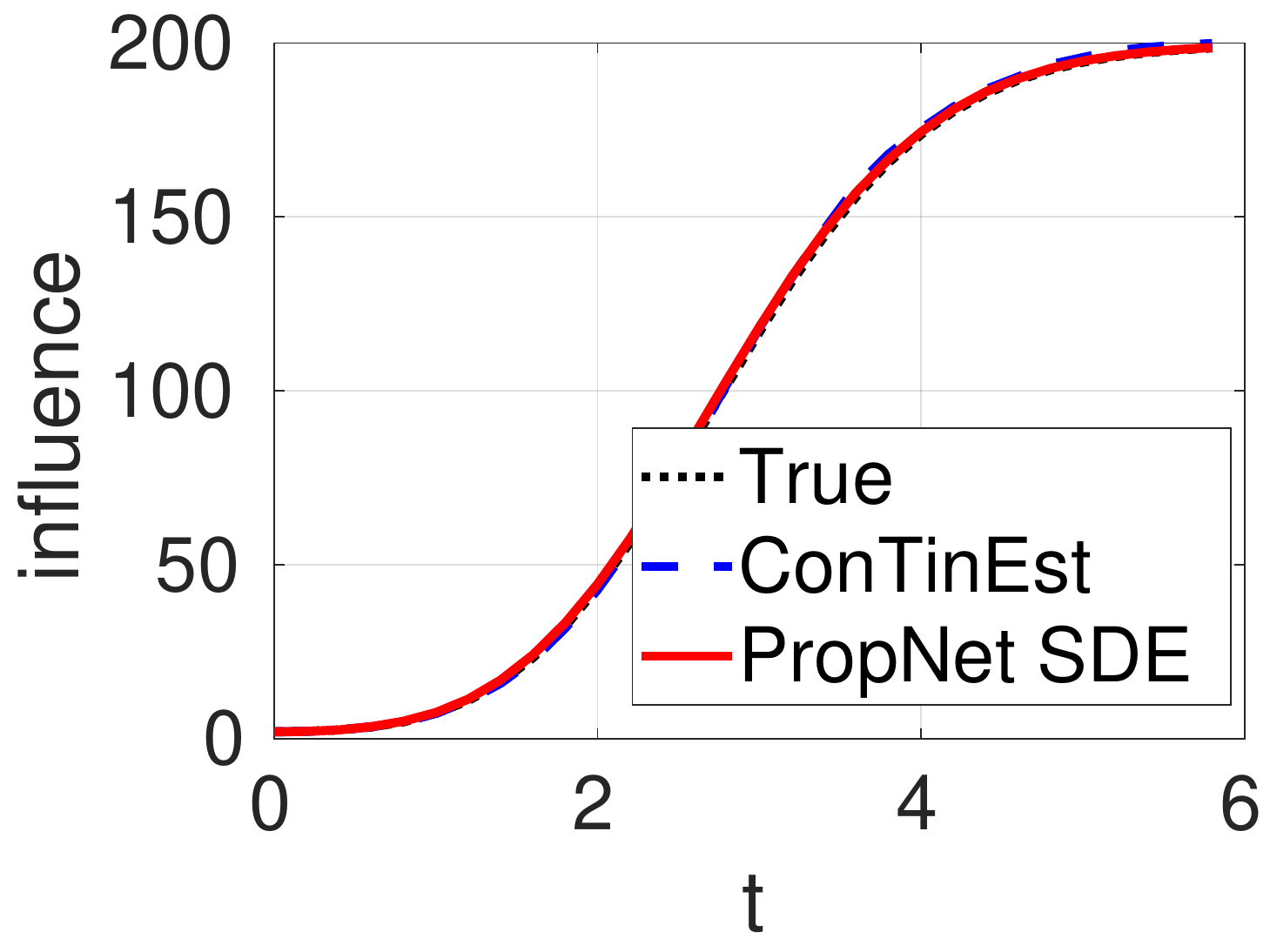}
\includegraphics[width=0.3\textwidth]{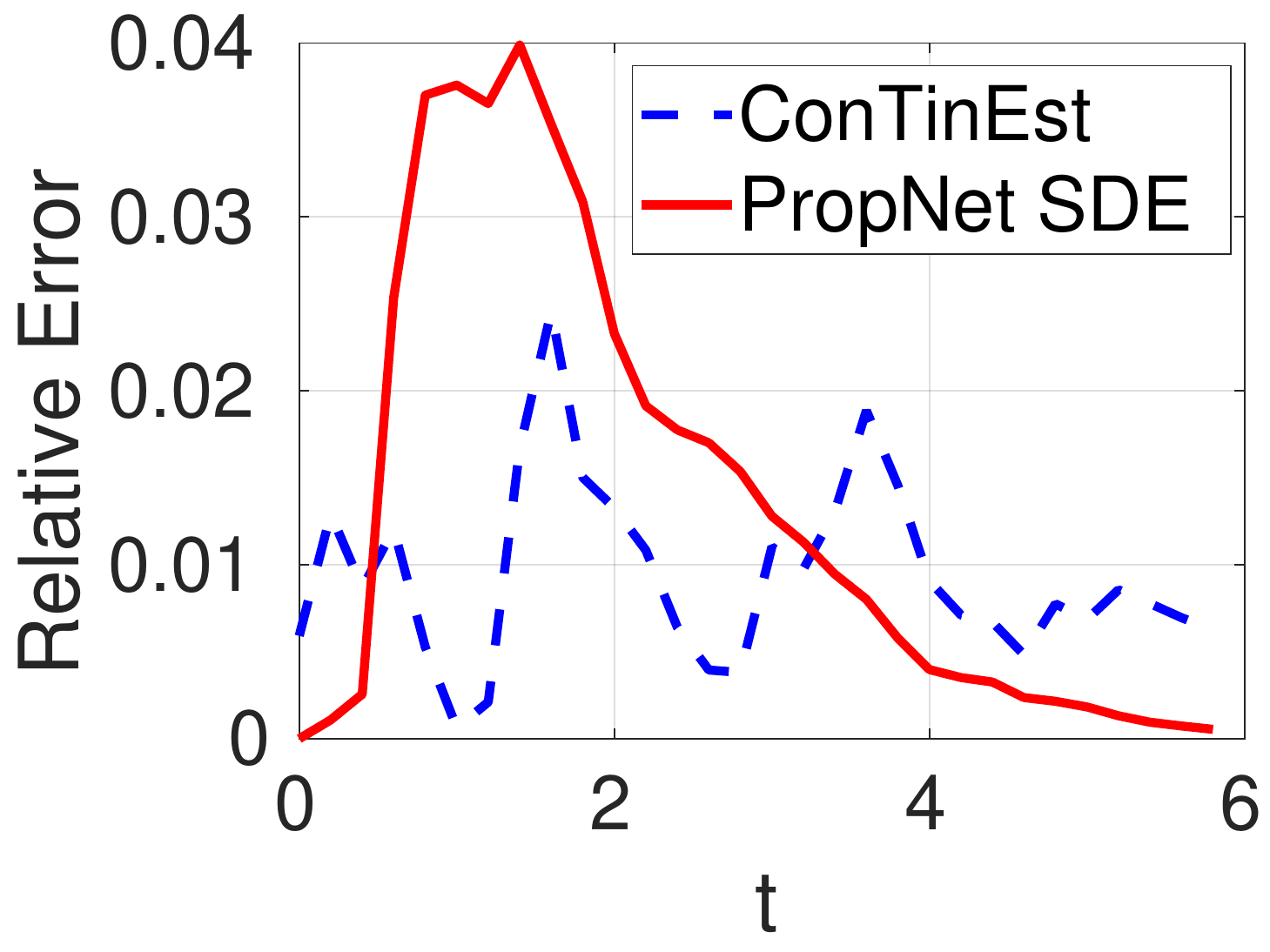}
\includegraphics[width=0.3\textwidth]{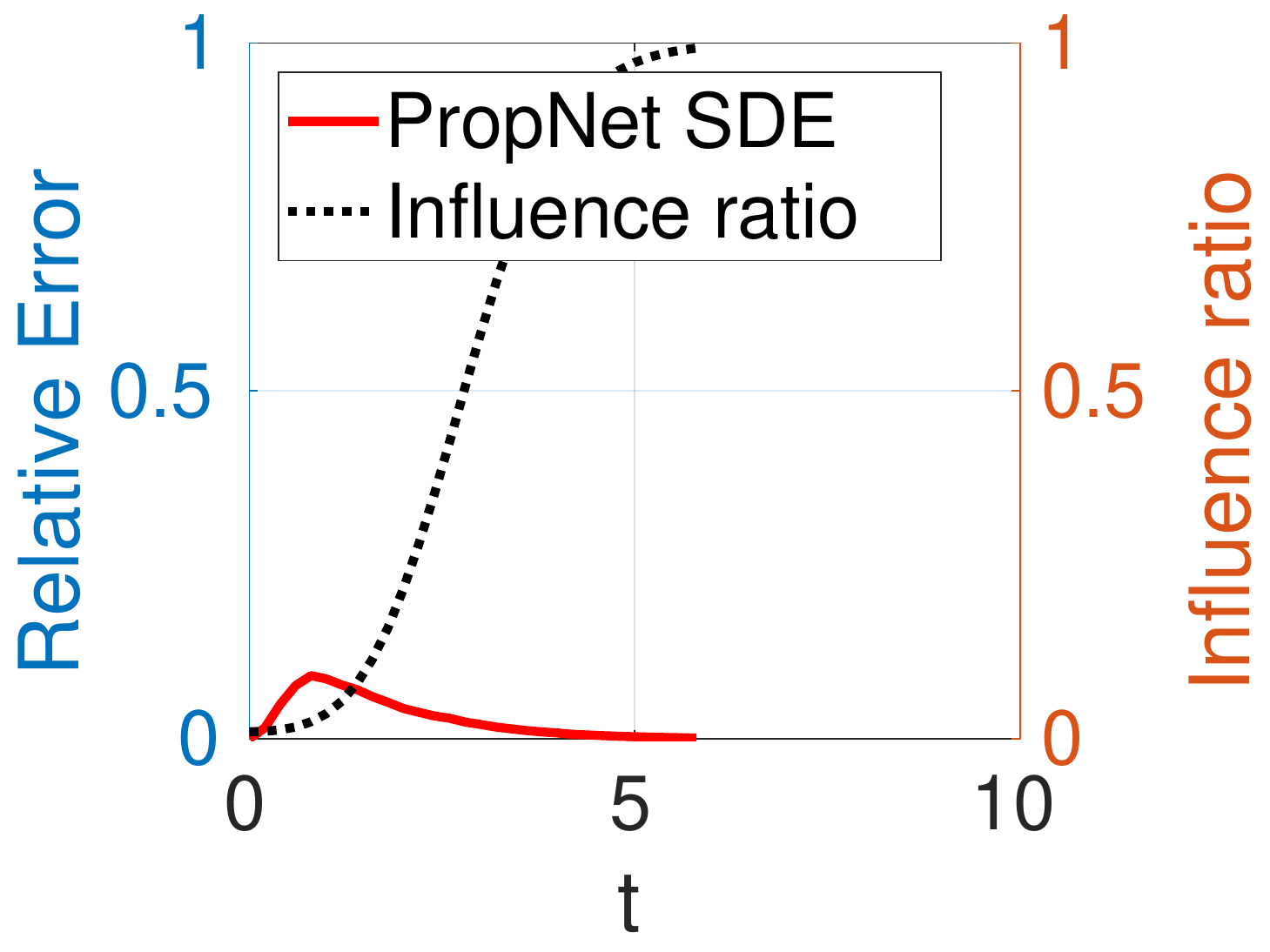}
\caption{Influence prediction by comparison methods on Erd\H{o}s-R\'{e}nyi network (top row), small-world network (middle row), and scale-free network (bottom row), all of size $n=200$, using the propagation model with Rayleigh distributed activation times \underline{without} recovery. 
\textit{Left column}: True total influence $\influ(t)=\sum_i x_i(t)$ and influences $\influhat(t)=\sum_i \xhat_i(x)$ obtained by ConTinEst and PropNet SDE.
\textit{Middle column}: relative error in influence $|\influhat(t)-\influ(t)|/\influ(t)$.
\textit{Right column}: relative error in individual activation probability $\sum_i|\xhat_i(t)-x_i(t)|/\sum_i x_i(t)$ (influence ratio $\influ(t)/n$ is plotted in black dotted line for reference).}
\label{fig:influence_weibull_SI}
\end{figure}
\begin{figure}[t!]
\centering
\includegraphics[width=0.3\textwidth]{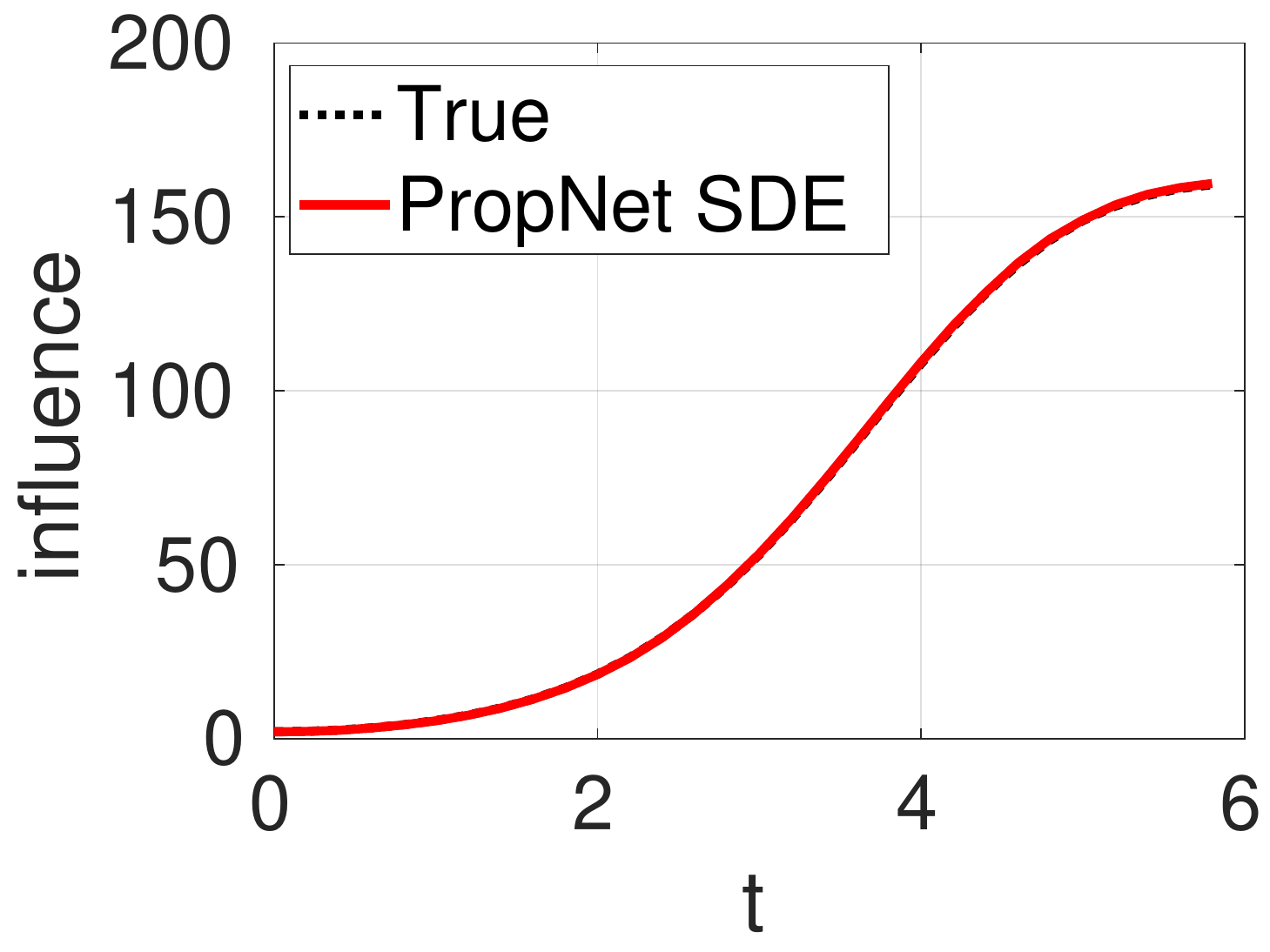}
\includegraphics[width=0.3\textwidth]{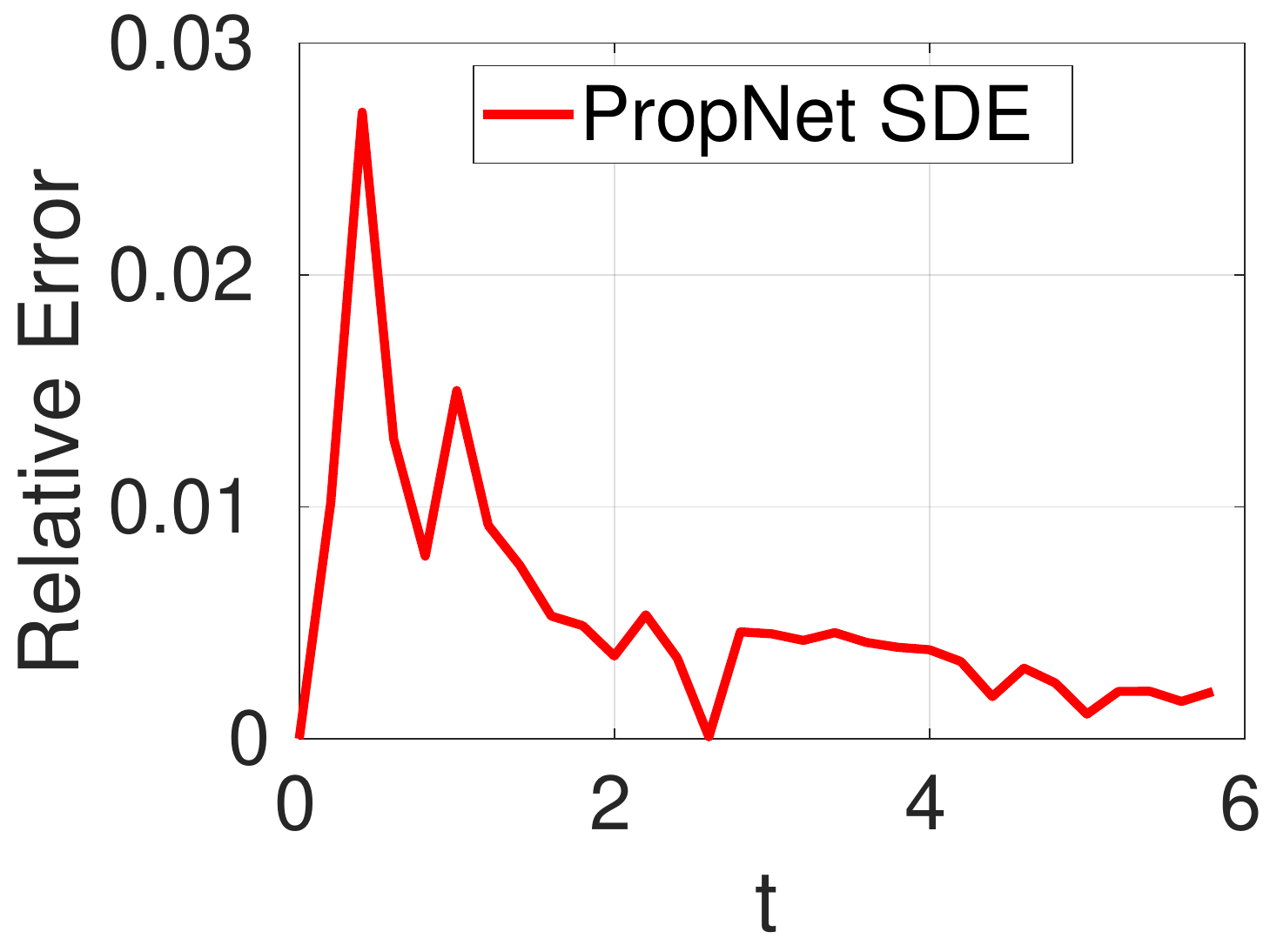}
\includegraphics[width=0.3\textwidth]{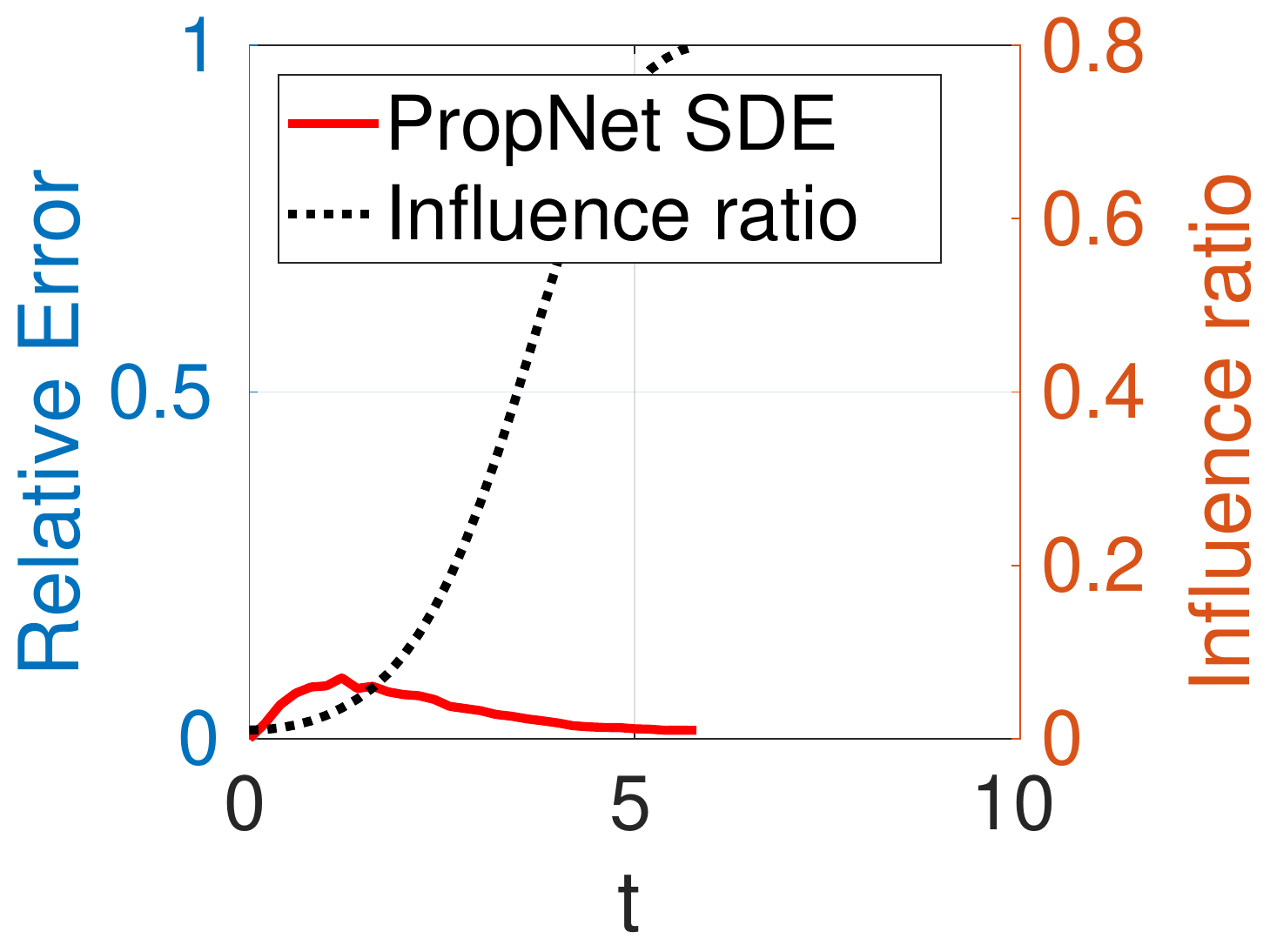}\\
\includegraphics[width=0.3\textwidth]{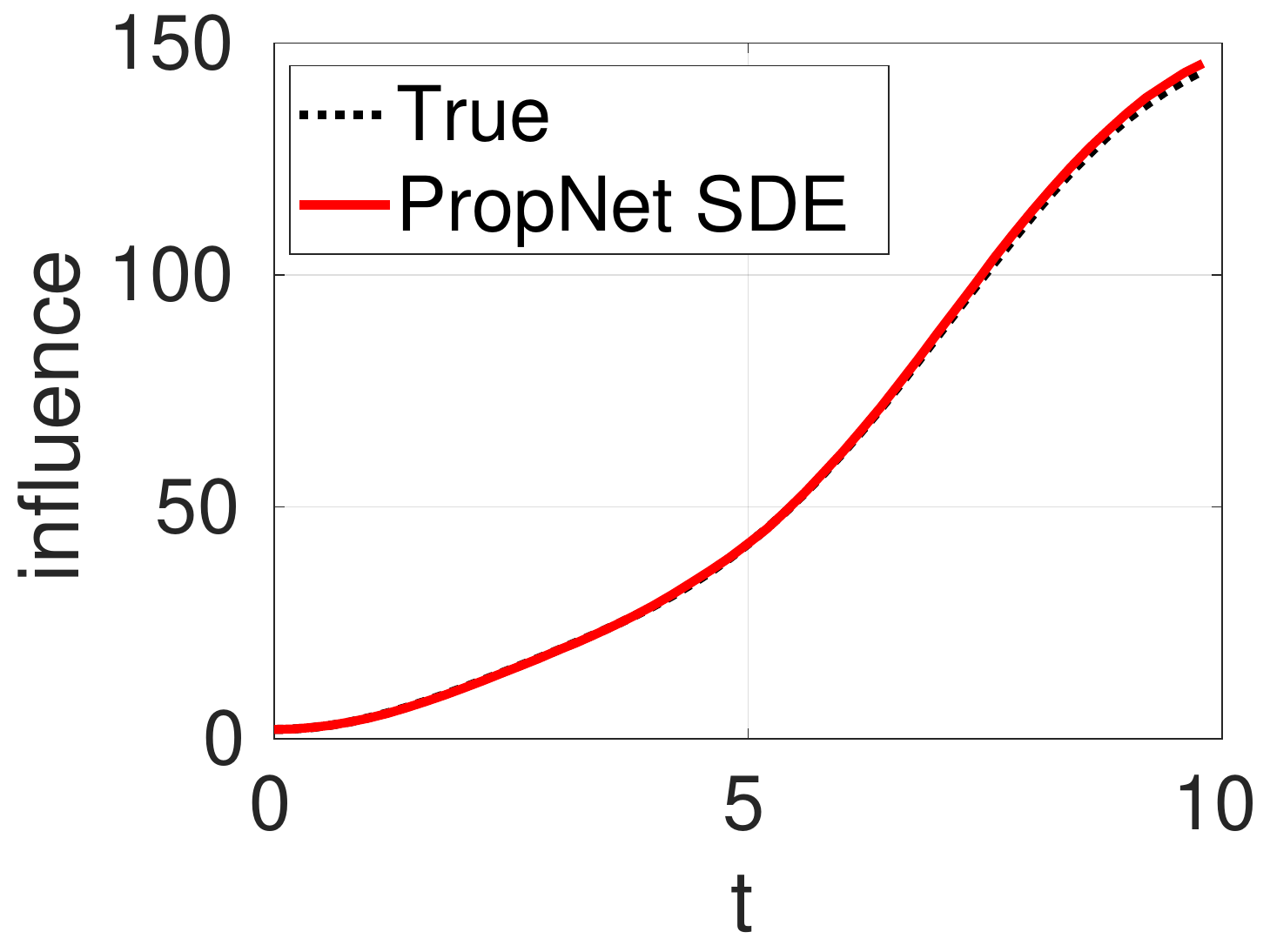}
\includegraphics[width=0.3\textwidth]{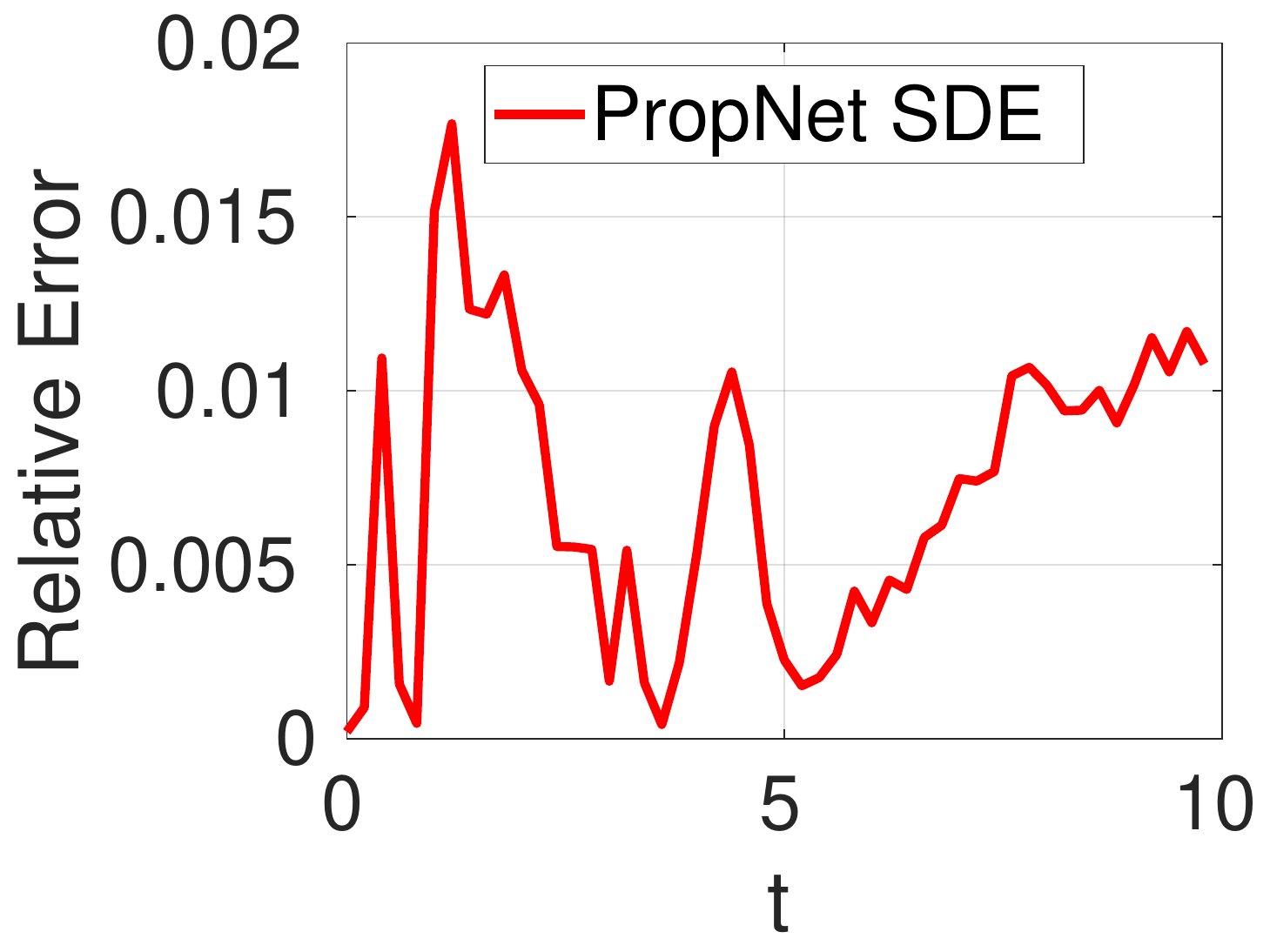}
\includegraphics[width=0.3\textwidth]{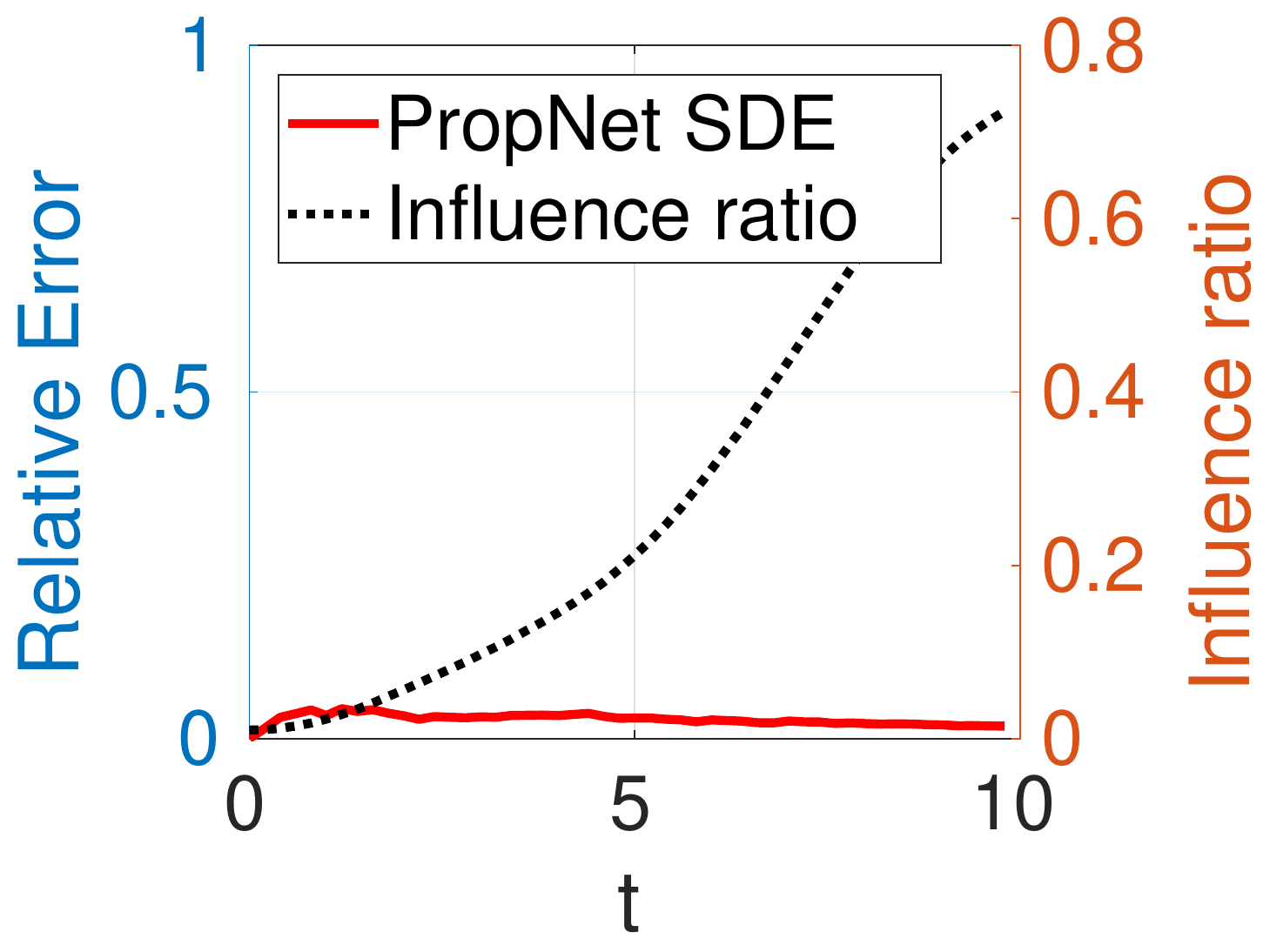}\\\includegraphics[width=0.3\textwidth]{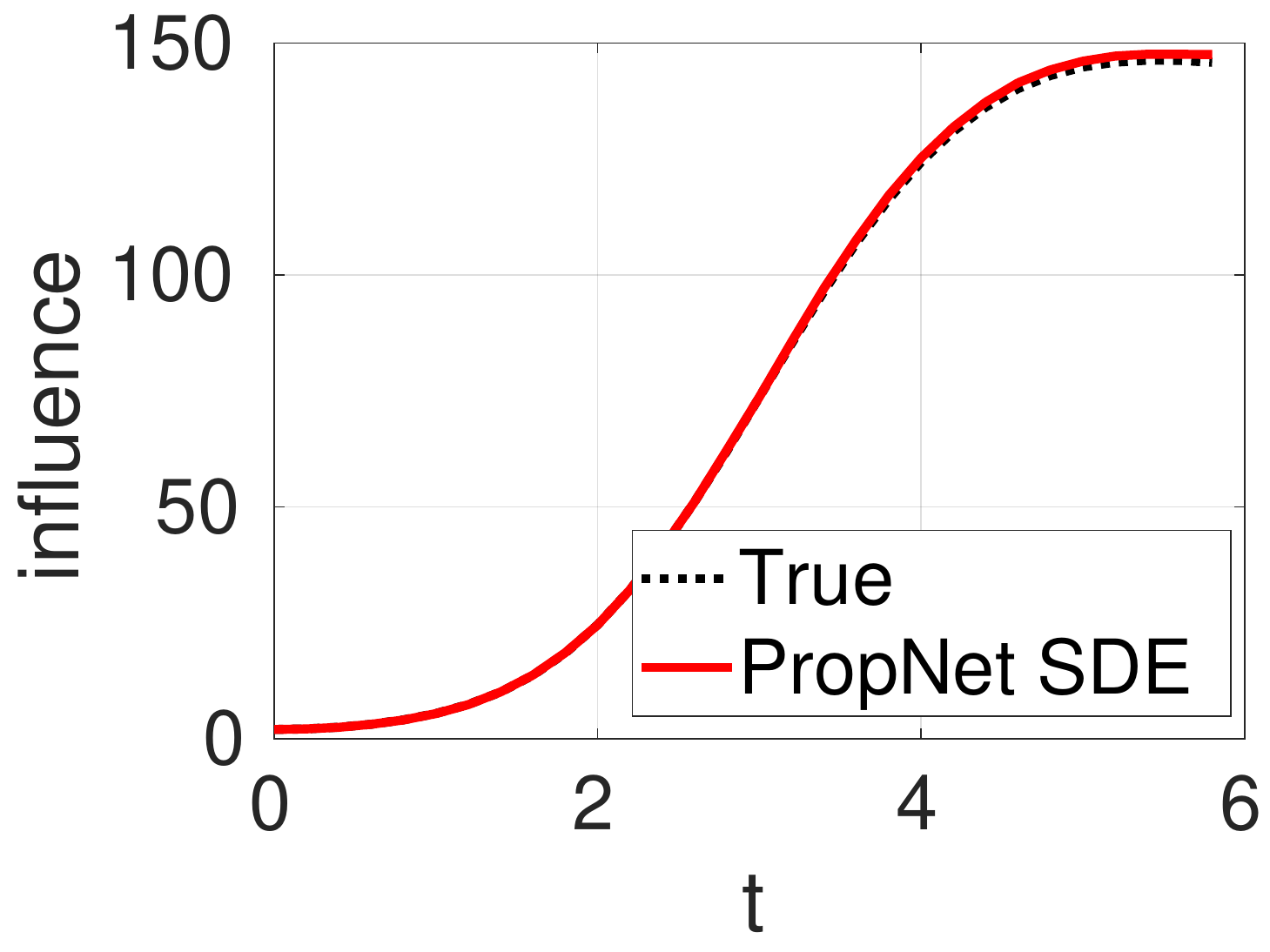}
\includegraphics[width=0.3\textwidth]{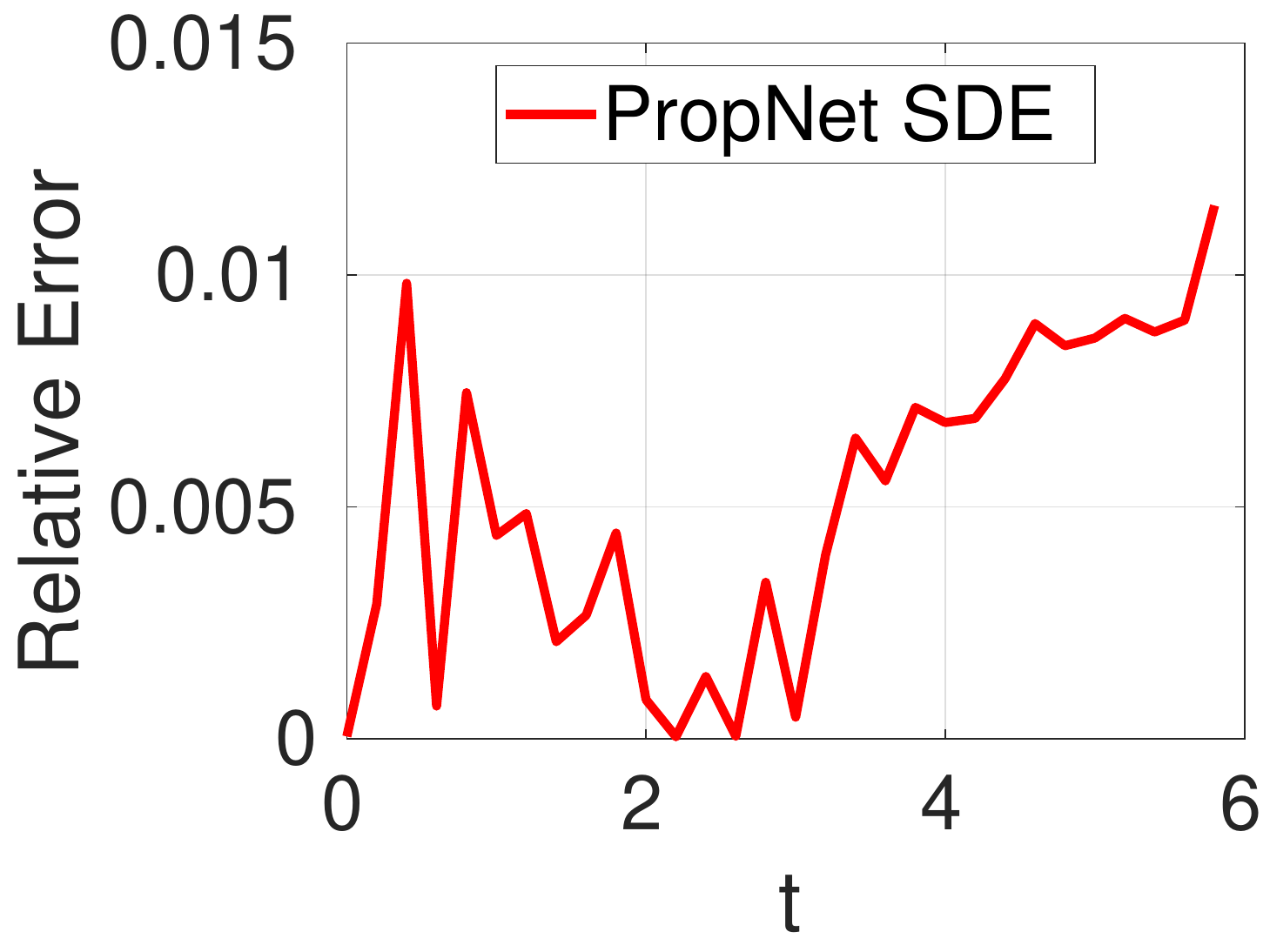}
\includegraphics[width=0.3\textwidth]{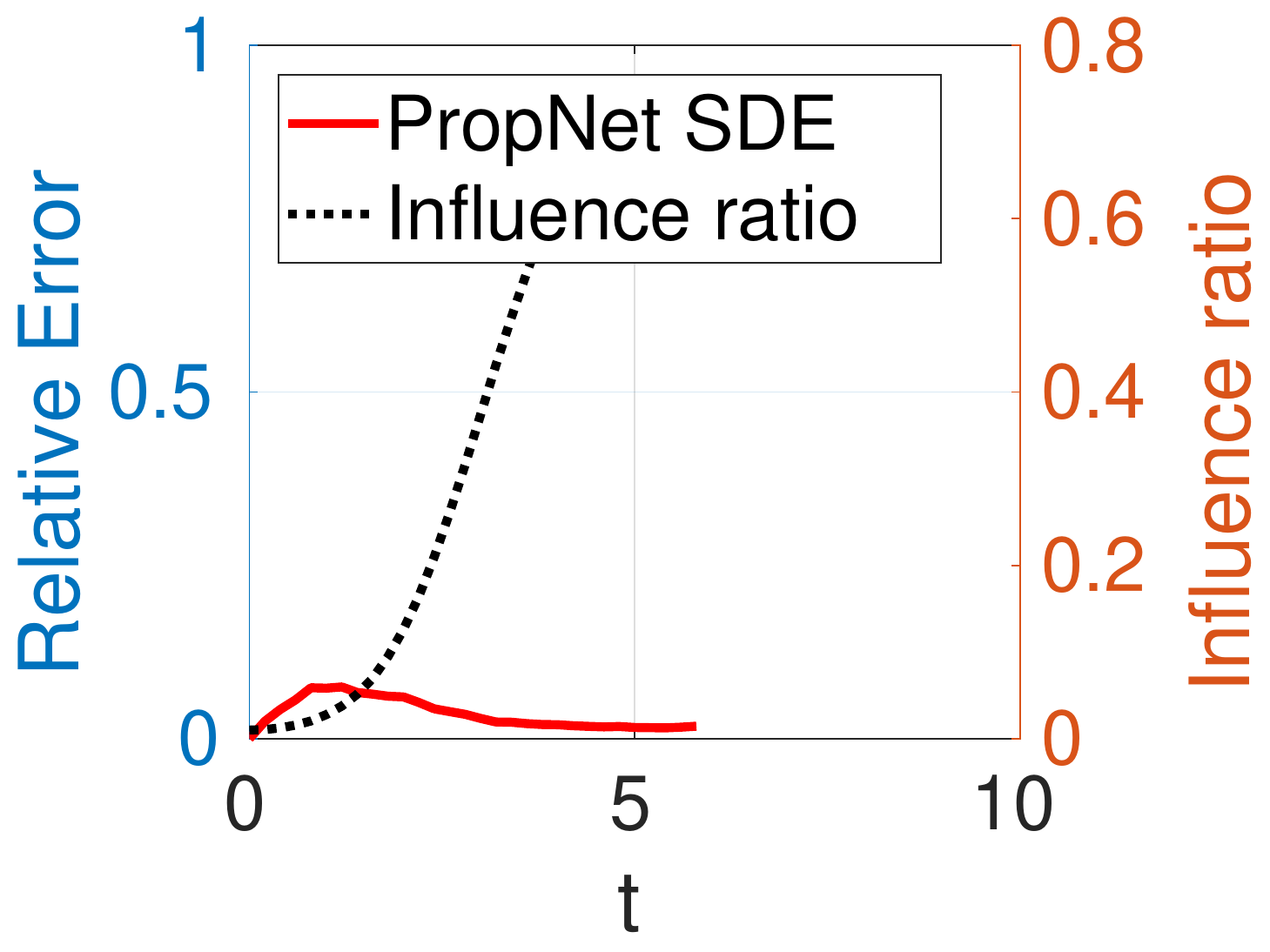}
\caption{Influence prediction by comparison methods on Erd\H{o}s-R\'{e}nyi network (top row), small-world network (middle row), and scale-free network (bottom row), all of size $n=200$, using the propagation model with Rayleigh distributed activation times \underline{with} recovery. 
\textit{Left column}: True total influence $\influ(t)=\sum_i x_i(t)$ and influences $\influhat(t)=\sum_i \xhat_i(x)$ obtained by ConTinEst and PropNet SDE.
\textit{Middle column}: relative error in influence $|\influhat(t)-\influ(t)|/\influ(t)$.
\textit{Right column}: relative error in individual activation probability $\sum_i|\xhat_i(t)-x_i(t)|/\sum_i x_i(t)$ (influence ratio $\influ(t)/n$ is plotted in black dotted line for reference).}
\label{fig:influence_weibull_SIS}
\end{figure}

The proposed PropNet SDE algorithm benefits from the variance reduction (VR) technique, which significantly reduces the number of samplings to achieve the same level of accuracy.
To demonstrate the effectiveness of variance reduction, we run PropNet SDE with standard sampling of Poisson numbers and VR (i.e., Algorithm \ref{alg:jsde}) using Erd\H{o}s-R\'{e}nyi network.
The parameter setting of Erd\H{o}s-R\'{e}nyi network is the same as above, but we run the algorithms with 5, 10, 20, 40, 80 and 100 sampled cascades to track the prediction errors. 
We plot both prediction errors with 95\% confidence intervals in Figure \ref{fig:vr}.
From the plots in Figure \ref{fig:vr}, we can see that PropNet SDE with VR produces lower prediction error than that without VR for the same amount of samplings.
This suggests that VR is a simple and effective implementation in the proposed PropNet SDE method to improve computational efficiency.
\begin{figure}[t!]
\centering
\includegraphics[width=0.35\textwidth]{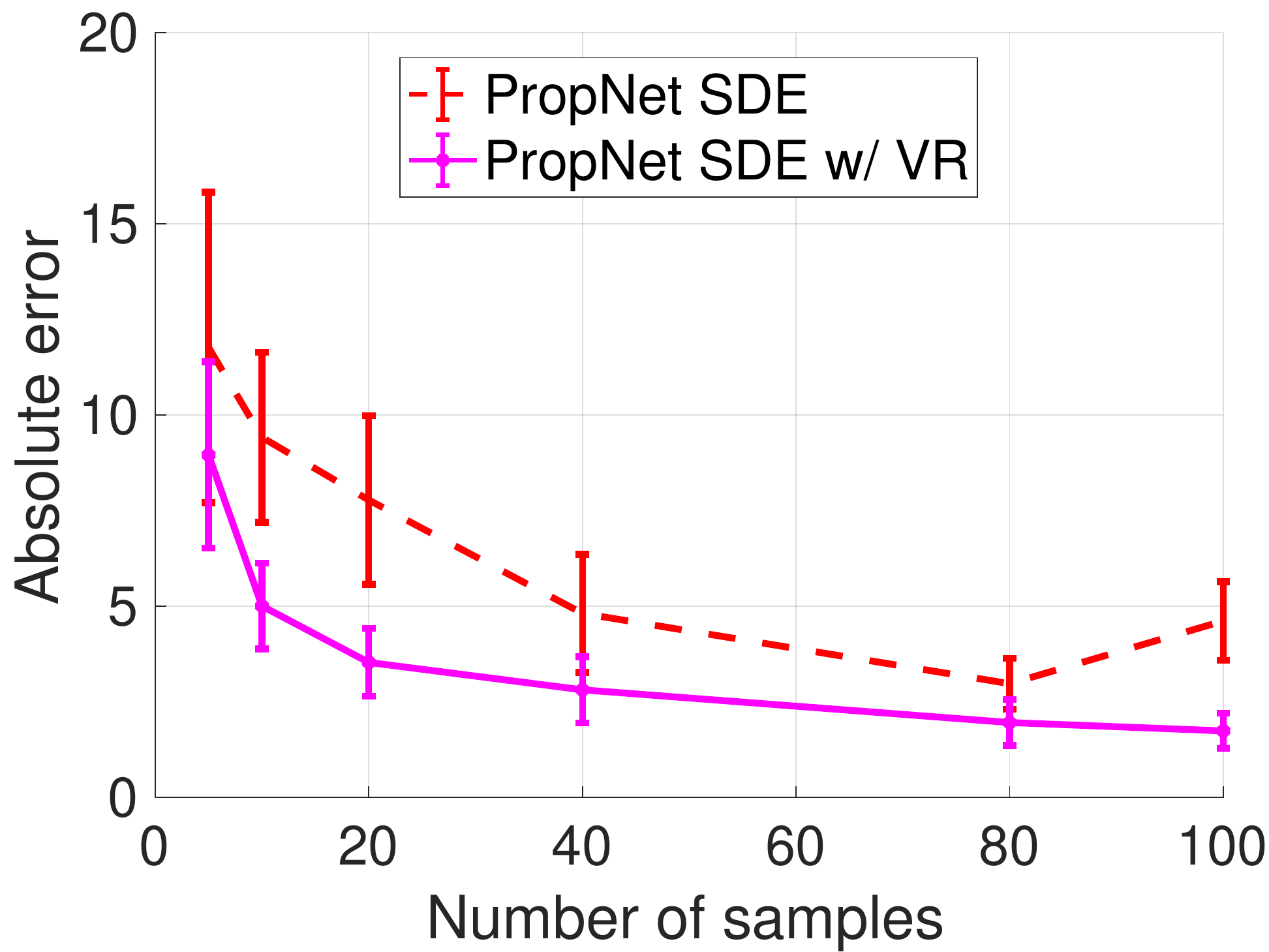}
\includegraphics[width=0.35\textwidth]{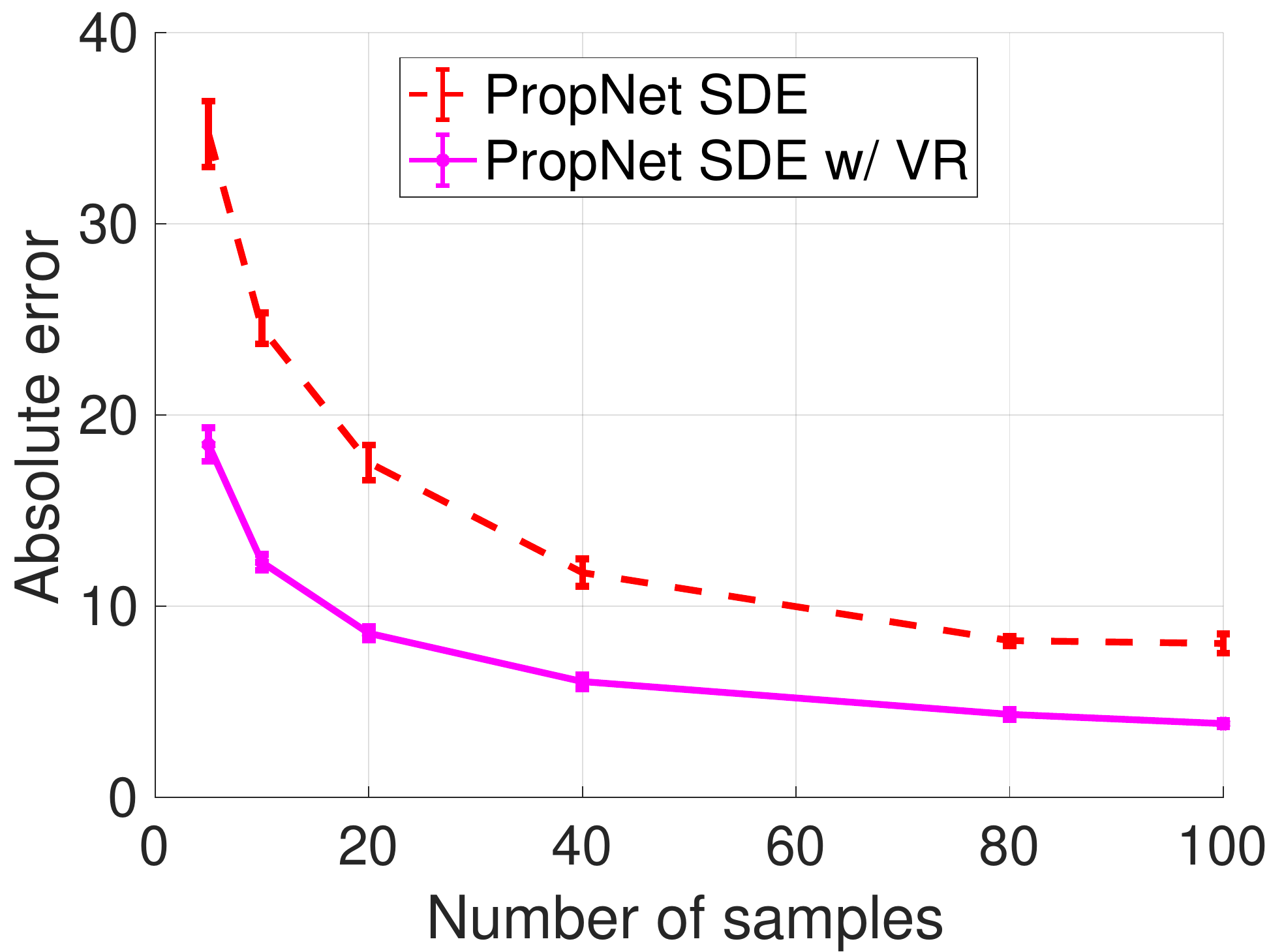}
\caption{Comparison of prediction errors in influence (left) and individual activation probability (right) versus the number of sampled cascades $L$ using the PropNet SDE without and with variance reduction (VR) for the Erd\H{o}s-R\'{e}nyi's network in the first experiment.}
\label{fig:vr}
\end{figure}

\subsection{Experiments on real data}
We also test the proposed influence prediction algorithm on the Weibo-Net-Tweet dataset\footnote{https://cn.aminer.org/data-sna\#Weibo-Net-Tweet}. The Weibo-Net-Tweet dataset contains 1.7 million users and 400 million edges (followee-follower relationships). 
This dataset also contains 300,000 popular microblog diffusion episodes (propagation cascades) posted in 2012. 
Each diffusion episode consists of the original microblog and its retweets.
We select the most influential 434 users and retrieve all the propagation cascades that contain at least 50 of these users.
Then we randomly select 60\% of these cascades, and apply the NetRate algorithm \cite{Gomez-Rodriguez:2012b} to learn the activation rate matrix $A$.
With the learned $A$, we apply the comparison algorithms to the source sets of the remaining 40\% cascades to predict influence, and compute the error of the predicted influence to the true influence in the first 24 hours exhibited by these cascades.
The results are shown in Figure \ref{fig:real}.
From these plots, we observe Mean Field generates severely large errors than ConTinEst and PropNet SDE, where the latter two have comparable accuracy in influence prediction, as shown in the left panel of Figure \ref{fig:real}.
In addition, PropNet SDE also accurately predicted the activation probabilities of individual users as shown in the right panel of Figure \ref{fig:real}.
\begin{figure}[t!]
\centering
\includegraphics[width=0.35\textwidth]{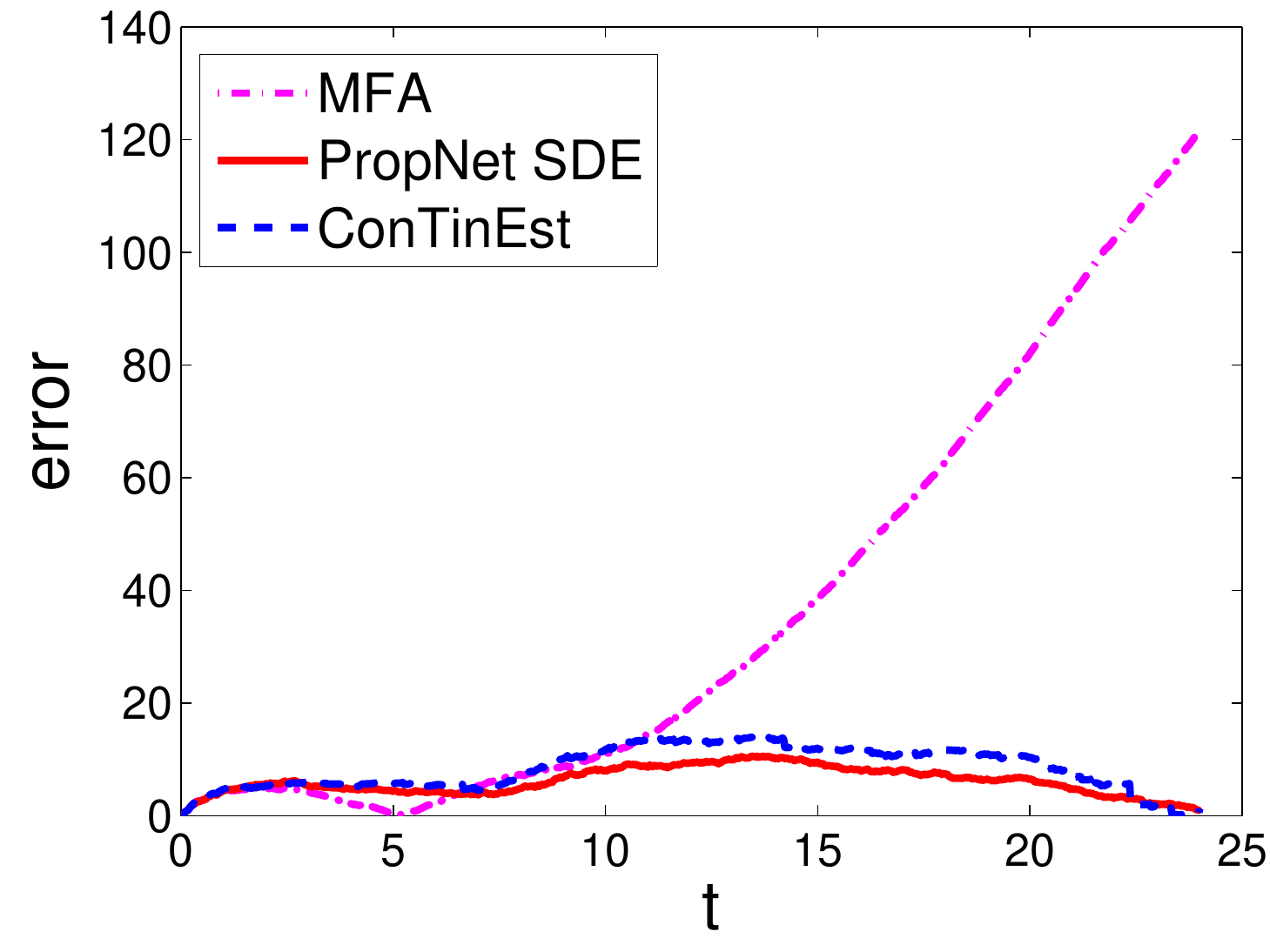}
\includegraphics[width=0.35\textwidth]{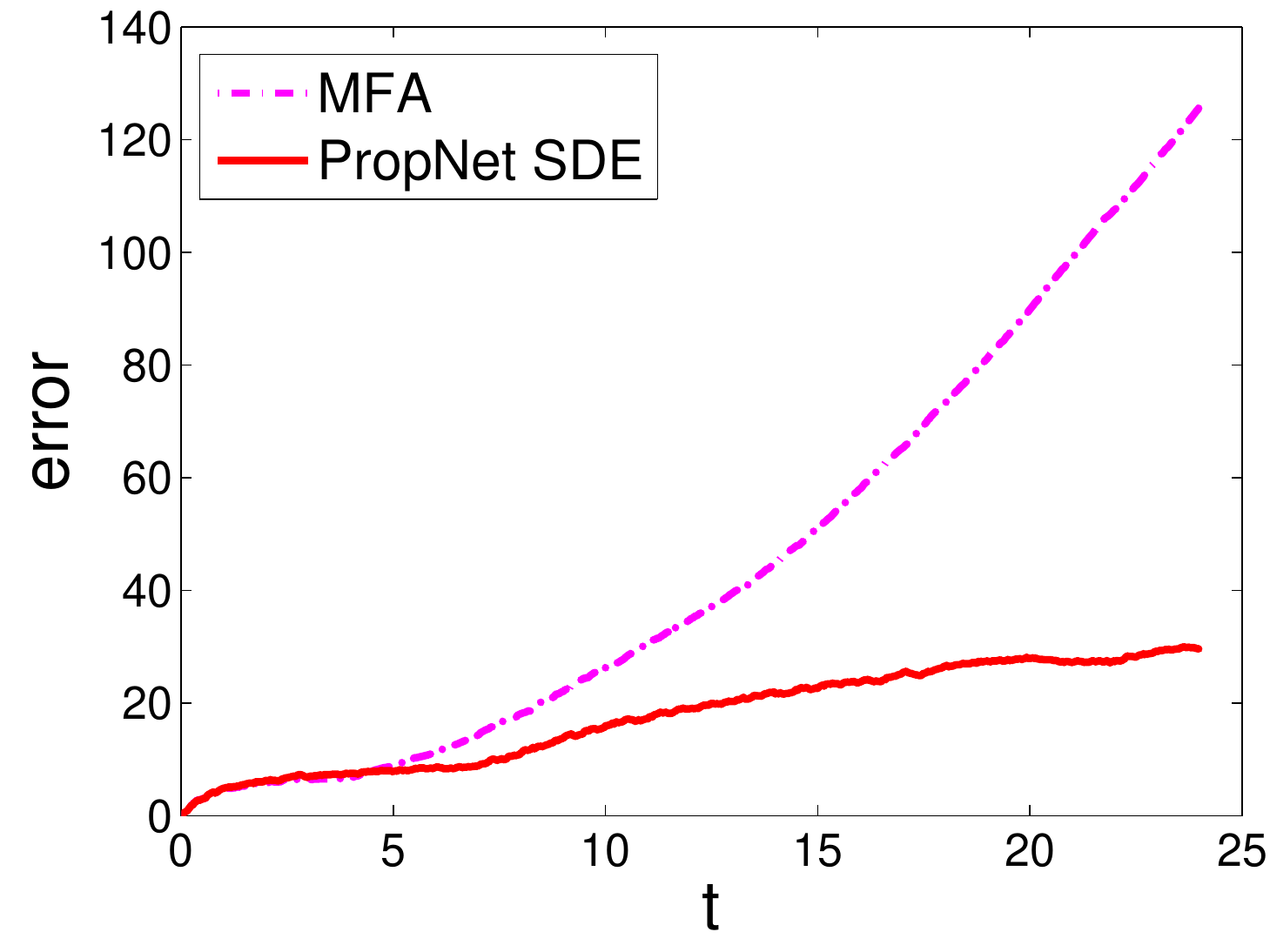}
\caption{Average error of predicted influence $|\influhat(t)-\influ(t)|$ (left) and individual activation probability $\sum_i|\xhat_i(t)-x_i(t)|$ (right) versus time (in hours) on the Weibo-Net-Tweet dataset for the first 24 hours of propagations.}
\label{fig:real}
\end{figure}

We also instantiate one prediction result of individual node activation probabilities starting from a single source node with index $2$ (actual user ID is concealed, same below) using the proposed PropNet method. 
The result is shown in the left panel of Figure \ref{fig:realnet}. 
The color of a node shows the probability that the node is activated at $T=24$ hours, where the color bar is plotted on the side for reference.
For comparison, we also extract the empirical probabilities from the cascades in testing dataset, and show the result on the right panel of Figure \ref{fig:realnet}.
For cleaner appearance, we removed all nodes that have lower than 50\% activation probability in both results.
Using the prediction result of the PropNet SDE, we can see that nodes indexed by 4, 7, 8, 18, 20, 23, 29 have relatively high probabilities to be activated (retweet the post), in addition to the active direct followers (e.g., nodes 6, 14, 15, 42) of node $2$.
On the other hand, some direct followers, such as node 257, of node $2$ do not often help to spread the post.
These claims are supported by the empirical results in the testing data, as shown on the right of Figure \ref{fig:realnet}, and they are also backed up by the small quantitative predication error we showed in the right panel of Figure \ref{fig:real}.

\begin{figure}[t!]
\centering
\includegraphics[width=0.4\textwidth]{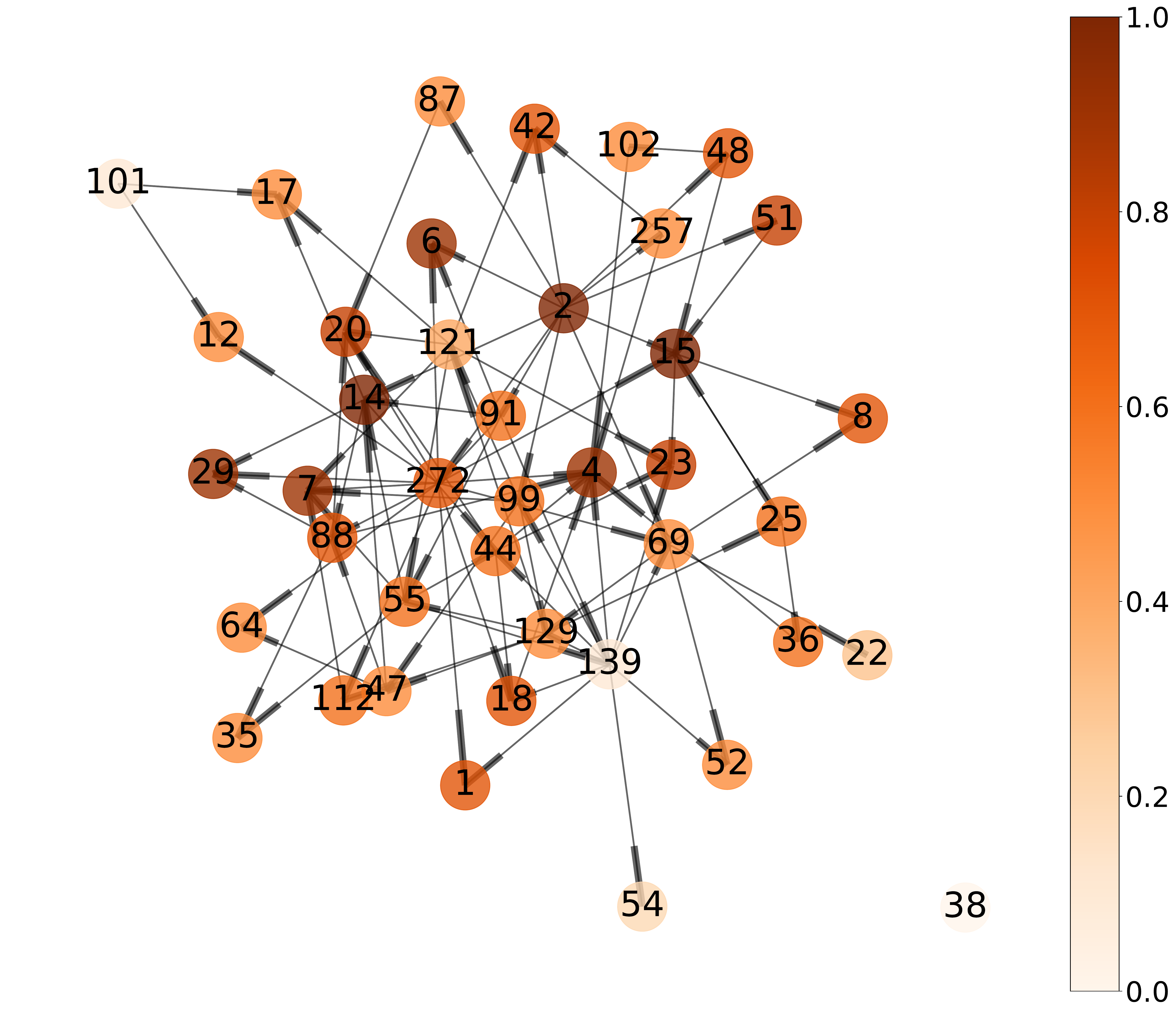}
\includegraphics[width=0.4\textwidth]{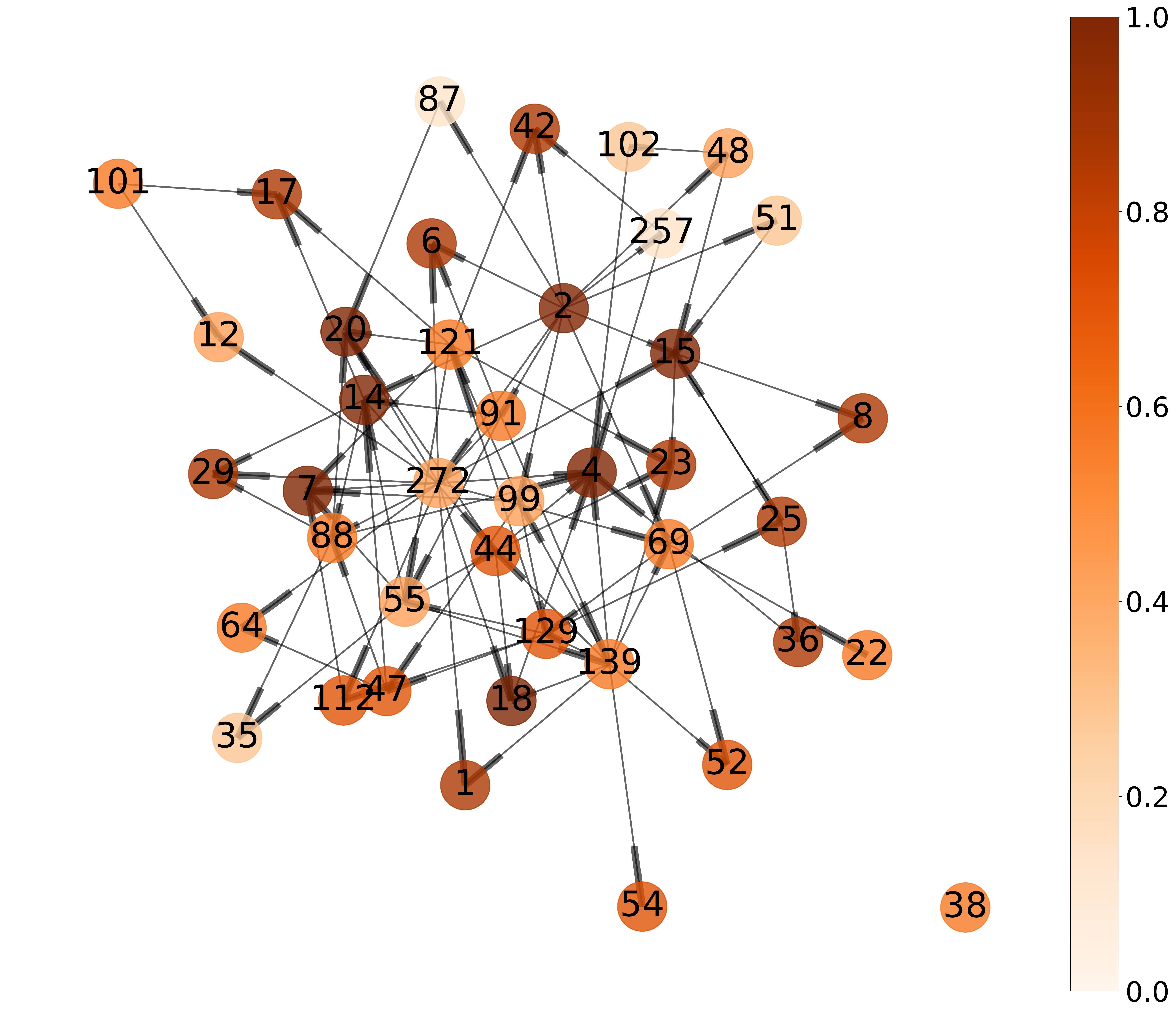}
\caption{Estimation of individual activation probabilities at $T=24$ hours using the proposed PronNet SDE method (left) and the empirical activation probabilities exhibited by the testing data (right) from source set $S=\{2\}$ on the Weibo-Net-Tweet dataset. Thicker ends of edges indicate inbounds. Nodes with $<50\%$ activation probabilities in both plots are removed for cleaner looks. }
\label{fig:realnet}
\end{figure}

\section{Conclusions}
\label{sec:conclusion}
In this paper, we proposed to reformulate propagations on heterogenous networks using jump SDE system.
We also develope an efficient numerical schemes to solve the SDE system and predict influence etc. 
Using a series of numerical experiments on variety of networks and propagation models, we show that our proposed method is accurate and efficient in influence prediction when compared to the state-of-the-art methods. 
Moreover, we showed that the proposed method can be readily modified and applied to more general propagation models which cannot be handled by any existing methods.

\bibliographystyle{abbrv}
\bibliography{/Users/xye/Dropbox/Documents/Library/library}

\end{document}